\documentclass[11pt]{article}
\usepackage{amsfonts,latexsym,amssymb,amsthm,amsmath,graphicx,cases,comment}
\usepackage{amsmath}
\usepackage{paralist}
\usepackage{graphics} 
\usepackage{epsfig} 
\usepackage{epstopdf}
\usepackage{ulem}

\usepackage[titletoc,title]{appendix}
\usepackage{empheq}
\usepackage{cancel}
\usepackage{tikz}
\usetikzlibrary{arrows,shapes}
\usetikzlibrary{decorations.pathmorphing,decorations.pathreplacing}
\usetikzlibrary{calc,patterns,angles,quotes}

\usepackage[colorlinks=true]{hyperref}
\hypersetup{urlcolor=blue, citecolor=red}

\newtheorem{theorem}{Theorem}[section]
\newtheorem{thm}{Theorem}

\newtheorem{lemma}[theorem]{Lemma}
\newtheorem{prop}{Proposition}

\theoremstyle{definition}
\newtheorem{definition}[theorem]{Definition}
\newtheorem{rmk}{Remark}

\newcommand{\be}{\begin{equation}}
\newcommand{\ee}{\end{equation}}
\newcommand{\bsubeq}{\begin{subequations}}
	\newcommand{\esubeq}{\end{subequations}}
\renewcommand{\div}{\text{div}}
\newcommand{\ds}{\displaystyle}

\newcommand{\calL}{{\mathcal{L}}}

\newcommand{\calN}{{\mathcal{N}}}

\newcommand{\calF}{{\mathcal{F}}}

\newcommand{\calD}{{\mathcal{D}}}
\newcommand{\calK}{{\mathcal{K}}}
\newcommand{\calA}{{\mathcal{A}}}

\newcommand{\calV}{{\mathcal{V}}}
\newcommand{\BA}{\mathbb{A}}
\newcommand{\BB}{\mathbb{B}}
\newcommand{\BR}{\mathbb{R}}
\newcommand{\BC}{\mathbb{C}}

\newcommand{\wti}{\widetilde}

\newcommand{\bpm}{\begin{pmatrix}}
	\newcommand{\epm}{\end{pmatrix}}

\newcommand{\bbm}{\begin{bmatrix}}
	\newcommand{\ebm}{\end{bmatrix}}

\numberwithin{equation}{section}
\numberwithin{thm}{section}
\numberwithin{rmk}{section}
\numberwithin{prop}{section}

\newcommand{\lso}{L^{q}_{\sigma}(\Omega)}
\newcommand{\lo}[1]{L^{#1}_{\sigma}(\Omega)}
\newcommand{\ls}{L^{q}_{\sigma}}

\newcommand{\phiij}{\{\phi_{ij}\}_{j=1}^{\ell_i}}
\newcommand{\phiijm}{\{\phi_{ij}\}_{i=1,j=1}^{M \ \ \ell_i}}
\newcommand{\phiijs}{\{\phi_{ij}^{*}\}_{j=1}^{\ell_i}}
\newcommand{\phiijms}{\{\phi_{ij}^{*}\}_{i=1,j=1}^{M \ \ \ell_i}}
\newcommand{\Bso}{B^{2-\rfrac{2}{p}}_{q,p}(\Omega)}

\newcommand{\Bt}{\widetilde{B}^{2-\rfrac{2}{p}}_{q,p}}
\newcommand{\Bto}{\widetilde{B}^{2-\rfrac{2}{p}}_{q,p}(\Omega)}
\newcommand{\xtpq}{X^T_{p,q}}
\newcommand{\xtpqs}{X^T_{p,q,\sigma}}
\newcommand{\xipqs}{X^{\infty}_{p,q,\sigma}}
\newcommand{\ytpq}{Y^T_{p,q}}
\newcommand{\yipq}{Y^{\infty}_{p,q}}
\newcommand{\xipq}{X^{\infty}_{p,q}}
\newcommand{\woiq}{W^{1,q}_0}
\newcommand{\wtq}{W^{2,q}}
\newcommand\rfrac[2]{{}^{#1}\!/_{#2}}
\newcommand{\Fs}{F_{\sigma}}
\newcommand{\lqaq}{\big( \lso, \calD(A_q) \big)_{1-\frac{1}{p},p}}
\newcommand{\norm}[1]{\left\lVert#1\right\rVert}
\newcommand{\abs}[1]{\left\lvert#1\right\rvert}

\newcommand{\lplqs}{L^p \big( 0,\infty; \lso \big)}

\begin{document}

\title{Uniform stabilization of Navier-Stokes equations in critical $L^q$-based Sobolev and Besov spaces by finite dimensional interior localized feedback controls \thanks{The research of I.L. and R.T. was partially supported by the National Science Foundation under Grant DMS-1713506. The research of B.P. was supported by the ERC advanced grant 668998 (OCLOC) under the EU's H2020 research program.}}	

\author{Irena Lasiecka \thanks{Department of Mathematical Sciences, University of Memphis, Memphis, TN 38152 USA and IBS, Polish Academy of Sciences, Warsaw, Poland.}
	\and Buddhika Priyasad \thanks{Institute for Mathematics and Scientific Computing, University of Graz, Heinrichstrasse 36, A-8010 Graz, Austria. (b.sembukutti-liyanage@uni-graz.at).}
	\and Roberto Triggiani \thanks{Department of Mathematical Sciences, University of Memphis, Memphis, TN 38152 USA.}}

	\date{}
\maketitle		
	\begin{abstract}
		\noindent We consider 2- or 3-dimensional incompressible Navier-Stokes equations defined on a bounded domain $\Omega$, with no-slip boundary conditions and subject to an external force, assumed to cause instability. We then seek to uniformly stabilize such N-S system, in the vicinity of an unstable equilibrium solution, in critical $L^q$-based Sobolev and Besov spaces, by finite dimensional feedback controls. These spaces are `close' to $L^3(\Omega)$ for $d=3$. This functional setting is significant. In fact, in the case of the uncontrolled N-S dynamics, extensive research efforts have recently lead to the space $L^3(\BR^3)$ as being a critical space for the issue of well-posedness in the full space. Thus, our present work manages to solve the stated uniform stabilization problem for the controlled N-S dynamics in a correspondingly related function space setting. In this paper, the feedback controls are localized on an arbitrarily small open interior subdomain $\omega$ of $\Omega$. In addition to providing a solution of the uniform stabilization problem in such critical function space setting, this paper manages also to much improve and simplify, at both the conceptual and computational level, the solution given in the more restrictive Hilbert space setting in the literature. Moreover, such treatment sets the foundation for the authors' final goal in a subsequent paper. Based critically on said low functional level where compatibility conditions are not recognized, the subsequent paper solves in the affirmative a presently open problem: whether uniform stabilization by localized tangential boundary feedback controls, which-in addition-are finite dimensional, is also possible in dim $\Omega = 3$.
	\end{abstract}
	
	
	
	\section{Introduction}\label{I-Sec-1}
	
	\subsection{Controlled Dynamic Navier-Stokes Equations}
	Let, at first, $\Omega$ be an open connected bounded domain in $\mathbb{R}^d, d = 2,3$ with sufficiently smooth boundary $\Gamma = \partial \Omega$.  More specific requirements will be given below. Let $\omega$ be an arbitrarily small open smooth subset of the interior $\Omega$, $\omega \subset \Omega$, of positive measure. Let $m$ denote the characteristic function of $\omega$: $m(\omega) \equiv 1, \ m(\Omega \backslash \omega) \equiv 0$. Consider the following controlled Navier-Stokes Equations with no-slip Dirichlet boundary conditions, where $Q =  (0, \infty) \times \Omega, \ \Sigma = (0,\infty) \times \Gamma$: 
	\begin{subequations}\label{I-1.1}
		\begin{align}
		y_t(t,x) - \nu \Delta y(t,x) +  (y \cdot \nabla)y + \nabla \pi(t,x) &= m(x)u(t,x) + f(x)  &\text{ in } Q \label{I-1.1a}\\ 
		\div \ y &= 0  &\text{ in } Q \label{I-1.1b}\\
		y &= 0 &\text{ on } \Sigma \label{I-1.1c}\\
		y(0,x) &= y_0(x) &\text{ in } \Omega \label{I-1.1d}
		\end{align}
	\end{subequations}
	
	\noindent \textbf{Notation:} As already done in the literature, for the sake of simplicity, we shall adopt the same notation for function spaces of scalar functions and function spaces of vector valued functions. Thus, for instance, for the vector valued ($d$-valued) velocity field $y$ or external force $f$, we shall simply write say $y,f \in L^q(\Omega)$ rather than $y,f \in (L^q(\Omega))^d$ or $y,f \in \mathbf{L}^q(\Omega)$. This choice is unlikely to generate confusion. By way of orientation, we state at the outset two main points. For the linearized $w$-problem (\ref{I-1.13}) below in the feedback form (\ref{I-2.16}), the corresponding well-posedness and global feedback uniform stabilization result, Theorem \ref{I-Thm-2.2}, holds in general for $1 < q < \infty$. Instead, the final, main well-posedness and feedback uniform, local stabilization result, Theorem \ref{I-Thm-2.5}, for the nonlinear problem (\ref{I-2.27}) or (\ref{I-2.28}) corresponding to the original problem (\ref{I-1.1}) will require $q > 3$, see (\ref{I-8.16}), in the $d=3$-case, hence $\displaystyle 1 < p < \frac{6}{5}$; and $q > 2$, in the $d = 2$-case, hence $1 < p < \frac{4}{3}$; see (\ref{I-1.16}).	Let $u \in L^p(0,T;L^q(\Omega))$ be the control input and $y = (y_1,\dots,y_d)$ be the corresponding state (velocity) of the system. Let $\nu>0$ be the viscosity coefficient. The function $v(t,x) = m(x)u(t,x)$ can be viewed as an interior controller with support in $Q_{\omega} = (0,\infty) \times \omega $. The initial condition $y_0$ and the body force $f \in L^q(\Omega)$ are given. The scalar function $\pi$ is the unknown pressure.
	
	\subsection{Stationary Navier-Stokes equations}
	The following result represents our basic starting point.
	
	\begin{thm}\label{I-Thm-1.1}
		Consider the following steady-state Navier-Stokes equations in $\Omega$
		\begin{subequations}\label{I-1.2}
			\begin{align}
			-\nu \Delta y_e +  (y_e.\nabla)y_e + \nabla \pi_e &= f &\text{ in }  \Omega \label{I-1.2a}\\ 
			\div \ y_e &= 0 &\text{ in }  \Omega\label{I-1.2b} \\
			y_e &= 0 &\text{ on } \Gamma \label{I-1.2c}
			\end{align}
		\end{subequations}
		Let $1 < q < \infty$. For any $f \in L^q(\Omega)$ there exits a solution (not necessarily unique) $(y_e,\pi_e) \in (W^{2,q}(\Omega) \cap W^{1,q}_{0}(\Omega)) \times (W^{1,q}(\Omega)/\mathbb{R})$.
	\end{thm}
	
	For the Hilbert case $q=2$, see \cite[Thm 7.3 p 59]{CF:1980} . For the general case $1 < q < \infty$, see \cite[Thm 5.iii p 58]{AR:2010}.
	
	\begin{rmk}
		It is well-known \cite{Lad:1969}, \cite{Li:1969}, \cite{Te:1979} that the stationary solution is unique when ``the data is small enough, or the viscosity is large enough" \cite[p 157; Chapt 2]{Te:1979} that is, if the ratio $\ds \rfrac{\norm{f}}{\nu^2}$ is smaller than some constant that depends only on $\Omega$ \cite[p 121]{FT:1984}. When non-uniqueness occurs, the stationary solutions depend on a finite number of parameters \cite[Theorem 2.1, p 121]{FT:1984} asymptotically in the time dependent case.
	\end{rmk}

	\begin{rmk}
		The case where $f(x)$ in (\ref{I-1.1a}) is replaced by $\ds \nabla g(x)$ is noted in the literature as arising in certain physical situations, where $f$ is a conservative vector field. The analysis of this relevant case is postponed to Remark \ref{I-Rmk-1.4}, at the end of Section \ref{I-Sec-1}.
	\end{rmk}
	
	\subsection{Main goal of the present paper}\label{I-Sec-1.3}
	\noindent For a given external force $f$, if the Reynolds number $\frac{1}{\nu}$ is sufficiently large, then the steady state solution $y_e$ in (\ref{I-1.2}) becomes unstable (in a quantitative sense to be made more precise in Section \ref{I-Sec-2.2} below) and turbulence occurs.\\
	
	\noindent The main goal of the present paper is then - at first qualitatively - \uline{to feedback stabilize the non-linear N-S model (\ref{I-1.1}) subject to rough (non-smooth) initial condition $y_0$, in the vicinity of an (unstable) equilibrium solution $y_e$ in (\ref{I-1.2})}. Thus this paper pertains to the general context of ``turbulence suppression or attenuation" in fluids. The general topic of turbulence suppression (or attenuation) in fluids has been the object of many studies over the years, mostly in the engineering literature – through experimental studies and via numerical simulation - and under different geometrical and dynamical settings. The references cited in the present paper by necessity refer mostly to the mathematical literature, and most specifically on the \uline{localized interior control setting of problem (\ref{I-1.1})}. A more precise description thereof is as follows: \textit{establish interior localized exponential stabilization of problem (\ref{I-1.1}) near an unstable equilibrium solution by means of a finite dimensional localized, spectral-based feedback control, in the important case of initial conditions $y_0$ of low regularity}, as technically expressed by $y_0$ being in suitable $L^q$/Besov space with tight indices. In particular, local exponential stability for the velocity field $y$ near an equilibrium solution $y_e$ will be achieved in the topology of the Besov space
	
	\begin{equation}\label{I-1.3}
	\big( L^q(\Omega), W^{2,q}(\Omega) \big)_{1-\rfrac{1}{p},p} = \Bso, \quad 1 < p < \frac{2q}{2q -1}; \ q > d,\ d = 2,3.
	\end{equation}
	\noindent See more precisely (\ref{I-1.16a}). We note the tight indices: e.g. $\ds 1 < p < \rfrac{6}{5} $ for $q > d = 3$. In such setting, the compatibility conditions on the boundary of the initial conditions are not recognized. This feature is precisely our key objective within the stabilization problem. The fundamental reason is that such feature will play a critical role in the successor paper \cite{LPT} in showing that: \uline{local tangential boundary feedback stabilization near an unstable equilibrium solution with finite dimensional controls is possible also in dimension $d=3$}, thus solving in the affirmative a recognized open problem in the stabilization area. This point will be more appropriately expanded in Section \ref{I-Sec-1.6} below. For $d = 3$, the space is (\ref{I-1.3}) in `close' to $L^3(\Omega)$.\\
	
	\noindent \textbf{Criticality of the space $L^3$:} We now expand on one message of the abstract regarding the `criticality' of the space $L^3(\Omega)$. In the case of the uncontrolled N-S equations defined on the full space $\BR^3$, extensive research efforts have recently lead to the discovery that the space $L^3(\BR^3)$ is a `critical' space for the issue of well-posedness. Assuming that some divergence free initial data in $L^3(\BR^3)$ would produce finite time singularity, then there exists a so-called minimal blow-up initial data in $L^3(\BR^3)$ \cite{GKP:2010}, \cite{JS:2013}. More precisely, let $y$ now be a solution of the N-S equations (\hyperref[I-1.1]{1.1a-b-d}) with $m \equiv 0, f \equiv 0$, as defined on the whole space $\BR^3$. For any divergence free I.C. $y_0 \in L^3(\BR^3)$, denote by $T_{max} (y_0)$ the maximal time of existence of the mild solution starting from $y_0$. Define	
	\begin{equation}
		\rho_{max} = \sup \big\{ \rho: \ T_{max}(y_0) = \infty \text{ for every divergence free } y_0 \in L^3(\BR^3), \text{ with } \norm{y_0}_{L^3(\BR^3)} < \rho \big\}.  \nonumber
	\end{equation}
	
	\noindent The following result holds \cite[Theorem 4.1, p 14]{JS:2013}: Suppose $\rho_{max} < \infty$. Then there exists some $\ds y_0 \in L^3(\BR^3)$, $\norm{y_0}_{L^3(\BR^3)} = \rho_{max}$, whose $T_{max}(y_0) < \infty$, i.e. the corresponding solution blows up in finite time. Of the numerous works that followed the pioneering work of \cite{Ler:1934} along this line of research, we quote in addition \cite{Ser:1962}, \cite{Ser:1963}, \cite{ESS:1991}, \cite{RS:2009}.\\
	
	\noindent Thus, our present work manages to solve the uniform stabilization problem for the controlled N-S equations as (\hyperref[I-1.1]{1.1a-b-d}) in a correspondingly related low-regularity function space setting. A further justification of our low-regularity level of the Besov space in (\ref{I-1.3}) is provided by  the final goal of our line of research in the subsequent Part II \cite{LPT}. Based critically on said low-regularity level of the Besov space (\ref{I-1.3}), which does not recognize compatibility conditions on the boundary of the initial conditions - such Part II solves in the affirmative a presently open problem by showing that uniform stabilization by localized tangential boundary feedback controls which moreover are finite dimensional is possible also in dim $\Omega = d = 3$.\\
	
	\noindent The successor \cite{LPT} to the present work will extensively review the literature as it pertains to the boundary stabilization case, particularly with tangential control action. Accordingly, below we shall review the present study mostly in comparison with the prior solution of the localized interior stabilization in the Hilbert-based treatment of the original work \cite{BT:2004} (via a Riccati-based finite dimensional control), followed by \cite{BLT3:2006} (via a spectral based finite dimensional control) and textbook versions thereof \cite[Chapter 4]{B:2011}.	
	
	\subsection{Qualitative Orientation}
	
	\subsubsection{On the local, interior, feedback stabilization problem: Past Literature}\label{I-Sec-1.4.1}
	 We start with an unstable steady state solution $y_e$, given an external force $f$, and a sufficiently large Reynold number $\frac{1}{\nu}$, as described in Section \ref{I-Sec-1.3}. We deliberately aim at rough (non-smooth) initial conditions $y_0$. We then seek a finite-dimensional interior localized feedback control $u$,  such that the corresponding N-S problem is well-posed in a suitable function space setting (depending on the I.C. $y_0$) and its solution $y$ in (\ref{I-1.1}) is locally exponentially stable near the equilibrium solution $y_e$, in a suitably corresponding norm. This problem was originally posed and solved in the Hilbert space setting in \cite[Theorem 2.2 p 1449]{BT:2004} by means of a finite dimensional Riccati-based feedback control $u$, where exponential decay is obtained in the $\calD(A^{\rfrac{1}{4}})$-topology. Here $A$ is the positive self-adjoint Stokes operator in (\ref{I-1.17}) with $q = 2$ in the space $H$ with $L^2(\Omega)$-topology. See (\ref{I-1.6}) below. A similar exponential decay result, in the same $\calD(A^{\rfrac{1}{4}})$-topology, is given in \cite[Thm 5.1, p 42]{BLT3:2006}, this time by means of a finite-dimensional, spectral-based feedback control $u$. These results are reproduced in textbook form in \cite[Chapter 4]{B:2011}.\\ 
	 
	 \noindent Regarding the solution given in these Hilbert-space based references, we point out (at present) two defining, linked characteristics of their finite dimensional treatment:
	 \begin{enumerate}[(i)]
	 \item The number of stabilizing (localized) controls for the (complex-valued) nonlinear dynamics (\ref{I-1.1}) is $N= \sup \{ N_i;\ i = 1,\dots, M\}$, that is, the max of the \textit{algebraic} multiplicity $N_i$ of the $M$ distinct unstable eigenvalues $\lambda_i$, see (\ref{I-2.2}), of the projected Oseen operator $\calA_N^u$ in (\ref{I-2.5}).
	 
	 \item in the fully general case, the algebraic (Kalman rank) conditions for controllability under which the finite dimensional feedback control is \textit{explicitly} constructed involve the Grahm-Schmidt orthogonalization  of the generalized eigenfunctions of the adjoint $(\calA_N^u)^*$, making the test difficult to verify. Only in the case where the restriction $\calA_N^u$ in (\ref{I-2.5}) of the Oseen operator $\calA$ in (\ref{I-1.10}) is semisimple (algebraic and geometric multiplicity of the unstable eigenvalues coincide), are the controllability tests given in terms of eigenfunctions of $(\calA_N^u)^*$.
	 
	 \end{enumerate} 
	
	\subsubsection{Additional goals of the present paper as definite improvements over the literature}
	
	We list these main additional goals of the present work aimed at markedly improving both the results and the approach of the original reference \cite{BT:2004}, followed by \cite{BLT3:2006} and their textbook version \cite[Chapter 4]{B:2011}. They are:
	
	\begin{enumerate}[(i)]
		\item With reference to part (i) in Section \ref{I-Sec-1.4.1}, our next goal is to obtain (in the general case of Theorem \ref{I-Thm-4.1}) that the number of finite dimensional stabilizing controls needed for the (complex valued version of the) dynamics (\ref{I-1.1}) is $K = \sup \{ \ell_i; \ i = 1, \dots, M \}$, where $\ell_i$ is the \textit{geometric} multiplicity of the $M$ distinct unstable eigenvalues $\lambda_i$ in (\ref{I-2.2}). This is a notable reduction in the number of needed controls over the max \underline{algebraic} multiplicity $N$ in (i) of Section \ref{I-Sec-1.4.1}.
		
		\item Intimately linked to goal (i) is the next goal to obtain the controllability Kalman rank condition of Section \ref{I-Sec-3} be expressed in terms of only the eigenfunctions - not the generalized eigenfunctions: as in the past literature \cite{BLT1:2006}, \cite{BLT2:2007}, \cite{BLT3:2006}, \cite[Chapter 4]{B:2011}  - of the adjoint $(\calA_N^u)^*$.
		
	
	
	\item An important additional goal is to simplify and make more transparent the well-posedness and local stabilization arguments for the non-linear problem, in particular through a direct analysis of the nonlinear operator $\calN_q$ in (\ref{I-1.11}), called $B$ in \cite{BT:2004}, not its approximation sequence $B_{\varepsilon}$ as in \cite[Section 4, p1480]{BT:2004} or \cite[proof of Therorem 3.4 p 107-110]{B:2011}. More precisely, unlike these references, the present paper carries out an analysis of the critical issues based on the \uline{maximal regularity property} of the linearized feedback operator $\ds \BA_{_F} (= \BA_{_{F,q}})$ in (\ref{I-6.1}) in the critical $L^q$/Besov setting. This point also is further expanded in Section \ref{I-Sec-1.6} below. 
	
	\item A final goal - in line with goal (iii) above - is to obtain corresponding results for the pressure $\pi$ in (\ref{I-1.1a}), as part of the same \uline{maximal regularity property} of the linearized feedback operator $\ds \BA_{_{F}} (= \BA_{_{F,q}})$ in (\ref{I-6.1}), unlike \cite[Theorem 2.3, p1450]{BT:2004}.
	\end{enumerate}	
	
	\subsection{What is one key motivation for seeking interior localized feedback exponential stabilization of problem (\ref{I-1.1}) in the topology of the Besov space in (\ref{I-1.3})?}
	
	As already discussed in Section \ref{I-Sec-1.3}, obtaining the resulting stabilization in a non-Hilbertian setting is of theoretical interest in itself, and is in line with recent developments in 3-d N-S equations defined on the entire space $\BR^3$, where recent breakthroughs have identified the space $L^3(\BR^3)$ as a critical space for well-posedness, possessing the minimal blow up initial data property \cite{GKP:2010}, \cite{JS:2013}. However, our main original motivation for the present study is another. The present paper intends to test $L^q$/Besov spaces techniques initially in the interior localized feedback stabilization problem (\ref{I-1.1}) in the critical low regularity setting of (\ref{I-1.3}). The true aim is, however, to export them with serious additional technical difficulties, to solve the presently recognized \textbf{open} problem of the local feedback exponential stabilization of the N-S equations with \textit{finite-dimensional} feedback \textit{tangential boundary} controllers in the case of dimension $d = 3$ \cite{LPT}. In fact, present state-of-the-art has succeeded \cite{LT1:2015}, \cite{LT2:2015} in establishing local exponential stabilization (asymptotic turbulence suppression) by means of \textit{finite-dimensional tangential} feedback boundary control in the Hilbert setting and with no assumptions whatsoever on the Oseen operator in two cases: 
	\begin{enumerate}[(i)]
		\item when the dimension $d = 2$,
		\item when the dimension $d = 3$ but the initial condition $y_0$ in (\ref{I-1.1}.d) is compactly supported. 
	\end{enumerate}
	In the general $d = 3$ case, the non-linearity of the N-S problem forces a Hilbert space setting with a high-topology $H^{\rfrac{1}{2} + \epsilon}(\Omega)$ for the initial conditions, whereby the compatibility conditions on the boundary kick in. These then cannot allow the stabilizing feedback control to be finite-dimensional in general. More precisely, even at the level of the \textit{linearized boundary} problem for $d = 3$, \textit{open loop} exponential stabilization \cite [Proposition 3.7.1 Remark 3.7.1]{BLT3:2006}, \cite[ Proposition 2.5, eq (2.48)]{LT1:2015} provide a \textit{boundary} control consisting of a finite-dimensional term plus the term $e^{-2 \gamma_1 t} ( (I.C.)|_{\Gamma})$, with $\gamma_1 > 0$ preassigned, which spoils the finite-dimensionality, unless the initial condition is compactly supported. These limitations are in subsequent literature. In contrast, the Besov space in (\ref{I-1.3}) above, which is ``close" to the space $L^3(\Omega)$ for $ d = 3$, has the key, fundamental advantage of not recognizing the boundary conditions. That is why this paper is interested in a stabilization result in such a low regularity space for $d=3$, at first in the case of interior localized control.
	
	\subsection{Comparison with prior work, once the present treatment is specialized to the Hilbert setting $(q=2)$}\label{I-Sec-1.6}
	
	\noindent A comparison with the original prior work \cite{BT:2004}, \cite{BLT3:2006}, reproduced in \cite[Chapter 4]{B:2011} which was carried out in the Hilbert setting, is in order.\\
	
	\noindent \textbf{Orientation} Even when specialized to the Hilbert space setting $(q=2)$, the present treatment offers distinct, notable advantages – both conceptual and computational over prior literature quoted above. These include not only definitely simpler and more direct arguments but also transparent simplifications in the actual construction of the finite dimensional stabilizing controllers as well as their number. Qualitative details are given below. The main conceptual approach and the final results of the present paper are (when specialized to the Hilbert setting) qualitatively in line with those in \cite{BT:2004}: local uniform stabilization of the non-linear $y$-problem (\ref{I-1.1}) near an unstable equilibrium solution $y_e$ by means of finite dimensional, arbitrarily localized controllers is based on the corresponding result on (global) uniform stabilization of the linearized $w$-system (\ref{I-1.13}). This in turn rests on the space decomposition technique introduced in \cite{RT:1975} for parabolic problems (and also for differentiable semigroups): its foundational starting point is the controllability of the finite dimensional unstable projected system $w_N$ in (\ref{I-2.8a}). However, in the implementation of these two fundamental phases, linear analysis - in particular, its finite dimensional $w_N$ component - and nonlinear analysis, the present paper provides a much more attractive, more powerful and mature treatment. We mention the most relevant new features. They are:
	
	\begin{enumerate}[1.]
		\item finite dimensional analysis leading, through a much more simplified and more direct approach, to a lower (optimal) number of feedback controls;
		\item infinite dimensional analysis on the nonlinear effects (the operator $\calN_q$ in (\ref{I-1.11})) handled by critical and clean use of \uline{maximal regularity of the linearized feedback operator} $\ds \BA_{_{F}} (\equiv \BA_{_{F,q}})$ in (\ref{I-6.1}), rather than by the approximation argument as in \cite{BT:2004}. This refers to both $y$ and $\pi$.
	\end{enumerate}
	\noindent These two points are explained below.	
	
	\begin{enumerate}[1.]
		\item \textbf{Stabilization of the linearized $w$-problem (\ref{I-1.13})}. The key foundational algebraic test for controllability  of the finite dimensional $w_N$-system (\ref{I-2.8a}) on the finite dimensional unstable subspace  $W^u_N$ is much simplified, sharper and leads, in principle, to checkable conditions and to an implementable procedure to obtain constructively the finite dimensional stabilizing vectors $u_k \in W^u_N,\ p_k \in (W^u_N)^*$. In fact, the present treatment shows that (for the complexified version of the N-S) the required number of feedback stabilizing finite dimensional controllers is $K = \sup \{\ell_i, \ i = 1, \dots, M \}$, the max of the geometric multiplicity $\ell_i$ of the $M$ distinct unstable eigenvalues $\lambda_i$ of the Oseen operator $\ds \calA$; not the larger $\sup \{ N_i, \ i = 1, \dots, M \}$, the max of the algebraic multiplicity $N_i$ of its distinct unstable eigenvalues, as in \cite{BT:2004}, \cite{BLT1:2006}, \cite{BLT3:2006}, \cite[Chapter 4]{B:2011}. Let alone $\ds N = \sum_{i = 1}^{M} N_i = $ dim $W^u_N$ (dimension of the generalized eigenspace of the unstable eigenvalues) as in the treatment of \cite[assumption K.2 p 123]{B:2011}, where, in addition, the simplifying assumption that algebraic and geometric multiplicities coincide for the unstable eigenvalues. Moreover, the entire analysis of the present paper rests only on the (true) eigenvectors corresponding to the unstable eigenvalues of the adjoint operator in (\ref{I-3.1}); not only under the Finite Dimensional Spectral Condition (semisimplicity) as in \cite{BT:2004} where $W^u_N$ has a basis of such (true) eigenvectors, but also in the most general case where the projected Oseen operator $\ds \calA_N^u $ is in Jordan form, and hence the basis on $W^u_N$ consists instead of all generalized eigenvectors corresponding to the unstable eigenvalues. As first noted in \cite{LT1:2015} in the study of \textit{tangential boundary stabilization} of the N-S equations, even in the general case possessing only a basis of generalized eigenfunctions arising from the Jordan form, the \underline{final test for controllability involves only the true eigenfunctions of the adjoint operator}: the algebraic test (\ref{I-4.13}), (\ref{I-4.14}) for controllability in the general case is exactly the same as the algebraic test (\ref{I-3.18}) in the semisimple (diagonalizable)-case; and only the true eigenfunctions count. This justifies why the number $K$ of (complex valued) stabilizing controllers as in Theorem \ref{I-Thm-2.1} is equal to the supremum of the geometric multiplicity of the unstable eigenvalues, not the supremum of their larger algebraic multiplicity as in past references \cite{BT:2004}, \cite{BLT1:2006}, \cite{BLT3:2006}, \cite[Chapter 4]{B:2011} as noted above. Moreover, in \cite{BT:2004} repeated in \cite[Chapter 4]{B:2011} the procedure for testing controllability in the general case was much more cumbersome and far less implementable: the original basis of generalized eigenfunctions of the adjoint operator in the general case was transformed into an orthonormal basis of $W^u_N$ via the Schmidt orthogonalization process, and the test for the finite dimensional controllability was then based on such transformed, and thus in principle difficult to check, orthogonalized system: a much more complicated test than the one using just the true eigenfunctions as in (\ref{I-4.13}).\\
		
		
		\item \textbf{Local Stabilization of the nonlinear translated $z$-equation (\ref{I-1.7}) near the origin, hence of the original $y$-equation (\ref{I-1.1}) near an equilibrium solution $y_e$}. Treatment of the nonlinearity in the present work is much more transparent and direct than the one performed in \cite{BT:2004}. Here the analysis is directly in terms of the nonlinear operator $\calN_q$ in (\ref{I-1.11}) and \uline{makes use of maximal regularity properties} of the linearized feedback operator $\ds \BA_{_F} (\equiv \BA_{F,q})$ in (\ref{I-6.1}) (maximal regularity is equivalent to analyticity of the semigroup in the Hilbert setting. Instead, in the Banach setting, maximal regularity implies, but is not necessarily implied by, analyticity of the semigroup). In contrast, in \cite{BT:2004} with $q = 2$, an approximation argument of the nonlinear operator $\calN_q$, denoted by $B$, was used, by introducing a sequence of approximating operators $B_{\varepsilon}$ thereof \cite[Section 4, p1480]{BT:2004}. A critical step in \cite{BT:2004} is that the nonlinearity $B$ (or its approximation) be controlled by the topology of the $A^{\rfrac{3}{4}}$-power; and this in turn is achieved by using an optimal control approach with $A^{\rfrac{3}{4}}$-penalization of the solution via Riccati equations. There is no need of this in the present treatment (the analysis of the optimization problem and Riccati equation in a non-Hilbert setting is not the right tool). We likewise note that our present treatment of the passage from the $w$-linearized problem (\ref{I-2.16}) to the fully non-linear $z$-system (\ref{I-2.20}) is also different from the one employed in \cite{BLT1:2006}, \cite{LT2:2015} which was also direct in terms of the nonlinear operator $\calN$. It was however not maximal regularity - based, as in the present paper.\\
		
		\item \textbf{Well-posedness of the pressure $\pi$ for the original $y$-problem in the feedback form as in (\ref{I-2.22}) in the vicinity of the equilibrium pressure $\pi_e$ in (\ref{I-1.2a}).} The well-posedness result Theorem \ref{I-Thm-10.2} on the pressure $\pi$ of the original $y$-problem on feedback form as given by (\ref{I-2.26}), (\ref{I-2.27}) is the $L^q$/Besov space counterpart of the Hilbert ($L^2$)-version given by \cite[Theorem 2.3 p1450]{BT:2004}. The present proof is much more direct as, again, is based on maximal regularity properties. In contrast, the proof in \cite[p 1484]{BT:2004} is based on the approximation of the original problem.
	\end{enumerate}
	
	\subsection{Helmholtz decomposition}
	A first difficulty one faces in extending the local exponential stabilization result for the interior localized problem (\ref{I-1.1}) from the Hilbert-space setting in \cite{BT:2004}, \cite{BLT1:2006} to the $L^q$ setting is the question of the existence of a Helmholtz (Leray) projection for the domain $\Omega$ in $\mathbb{R}^d$. More precisely: Given an open set $\Omega \subset \mathbb{R}^d$, the Helmholtz decomposition answers the question as to whether $L^q(\Omega)$ can be decomposed into a direct sum of the solenoidal vector space $\lso$ and the space $G^q(\Omega)$ of gradient fields. Here,	
	\begin{equation}\label{I-1.4}
	\begin{aligned}
	\lso &= \overline{\{y \in C_c^{\infty}(\Omega): \div \ y = 0 \text{ in } \Omega \}}^{\norm{\cdot}_q}\\
	&= \{g \in L^q(\Omega): \div \ g = 0; \  g\cdot \nu = 0 \text{ on } \partial \Omega \},\\
	& \hspace{3cm} \text{ for any locally Lipschitz domain } \Omega \subset \mathbb{R}^d, d \geq 2 \\
	G^q(\Omega) &= \{y \in L^q(\Omega):y = \nabla p, \ p \in W_{loc}^{1,q}(\Omega) \} \ \text{where } 1 \leq q < \infty.
	\end{aligned}
	\end{equation}
	
	\noindent Both of these are closed subspaces of $L^q$.
	
	\begin{definition}\label{I-Def-1.1}
		Let $1 < q < \infty$ and $\Omega \subset \mathbb{R}^n$ be an open set. We say that the Helmholtz decomposition for $L^q(\Omega)$ exists whenever $L^q(\Omega)$ can be decomposed into the direct sum (non-orthogonal)
		\begin{equation}
		L^q(\Omega) = \lso \oplus G^q(\Omega).\label{I-1.5}
		\end{equation}
		The unique linear, bounded and idempotent (i.e. $P_q^2 = P_q$) projection operator $P_q:L^q(\Omega) \longrightarrow \lso$ having $\lso$ as its range and $G^q(\Omega)$ as its null space is called the Helmholtz projection. Additional information is given in Appendix \hyperref[I-app-A]{A}.
	\end{definition}
	
	\noindent This is an important property in order to handle the incompressibility condition $div \ y \equiv 0$. For instance, if such decomposition exists, the Stokes equation (say the linear version of (\ref{I-1.1}) with control $u \equiv 0$) can be formulated as an equation in the $L^q$ setting. Here below we collect a subset of known results about Helmholtz decomposition. We refer to \cite[Section 2.2]{HS:2016}, in particular to the comprehensive Theorem 2.2.5 in this reference, which collects domains for which the Helmholtz decomposition is known to exist. These include the following cases:
	
	\begin{enumerate}[(i)]
		\item any open set $\Omega \subset \mathbb{R}^d$ for $q = 2$, i.e. with respect to the space $L^2(\Omega)$; more precisely, for $q = 2$, we obtain the well-known orthogonal decomposition (in the standard notation, where $\nu=$unit outward normal vector on $\Gamma$) \cite[Prop 1.9, p 8]{CF:1980}
		\begin{subequations}\label{I-1.6}
			\begin{align}
			L^2(\Omega) &= H \oplus H^{\perp}\\
			H &= \{ \phi \in L^2(\Omega): div \ \phi \equiv 0 \text{ in } \Omega; \ \phi \cdot \nu \equiv 0 \text{ on } \Gamma \}\\
			H^{\perp} &= \{ \psi \in L^2(\Omega): \psi = \nabla h, \ h \in H^1(\Omega) \};
			\end{align}
		\end{subequations}
		\item a bounded $C^1$-domain in $\mathbb{R}^d$ \cite{FMM:1998}, $1 < q < \infty $ \cite[Theorem 1.1 p 107, Theorem 1.2 p 114]{Ga:2011} for $C^2$-boundary;
		\item a bounded Lipschitz domain $\Omega \subset \mathbb{R}^d \ (d = 3)$ and for $\frac{3}{2} - \epsilon < q < 3 + \epsilon$ sharp range \cite{FMM:1998};
		\item a bounded convex domain $\Omega \subset \mathbb{R}^d, d \geq 2, 1 < q < \infty$ \cite{FMM:1998}.
	\end{enumerate}
	
	\noindent On the other hand, on the negative side, it is known that there exist domains $\Omega \subset \mathbb{R}^d$ such that the Helmholtz decomposition does not hold for some $q \neq 2$ \cite{MS:1986}.\\
	
	\noindent \textbf{Assumption (H-D)} Henceforth in this paper, we assume that the bounded domain $\Omega \subset \mathbb{R}^d$ under consideration admits a Helmholtz decomposition for the values of $q, \ 1 < q < \infty$, here considered at first, for the linearized problem (\ref{I-1.13}) below. The final result Theorem \ref{I-Thm-2.5} for the non-linear problem (\ref{I-1.1}) will require $q > d$, see (\ref{I-8.16}), in the case of interest $d = 2, 3$. 
	
	\subsection{Translated nonlinear Navier-Stokes $z$-problem: reduction to zero equilibrium}\label{I-Sec-1.8}
	We return to Theorem \ref{I-Thm-1.1} which provides an equilibrium pair $\{y_e, \pi_e\}$. Then, as in \cite{BT:2004}, \cite{BLT1:2006}, \cite{LT1:2015} we translate by $\{y_e, p_e\}$ the original N-S problem (\ref{I-1.1}). Thus we introduce new variables
	\begin{subequations}\label{I-1.7}
		\begin{align}\label{I-1.7a}
		z = y - y_e, \quad \chi = \pi - \pi_e
		\end{align}
		and obtain the translated problem
		\begin{align}
		z_t - \nu \Delta z + (y_e \cdot \nabla)z +(z \cdot \nabla)y_e + (z \cdot \nabla) z + \nabla \chi &= mu   &\text{ in } Q \\ 
		\div \ z &= 0   &\text{ in } Q \\
		z &= 0 &\text{ on } \Sigma\\
		z(0,x) &= y_0(x) - y_e(x) &\text{ on } \Omega \label{I-1.7e}
		\end{align}
	\end{subequations}
	
	\noindent We shall accordingly study the local null feedback stabilization of the $z$-problem (\ref{I-1.7}), that is, feedback stabilization in a neighborhood of the origin. As usual, we next apply the projection $P_q$ below (\ref{I-1.5}) to the translated N-S problem (\ref{I-1.7}) to eliminate the pressure $\chi$. We thus proceed to obtain the corresponding abstract setting for the problem (\ref{I-1.7}) as in \cite{BT:2004} except in the $L^q$-setting rather than in the $L^2$-setting as in this reference. Note that $P_q z_t = z_t$, since $z \in \lso$ in (\ref{I-1.4}).
	
	\subsection{Abstract nonlinear translated model}
	
	First, for $1 < q < \infty$ fixed, the Stokes operator $A_q$ in $\lso$ with Dirichlet boundary conditions  is defined by \cite[p 1404]{GGH:2012}, \cite[p 1]{HS:2016}
	\begin{equation}\label{I-1.8}
	A_q z = -P_q \Delta z, \quad
	\mathcal{D}(A_q) = W^{2,q}(\Omega) \cap W^{1,q}_0(\Omega) \cap \lso.
	\end{equation}
	The operator $A_q$ has a compact inverse $A_q^{-1}$ on $\lso$, hence $A_q$ has a compact resolvent on $\lso$. 
	
	\noindent Next, we introduce the first order operator $A_{o,q}$,
	\begin{equation}\label{I-1.9}
	A_{o,q} z = P_q[(y_e \ . \ \nabla )z + (z \ . \ \nabla )y_e], \quad \mathcal{D}(A_{o,q}) = \mathcal{D}(A_q^{\rfrac{1}{2}}) \subset \lso,
	\end{equation}
	where the $\ds \calD(A_q^{\rfrac{1}{2}})$ is defined explicitly in (\ref{I-1.22}) below. Thus, $A_{o,q}A_q^{-\rfrac{1}{2}}$ is a bounded operator on $\lso$, and thus $A_{o,q}$ is bounded on $\calD(A_q^{\rfrac{1}{2}})$ 
	\begin{equation*}
	\norm{A_{o,q}f} = \norm{A_{o,q} A_q^{-\rfrac{1}{2}} A_q^{-\rfrac{1}{2}} A_q f} \leq C_q \norm{A_q^{\rfrac{1}{2}} f}, \quad f \in \calD(A_q^{\rfrac{1}{2}}).
	\end{equation*}
	This leads to the definition of the Oseen operator
	\begin{equation}\label{I-1.10}
	\calA_q  = - (\nu A_q + A_{o,q}), \quad \calD(\calA_q) = \calD(A_q) \subset \lso.
	\end{equation}
	Finally, we define the projection of the nonlinear portion of the static operator in (\ref{I-1.7}b)
	\begin{equation}\label{I-1.11}
	\calN_q(z) = P_q [(z \cdot \nabla) z], \quad \calD(\calN_q) = W^{1,q}(\Omega) \cap L^{\infty} (\Omega) \cap \lso.	
	\end{equation}
	\noindent [As shown in (\ref{I-8.16}) in the analysis of the non-linear problem, at the end we shall use $\ds W^{1,q}(\Omega) \subset L^{\infty}(\Omega)$ for $q > $ dim $\Omega = 3$ \cite[Theorem 2.4.4, p74, requiring $C^1$ boundary.]{SK:1989}]\\
	
	\noindent Thus, the Navier-Stokes translated problem (\ref{I-1.7}), after application of the Helmholtz projector $P_q$ in Definition \ref{I-Def-1.1} and use of (\ref{I-1.8})-(\ref{I-1.11}), can be rewritten as the following abstract equation in $\lso$:
	
	\begin{subequations}\label{I-1.12}
		\begin{align}
		\frac{dz}{dt} + \nu A_q z + A_{o,q} z + P_q [(z \cdot \nabla )z] &= P_q(mu) \quad \text{ or } \quad \frac{dz}{dt} - \calA_q z + \calN_q z = P_q(mu) \text{ in } \lso \label{I-1.12a}\\
		\begin{picture}(0,0)
		\put(-135,12){ $\left\{\rule{0pt}{25pt}\right.$}\end{picture}
		z(x,0) &= z_0(x) = y_0(x) - y_e \text{ in } \lso. \label{I-1.12b}
		\end{align}
	\end{subequations}
	
	\subsection{The linearized problem of the translated model}
	Next, still for $1 < q < \infty$, we consider the following linearized system of the translated model (\ref{I-1.7}) or (\ref{I-1.12}):
	\begin{subequations}\label{I-1.13}
		\begin{align}
		\frac{dw}{dt} + \nu A_q w + A_{o,q} w &= P_q(mu) \quad \text{ or } \quad \frac{dw}{dt} - \calA_q w = P_q(mu) \text{ in } \lso \label{I-1.13a}\\
		\begin{picture}(0,0)
		\put(-80,12){ $\left\{\rule{0pt}{25pt}\right.$}\end{picture}
		w_0(x) &= y_0(x) - y_e \text{ in } \lso. \label{I-1.13b}
		\end{align}
	\end{subequations}
	
	\subsection{Some auxiliary results for problem (\ref{I-1.13}): analytic semigroup generation, maximal regularity, domains of fractional powers}
	
	In this subsection we collect some known results to be used in the sequel.
	
	\begin{enumerate}[(a)]
		\item \textbf{Definition of Besov spaces $B^s_{q,p}$ on domains of class $C^1$ as real interpolation of Sobolev spaces:} Let $m$ be a positive integer, $m \in \mathbb{N}, 0 < s < m, 1 \leq q < \infty,1 \leq p \leq \infty,$ then we define \cite[p 1398]{GGH:2012}
		\begin{subequations}\label{I-1.14}
			\begin{equation}
			B^{s}_{q,p}(\Omega) = (L^q(\Omega),W^{m,q}(\Omega))_{\frac{s}{m},p} \label{I-1.14a}
			\end{equation}
			This definition does not depend on $\ds m \in \mathbb{N}$ \cite[p xx]{W:1985}. This clearly gives
			\begin{equation}
			W^{m,q}(\Omega) \subset B_{q,p}^s(\Omega) \subset L^q(\Omega) \quad \text{ and } \quad \norm{y}_{L^q(\Omega)} \leq C \norm{y}_{B_{q,p}^s(\Omega)}. \label{I-1.14b}
			\end{equation}
		\end{subequations}
		
		We shall be particularly interested in the following special real interpolation space of the $L^q$ and $W^{2,q}$ spaces $\Big( m = 2, s = 2 - \frac{2}{p} \Big)$:
		\begin{equation}\label{I-1.15}
		B^{2-\frac{2}{p}}_{q,p}(\Omega) = \big(L^q(\Omega),W^{2,q}(\Omega) \big)_{1-\frac{1}{p},p}.
		\end{equation}
		Our interest in (\ref{I-1.15}) is due to the following characterization \cite[Thm 3.4]{HA:2000}, \cite[p 1399]{GGH:2012}: if $A_q$ denotes the Stokes operator introduced in (\ref{I-1.8}), then 
		\begin{subequations}\label{I-1.16}
			\begin{align}
			\Big( \lso,\mathcal{D}(A_q) \Big)_{1-\frac{1}{p},p} &= \Big\{ g \in \Bso : \text{ div } g = 0, \ g|_{\Gamma} = 0 \Big\} \quad \text{if } \frac{1}{q} < 2 - \frac{2}{p} < 2 \label{I-1.16a}\\
			\Big( \lso,\mathcal{D}(A_q) \Big)_{1-\frac{1}{p},p} &= \Big\{ g \in \Bso : \text{ div } g = 0, \ g\cdot \nu|_{\Gamma} = 0 \Big\} \equiv \Bt(\Omega) \label{I-1.16b}\\
			&\hspace{4cm} \text{ if } 0 < 2 - \frac{2}{p} < \frac{1}{q}; \text{ or } 1 < p < \frac{2q}{2q - 1}.\nonumber
			\end{align}	
		\end{subequations}
		\noindent Notice that, in (\ref{I-1.16b}), the condition $\ds g \cdot \nu |_{\Gamma} = 0$ is an intrinsic condition of the space $\ds \lso$ in (\ref{I-1.4}), not an extra boundary condition as $\ds g|_{\Gamma} = 0$ in (\ref{I-1.16a}).
		\begin{rmk}
			In the analysis of well-posedness and stabilization of the nonlinear N-S problem (\ref{I-1.1}), with control $u$ in feedback form - such as the non linear translated feedback problem (\ref{I-2.20}) = (\ref{I-8.1})- we shall need to impose the constrain $q > 3$, see Eq (\ref{I-8.16}), to obtain the embedding $\ds W^{1,q} \hookrightarrow L^{\infty}(\Omega)$ in our case of interest $d=3$, as already noted below (\ref{I-1.11}). What is then the allowable range of the parameter $p$ in such case $q>3$? The intended goal of the present paper is to obtain the sought-after stabilization result in a function space, such as a $\ds \Bso$-space, that does not recognize boundary conditions of the I.C. Thus, we need to avoid the case in (\ref{I-1.16a}), as this implies a Dirichlet homogeneous B.C. Instead, we need to fit into the case (\ref{I-1.16b}). We shall then impose the condition $\ds 2 - \frac{2}{p} < \frac{1}{q} < \frac{1}{3}$ and then obtain that $p$ must satisfy $\ds p < \frac{6}{5}$. Moreover, analyticity and maximal regularity of the Stokes problem will require $p>1$. Thus, in conclusion, the allowed range of the parameters $p,q$ under which we shall solve the well-posedness and stabilization problem of the nonlinear N-S feedback system (\ref{I-2.20}) = (\ref{I-8.1}) for $d=3$, in the space $\ds \Bto$ which - as intended - does not recognize boundary conditions is: $\ds q > 3, \ 1 < p < \frac{6}{5}$. See Theorems \ref{I-Thm-2.3} through \ref{I-Thm-2.5}. 
		\end{rmk}
		
		\item \textbf{The Stokes and Oseen operators generate a strongly continuous analytic semigroup on $\lso$, $1 < q < \infty$}.
		\begin{thm}\label{I-Thm-1.2}
			Let $d \geq 2, 1 < q < \infty$ and let $\Omega$ be a bounded domain in $\mathbb{R}^d$ of class $C^3$. Then
			\begin{enumerate}[(i)]
				\item the Stokes operator $-A_q = P_q \Delta$ in (\ref{I-1.8}), repeated here as 
				\begin{equation}\label{I-1.17}
				-A_q \psi  = P_q \Delta \psi , \quad
				\psi \in \mathcal{D}(A_q) = W^{2,q}(\Omega) \cap W^{1,q}_0(\Omega) \cap \lso
				\end{equation}
				generates a s.c analytic semigroup $e^{-A_qt}$ on $\lso$. See \cite{Gi:1981} and the review paper \cite[Theorem 2.8.5 p 17]{HS:2016}.			
				\item The Oseen operator $\calA_q$ in (\ref{I-1.10}) \label{I-Thm-1.2(iii)}
				\begin{equation}\label{I-1.18}
				\calA_q  = - (\nu A_q + A_{o,q}), \quad \calD(\calA_q) = \calD(A_q) \subset \lso
				\end{equation}
				generates a s.c analytic semigroup $e^{\calA_qt}$ on $\lso$. This follows as $A_{o,q}$ is relatively bounded with respect to $A^{\rfrac{1}{2}}_q$, defined in (\ref{I-1.22}), see below (\ref{I-1.9}): thus a standard theorem on perturbation of an analytic semigroup generator applies \cite[Corollary 2.4, p 81]{P:1983}.
				
				\item \begin{subequations}\label{I-1.19}
					\begin{align}
					0 \in \rho (A_q) &= \text{ the resolvent set of the Stokes operator } A_q\\
					\begin{picture}(0,0)
					\put(-40,10){ $\left\{\rule{0pt}{18pt}\right.$}\end{picture}
					A_q^{-1} &: \lso \longrightarrow \lso \text{ is compact}
					\end{align}
				\end{subequations}			
			\item The s.c. analytic Stokes semigroup $e^{-A_qt}$ is uniformly stable on $\lso$: there exist constants $M \geq 1, \delta > 0$ (possibly depending on $q$) such that 
			\begin{equation}\label{I-1.20}
			\norm{e^{-A_qt}}_{\calL(\lso)} \leq M e^{-\delta t}, \ t > 0.
			\end{equation}
			
		\end{enumerate}
		\end{thm}
		\item \textbf{Domains of fractional powers, $\calD(A_q^{\alpha}), 0 < \alpha < 1$ of the Stokes operator $A_q$ on $\lso, 1 < q < \infty$}, 
		\begin{thm}\label{I-Thm-1.3}
			For the domains of fractional powers $\calD(A_q^{\alpha}), 0 < \alpha < 1$, of the Stokes operator $A_q$ in (\ref{I-1.8}) = (\ref{I-1.17}), the following complex interpolation relation holds true \cite{Gi:1985} and \cite[Theorem 2.8.5, p 18]{HS:2016}
			\begin{equation}\label{I-1.21}
			[ \calD(A_q), \lso ]_{1-\alpha} = \calD(A_q^{\alpha}), \ 0 < \alpha < 1, \  1 < q < \infty;
			\end{equation}
			in particular
			\begin{equation}\label{I-1.22}
			[ \calD(A_q), \lso ]_{\frac{1}{2}} = \calD(A_q^{\rfrac{1}{2}}) \equiv W_0^{1,q}(\Omega) \cap \lso.
			\end{equation}
			Thus, on the space $\calD(A_q^{\rfrac{1}{2}})$, the norms
			\begin{equation}\label{I-1.23}
			\norm{\nabla \ \cdot \ }_{L^q(\Omega)} \text{ and } \norm{ \ }_{L^q(\Omega)}
			\end{equation}
			are equivalent via Poincar\'{e} inequality.
		\end{thm}
		
		\item \textbf{The Stokes operator $-A_q$ and the Oseen operator $\calA_q, 1 < q < \infty$ generate s.c. analytic semigroups on the Besov space}\label{I-Sec-1.10d}
		\begin{subequations}\label{I-1.24}
			\begin{align}
			\Big( \lso,\mathcal{D}(A_q) \Big)_{1-\frac{1}{p},p} &= \Big\{ g \in \Bso : \text{ div } g = 0, \ g|_{\Gamma} = 0 \Big\} \quad \text{if } \frac{1}{q} < 2 - \frac{2}{p} < 2; \label{I-1.24a}\\
			\Big( \lso,\mathcal{D}(A_q) \Big)_{1-\frac{1}{p},p} &= \Big\{ g \in \Bso : \text{ div } g = 0, \ g\cdot \nu|_{\Gamma} = 0 \Big\} \equiv \Bt(\Omega) \label{I-1.24b}\\
			&\hspace{7cm} \text{ if } 0 < 2 - \frac{2}{p} < \frac{1}{q}.\nonumber
			\end{align}	
		\end{subequations}
		Theorem \ref{I-Thm-1.2} states that the Stokes operator $-A_q$ generates a s.c analytic semigroup on the space $\lso, \ 1 < q < \infty$, hence on the space $\calD(A_q)$ in (\ref{I-1.17}), with norm $\ds \norm{ \ \cdot \ }_{\calD(A_q)} = \norm{ A_q \ \cdot \ }_{\lso}$ as $0 \in \rho(A_q)$.  Then, one obtains that the Stokes operator $-A_q$ generates a s.c. analytic semigroup on the real interpolation spaces in (\ref{I-1.24}). Next, the Oseen operator $\calA = -(\nu A_q + A_{o,q})$ likewise generates a s.c. analytic semigroup $\ds e^{\calA_q t}$ on $\ds \lso$ since $A_{o,q}$ is relatively bounded w.r.t. $A_q^{\rfrac{1}{2}},$ as $A_{o,q}A_q^{-\rfrac{1}{2}}$ is bounded on $\lso$. Moreover $\calA_q$ generates a s.c. analytic semigroup on $\ds \calD(\calA_q) = \calD(A_q)$ (equivalent norms). Hence $\calA_q$ generates a s.c. analytic semigroup on the real interpolation space of (\ref{I-1.24}). Here below, however, we shall formally state the result only in the case $\ds 2-\rfrac{2}{p} < \rfrac{1}{q}$. i.e. $\ds  1 < p < \rfrac{2q}{2q-1}$, in the space $\ds \Bto$, as this does not contain B.C. The objective of the present paper is precisely to obtain stabilization results on spaces that do not recognize B.C.
		
		\begin{thm}\label{I-Thm-1.4}
			Let $1 < q < \infty, 1 < p < \rfrac{2q}{2q-1}$.
			\begin{enumerate}[(i)]
				\item The Stokes operator $-A_q$ in (\ref{I-1.17}) generates a s.c. analytic semigroup $e^{-A_qt}$ on the space $\Bt(\Omega)$ defined in (\ref{I-1.16}) = (\ref{I-1.24}) which moreover is uniformly stable, as in (\ref{I-1.20}),
				\begin{equation}\label{I-1.25}
				\norm{e^{-A_qt}}_{\calL \big(\Bt(\Omega)\big)} \leq M e^{-\delta t}, \quad t > 0.
				\end{equation}
				\item The Oseen operator $\calA_q$ in (\ref{I-1.18}) generates a s.c. analytic semigroup $e^{\calA_qt}$ on the space $\Bt(\Omega)$ in (\ref{I-1.16}) = (\ref{I-1.24}).
			\end{enumerate}
		\end{thm}
		\item \textbf{Space of maximal $L^p$ regularity on $\lso$ of the Stokes operator $-A_q, \ 1 < p < \infty, \ 1 < q < \infty $ up to $T = \infty$.}
		We return to the dynamic Stokes problem in $\{\varphi(t,x), \pi(t,x) \}$
		\begin{subequations}\label{I-1.26}
			\begin{align}
			\varphi_t - \Delta \varphi + \nabla \pi &= F &\text{ in } (0, T] \times \Omega \equiv Q\\		
			div \ \varphi &\equiv 0 &\text{ in } Q\\
			\begin{picture}(0,0)
			\put(-70,5){ $\left\{\rule{0pt}{35pt}\right.$}\end{picture}
			\left. \varphi \right \rvert_{\Sigma} &\equiv 0 &\text{ in } (0, T] \times \Gamma \equiv \Sigma\\
			\left. \varphi \right \rvert_{t = 0} &= \varphi_0 &\text{ in } \Omega,
			\end{align}
		\end{subequations}
		
		rewritten in abstract form, after applying the Helmholtz projection $P_q$ to (\ref{I-1.26}a) and recalling $A_q$ in (\ref{I-1.17}) as 
		\begin{equation}\label{I-1.27}
		\varphi' + A_q \varphi = \Fs \equiv P_q F, \quad \varphi_0 \in \lqaq
		\end{equation}
		
		Next, we introduce the space of maximal regularity for $\{\varphi, \varphi'\}$ as \cite[p 2; Theorem 2.8.5.iii, p 17]{HS:2016}, \cite[p 1404-5]{GGH:2012}, with $T$ up to $\infty$: 
		\begin{equation}\label{I-1.28}
		X^T_{p,q, \sigma} = L^p(0,T;\calD(A_q)) \cap W^{1,p}(0,T;\lso)
		\end{equation}
		(recall (\ref{I-1.8}) for $\calD(A_q)$) and the corresponding space for the pressure as 
		\begin{equation}\label{I-1.29}
		Y^T_{p,q} = L^p(0,T;\widehat{W}^{1,q}(\Omega)), \quad \widehat{W}^{1,q}(\Omega) = W^{1,q}(\Omega) / \mathbb{R}. 
		\end{equation}
		The following embedding, also called trace theorem, holds true \cite[Theorem 4.10.2, p 180, BUC for $T=\infty$]{HA:2000}, \cite{PS:2016}.
		\begin{equation}\label{I-1.30}
		\xtpqs \subset \xtpq \equiv L^p(0,T; W^{2,q}(\Omega)) \cap W^{1,p}(0,T; L^q(\Omega)) \hookrightarrow C \Big([0,T]; \Bso \Big).
		\end{equation}
		For a function $g$ such that $div \ g \equiv 0, \ g|_{\Gamma} = 0$ we have $g \in \xtpq \iff g \in \xtpqs$, by (\ref{I-1.4}).\\
		The solution of Eq(\ref{I-1.27}) is 
		\begin{equation}\label{I-1.31}
		\varphi(t) = e^{-A_qt} \varphi_0 + \int_{0}^{t} e^{-A_q(t-s)} \Fs(\tau) d \tau.
		\end{equation}
		The following is the celebrated result on maximal regularity on $\lso$ of the Stokes problem due originally to Solonnikov \cite{VAS:1977} reported in \cite[Theorem 2.8.5.(iii) and Theorem 2.10.1 p24 for $\varphi_0 = 0$]{HS:2016}, \cite{S:2006}, \cite[Proposition 4.1 , p 1405]{GGH:2012}. 
		\begin{thm}\label{I-Thm-1.5} 
			Let $1 < p,q < \infty, T \leq \infty$. With reference to problem (\ref{I-1.26}) = (\ref{I-1.27}), assume
			\begin{equation}\label{I-1.32}
			\Fs \in L^p(0,T;\lso), \ \varphi_0 \in \Big( \lso, \calD(A_q)\Big)_{1-\frac{1}{p},p}.
			\end{equation}
			Then there exists a unique solution $\varphi \in \xtpqs, \pi \in \ytpq$ to the dynamic Stokes problem (\ref{I-1.26}) or (\ref{I-1.27}), continuously on the data: there exist constants $C_0, C_1$ independent of $T, \Fs, \varphi_0$ such that via (\ref{I-1.30})
			\begin{equation}\label{I-1.33}
			\begin{aligned}
			C_0 \norm{\varphi}_{C \big([0,T]; \Bso \big)} &\leq \norm{\varphi}_{\xtpqs} +  \norm{\pi}_{\ytpq}\\ &\equiv \norm{\varphi'}_{L^p(0,T;\lso)} + \norm{A_q \varphi}_{L^p(0,T;\lso)} +  \norm{\pi}_{\ytpq}\\
			&\leq C_1 \bigg \{ \norm{\Fs}_{L^p(0,T;\lso)}  + \norm{\varphi_0}_{\big( \lso, \calD(A_q)\big)_{1-\frac{1}{p},p}} \bigg \}.
			\end{aligned}
			\end{equation}
			In particular,
			\begin{enumerate}[(i)]
				\item With reference to the variation of parameters formula (\ref{I-1.31}) of problem (\ref{I-1.27}) arising from the Stokes problem (\ref{I-1.26}), we have recalling (\ref{I-1.28}): the map
				\begin{align}
				\Fs &\longrightarrow \int_{0}^{t} e^{-A_q(t-\tau)}\Fs(\tau) d\tau \ : \text{continuous} \label{I-1.34}\\
				L^p(0,T;\lso) &\longrightarrow \xtpqs \equiv L^p(0,T; \calD(A_q)) \cap W^{1,p}(0,T; \lso) \label{I-1.35}				
				\end{align}			
				\item The s.c. analytic semigroup $e^{-A_q t}$ generated by the Stokes operator $-A_q$ (see (\ref{I-1.17})) on the space $\ds \Big( \lso, \calD(A_q)\Big)_{1-\frac{1}{p},p}$ (see statement below (\ref{I-1.24})) satisfies
				\begin{subequations}\label{I-1.36}
					\begin{equation}
					e^{-A_q t}: \ \text{continuous} \quad \Big( \lso, \calD(A_q)\Big)_{1-\frac{1}{p},p} \longrightarrow \xtpqs \equiv L^p(0,T; \calD(A_q)) \cap W^{1,p}(0,T; \lso) \label{I-1.36a}
					\end{equation}
					In particular via (\ref{I-1.24b}), for future use, for $1 < q < \infty, 1 < p < \frac{2q}{2q - 1}$, the s.c. analytic semigroup $\ds e^{-A_q t}$ on the space $\ds \Bto$, satisfies
					\begin{equation}
					e^{-A_q t}: \ \text{continuous} \quad \Bto \longrightarrow \xtpqs. \label{I-1.36b}
					\end{equation}
				\end{subequations}				 
				\item Moreover,
					 for future use, for $1 < q < \infty, 1 < p < \frac{2q}{2q - 1}$, then (\ref{I-1.33}) specializes to
					\begin{equation}\label{I-1.37}
					\norm{\varphi}_{\xtpqs} + \norm{\pi}_{\ytpq} \leq C \bigg \{ \norm{\Fs}_{L^p(0,T;\lso)} + \norm{\varphi_0}_{\Bto} \bigg \}.
					\end{equation}
			\end{enumerate}		
		\end{thm}
		
		\item \textbf{Maximal $L^p$ regularity on $\lso$ of the Oseen operator $\calA_q, \ 1 < p < \infty, \ 1 < q < \infty$, up to $T < \infty$.} We next transfer the maximal regularity of the Stokes operator $(-A_q)$ on $\lso$-asserted in Theorem \ref{I-Thm-1.5} into the maximal regularity of the Oseen operator $\calA_q = -\nu A_q - A_{o,q}$ in (\ref{I-1.18}) exactly on the same space $\xtpqs$ defined in (\ref{I-1.28}), however only up to $T < \infty$.
		
		\noindent Thus, consider the dynamic Oseen problem in $\{ \psi(t,x), \pi(t,x) \}$ with equilibrium solution $y_e$, see (\ref{I-1.2}):		
		\begin{subequations}\label{I-1.38}
			\begin{align}
			\psi_t - \Delta \psi + L_e(\psi) + \nabla \pi &= F &\text{ in } (0, T] \times \Omega \equiv Q \label{I-1.38a}\\		
			div \ \psi &\equiv 0 &\text{ in } Q\\
			\begin{picture}(0,0)
			\put(-100,7){$\left\{\rule{0pt}{35pt}\right.$}\end{picture}
			\left. \psi \right \rvert_{\Sigma} &\equiv 0 &\text{ in } (0, T] \times \Gamma \equiv \Sigma\\
			\left. \psi \right \rvert_{t = 0} &= \psi_0 &\text{ in } \Omega,
			\end{align}
		\end{subequations}
		\begin{equation}
		L_e(\psi) = (y_e . \nabla) \psi + (\psi. \nabla) y_e \hspace{6cm} \label{I-1.39}
		\end{equation}
		rewritten in abstract form, after applying the Helmholtz projector $P_q$ to (\ref{I-1.38a}) and recalling $\calA_q$ in (\ref{I-1.18}), as
		\begin{equation}\label{I-1.40}
		\psi_t = \calA_q \psi + P_q F = - \nu A_q \psi - A_{o,q} \psi + \Fs, \quad \psi_0 \in \big( \lso, \calD(A_q)\big)_{1-\frac{1}{p},p}
		\end{equation}
		whose solution is 
		\begin{equation}\label{I-1.41}
		\psi(t) = e^{\calA_qt} \psi_0 + \int_{0}^{t} e^{\calA_q(t-\tau)} \Fs(\tau) d \tau.
		\end{equation}
		\begin{equation}\label{I-1.42}
		\psi(t) = e^{-\nu A_qt} \psi_0 + \int_{0}^{t} e^{-\nu A_q(t-\tau)} \Fs(\tau) d \tau - \int_{0}^{t} e^{- \nu A_q(t-\tau)} A_{o,q} \psi(\tau) d \tau.
		\end{equation}
		
		\begin{thm}\label{I-Thm-1.6}
			Let $1 < p,q < \infty, \ 0 < T < \infty$. Assume (as in (\ref{I-1.32})) 
			\begin{equation}\label{I-1.43}
			\Fs \in L^p \big( 0, T; L^q_{\sigma} (\Omega) \big), \quad \psi_0 \in \lqaq
			\end{equation}
			where $\calD(A_q) = \calD(\calA_q)$, see (\ref{I-1.18}). Then there exists a unique solution $\psi \in \xtpqs, \ \pi \in \ytpq$ of the dynamic Oseen problem (\ref{I-1.38}), continuously on the data: that is, there exist constants $C_0, C_1$ independent of $\Fs, \psi_0$ such that
			\begin{align}
			C_0 \norm{\varphi}_{C \big([0,T]; \Bso \big)} &\leq \norm{\varphi}_{\xtpqs} + \norm{\pi}_{\ytpq}\nonumber \\ &\equiv \norm{\varphi'}_{L^p(0,T;L^q(\Omega))} + \norm{A_q \varphi}_{L^p(0,T;L^q(\Omega))} + \norm{\pi}_{\ytpq}\\
			&\leq C_T \bigg \{ \norm{\Fs}_{L^p(0,T;\lso)}  + \norm{\varphi_0}_{\lqaq} \bigg \}
			\end{align}
			where $T < \infty$. Equivalently, for $1 < p, q < \infty$
			\begin{enumerate}[i.]
				\item The map
				\begin{equation}
				\begin{aligned}
				\Fs \longrightarrow \int_{0}^{t} e^{\calA_q(t-\tau)}\Fs(\tau) d\tau \ : \text{continuous}&\\
				L^p(0,T;\lso) &\longrightarrow L^p \big(0,T;\calD(\calA_q) = \calD(A_q) \big)\label{I-1.46}
				\end{aligned}			
				\end{equation}
				where then automatically, see (\ref{I-1.40}) 
				\begin{equation}
				L^p(0,T;\lso) \longrightarrow W^{1,p}(0,T;\lso) \label{I-1.47}
				\end{equation}
				and ultimately
				\begin{equation}
				L^p(0,T;\lso) \longrightarrow \xtpqs \equiv L^p \big(0,T;\calD(A_q) \big) \cap W^{1,p}(0,T;\lso). \label{I-1.48}
				\end{equation}
				\item The s.c. analytic semigroup $e^{\calA_q t}$ generated by the Oseen operator $\calA_q$ (see (\ref{I-1.18})) on the space $\ds \lqaq $ satisfies for $1 < p, q < \infty$
				\begin{equation}
				e^{\calA_q t}: \ \text{continuous} \quad \lqaq \longrightarrow L^p \big(0,T;\calD(\calA_q) = \calD(A_q)  \big) \label{I-1.49}
				\end{equation}
				and hence automatically by (\ref{I-1.28})
				\begin{equation}
				e^{ \calA_q t}: \ \text{continuous} \quad \lqaq \longrightarrow \xtpqs. \label{I-1.50}
				\end{equation}
				In particular, for future use, for $1 < q < \infty, 1 < p < \frac{2q}{2q - 1}$, we have that the s.c. analytic semigroup $\ds e^{\calA_q t}$ on the space $\ds \Bto$, satisfies 
				\begin{equation}
				e^{\calA_q t}: \ \text{continuous} \quad \Bto \longrightarrow L^p \big(0,T;\calD(\calA_q) = \calD(A_q)  \big), \ T < \infty. \label{I-1.51}
				\end{equation}
				and hence automatically
				\begin{equation}
				e^{ \calA_q t}: \ \text{continuous} \quad \Bto \longrightarrow \xtpqs , \ T < \infty. \label{I-1.52}
				\end{equation}
			\end{enumerate}
		\end{thm}
		\noindent A proof is given in Appendix \hyperref[I-app-B]{B}. 
	\end{enumerate} 
	
	\begin{rmk}\label{I-Rmk-1.4}
		The literature reports physical situations where the volumetric force $f$ is actually replaced by $\nabla g(x)$; that is, $f$ is a conservative vector field. Thus, returning to Eq (\ref{I-1.2a}) with $f(x)$ replaced now by $\nabla g(x)$ we see that a solution of such stationary problem is $y_e = 0, \ \pi_e = g$, hence $L_e(\cdot) \equiv 0$ by (\ref{I-1.39}). Returning to Eq (\ref{I-1.1a}) with $f$ replaced by $\nabla g(x)$ and applying to the resulting equation the projection operator $P_q$, one obtains in this case the projected equation 
		\begin{equation}\label{I-1.53}
			y_t - \nu P_q \Delta y + P_q \big[ (y \cdot \nabla) y \big] = P_q (mu) \quad \text{in } Q.
		\end{equation}
		\noindent This, along with the solenoidal and boundary conditions (\ref{I-1.1b}), (\ref{I-1.1c}), yields the corresponding abstract form recalling also (\ref{I-1.11})
		\begin{equation}\label{I-1.54}
			y_t + \nu A_q y + \calN_q y = P_q (mu) \quad \text{in } \lso.
		\end{equation}
		\noindent Then $y$-problem (\ref{I-1.54}) is the same as the $z$-problem (\ref{I-1.12a}), except without the Oseen term $A_{o,q}$. The linearized version of problem (\ref{I-1.54}) is then 
		\begin{equation}\label{I-1.55}
			\eta_t + \nu A_q \eta = P_q (mu) \quad \text{in } \lso,
		\end{equation}
		\noindent which is the same as the $w$-problem (\ref{I-1.13a}), except without the Oseen term $A_{o,q}$. The s.c. analytic semigroup $\ds e^{-\nu A_q t}$ driving the linear equation (\ref{I-1.55}) is uniformly stable in $\lso$, see (\ref{I-1.20}), as well as in $\ds \Bto$, see (\ref{I-1.25}). Then, in the case of the present Remark, the present paper may be used to \uline{enhance at will the uniform stability} of the corresponding problem with $u$ given in feedback form as in the RHS of Eq (\ref{I-2.20}) as to obtain a decay rate much bigger than the original $\delta > 0$ in (\ref{I-1.20}) or (\ref{I-1.25}). Thus there is no need to perform the translation $y \longrightarrow z$ of Section \ref{I-Sec-1.8}, when $f$ in (\ref{I-1.2a}) is replaced by $\nabla g(x)$; i.e. $y_e = 0$ in this case. The important relevance of the present Remark will be pointed out in the follow-out paper \cite{LPT} where only finitely many localized tangential boundary feedback controls will be employed to the so far \underline{open} case dim $\Omega = 3$. The corresponding required ``unique continuation property" holds true for the Stokes problem ($y_e = 0$), see \cite{RT:2009}, \cite{RT:2008}.
 	\end{rmk}
	
	\section{Main results}
	
	\subsection{Orientation}\label{I-Sec-2.1}
	
	\noindent All the main results of this paper, Theorems \ref{I-Thm-2.1} through \ref{I-Thm-2.5}, are stated (at first) in the complex state space setting $\ds \lso + i \lso$. Thus, the finitely many stabilizing feedback vectors $p_k,u_k$ constructed in the subsequent proofs belong to the complex finite dimensional unstable subspace $(W^u_N)^*$ and $W^u_N$ respectively. The question then arises as to transfer back these results into the original real setting. This issue was resolved in \cite{BT:2004}. Here, such translation, taken from \cite{BT:2004}, from the results in the complex setting (Theorems \ref{I-Thm-2.1} through \ref{I-Thm-2.5})   into corresponding results in the original real setting is given in Section \ref{I-Sec-2.7}.\\         
	
	\noindent \underline{Step 1}: First, we will show in Theorem \ref{I-Thm-2.2} that the linearized Navier-Stokes problem $w_t = \calA_q w + P_q(mu)$ in (\ref{I-1.13}) can be uniformly (exponentially) stabilized in the basic space $\ds \lso, 1 < q < \infty$ in fact, in the space $\ds \calD(A_q^{\theta}), \ 0 \leq \theta \leq 1$, or $\ds \lqaq$, in particular $\ds \Bto$ by means of an explicitly constructed, finite dimensional spectral-based feedback controller $mu$, localized on $\omega$, whose structure is given in (\ref{I-2.16}).\\
	
	\noindent \underline{Step 2}: Next, we proceed to the non-linear translated Navier-Stokes $z$-problem (\ref{I-1.12}) with a control $u$ having the same structure as the finite-dimensional, spectral based stabilizing control used in the linearized problem (\ref{I-1.13}). This strategy leads to the non-linear feedback $z$-problem (\ref{I-2.20}). We then establish two results for problem (\ref{I-2.20}):\\
	
	\noindent (i) The first, Theorem \ref{I-Thm-2.3}, is that problem (\ref{I-2.20}) is locally well-posed,  i.e. for small initial data $z_0$, in the desired space $\ds \Bto$. It will require the constraint $q>3$, see (\ref{I-8.16}), to obtain $\ds W^{1,q}(\Omega) \hookrightarrow L^{\infty}(\Omega)$ for $d = 3$. In achieving this result, we must factor in what is the deliberate, sought-after goal of the present paper: that is, to obtain (well-posedness and) uniform stabilization of the original non-linear problem (\ref{I-1.1}) near an equilibrium solution, in a function space that does not recognize boundary conditions. This is the space $\ds \Bto$ having only the boundary condition $\ds \left. g \cdot \nu \right \rvert_{\Gamma} = 0$ inherited from the basic $\ds \lso$-space, see (\ref{I-1.16b})  and statement below it. In contrast, we deliberately exclude then the space in (\ref{I-1.16a}), $\ds p > \rfrac{2q}{2q - 1}$, having an explicit additional B.C. In conclusion, for the nonlinear problem, we need to work with the space $\ds \Bto$ in (\ref{I-1.16b}), and this requires for $d = 3$ the range $\ds q > 3, 1 < p < \rfrac{2q}{2q-1}$, that is $\ds 1 < p < \rfrac{6}{5}$ where the boundary conditions are not recognized. In this case the space $\ds \Bto = \big( \lso, \calD(A_q) \big)_{1-\rfrac{1}{p},p}$ with index $\ds 1 - \rfrac{1}{p}$ close to zero is ``close" to the space $\ds L^q(\Omega)$, for $q > 3$. Accordingly, with reference to the feedback $z$-problem (\ref{I-2.20}), we take $z_0 \in \Bto, \ q > 3, 1 < p < \rfrac{6}{5}$ sufficiently small, and show that (\ref{I-2.20}) is well-posed in the function space $\xipqs$ in (\ref{I-1.28}). To this end, we use critically the maximal regularity result Theorem \ref{I-Thm-7.1}. This is Theorem \hyperref[I-Thm-2.3]{2.3}.\\
	
	\noindent (ii) Second, we address the stabilization problem and show that such Navier-Stokes feedback problem  (\ref{I-2.20}) is, in fact, locally exponentially stabilizable in a neighborhood of the \textit{zero} equilibrium solution in the state space $\Bto$. This is Theorem \ref{I-Thm-2.4}.\\
	
	\noindent Such results, Theorem \ref{I-Thm-2.3} and the Theorem \ref{I-Thm-2.4} for the translated Navier-Stokes $z$-problem (\ref{I-2.20}) in feedback form  then at once translate into counterpart results of local well-posedness and local interior stabilization of the \textit{original} $y$-problem (\ref{I-1.1}) in a neighborhood of the equilibrium solution $y_e$, with an explicit finite dimensional feedback control localized on $\omega$ whose structure is given in (\ref{I-2.28b}). Thus Theorem \ref{I-2.5} gives the main result of the present paper.
	
	\subsection{Introducing the problem of feedback stabilization of the linearized $w$-problem (\ref{I-1.13}) on the complexified $\lso$ space.}\label{I-Sec-2.2}
	
	\noindent \textbf{Preliminaries:} In this subsection we take $q$ fixed, $1 < q < \infty$ throughout. Accordingly, to streamline the notation in the preceding setting of Section 1, we shall drop the dependence on $q$ of all relevant quantities and thus write $P, A, A_o, \calA$ instead of $P_q, A_q, A_{o,q}, \calA_q$. We return to the linearized system (\ref{I-1.13}).\\
	
	\noindent Moreover, as in \cite{BT:2004}, \cite{BLT1:2006}, we shall henceforth let $\lso$ denote the complexified space $\lso + i\lso$, whereby then we consider the extension of the linearized problem (\ref{I-1.13}) to such complexified space. Thus, henceforth, $w$ will mean $w + i \tilde{w}$, $u$ will mean $u + i \tilde{u}$, $w_0$ will mean $w_0 + i \tilde{w}_0$:	
	\begin{equation}\label{I-2.1}
	\frac{dw}{dt} + \nu Aw + A_0 w = P(mu), \quad \text{or} \quad \frac{dw}{dt} - \calA w = P(mu), \ w(0) = w_0 \text{ on } \lso.
	\end{equation}	
	\noindent As noted in Theorem \hyperref[I-Thm-1.2(iii)]{1.2(iii)}, the Oseen operator $\calA$ has compact resolvent on $\lso$. It follows that $\calA$ has a discreet point spectrum $\sigma(\calA) = \sigma_p(\calA)$ consisting of isolated eigenvalues $\{ \lambda_j\}_{j = 1}^{\infty}$, which are repeated according to their (finite) algebraic multiplicity $\ell_j$. However, since $\calA$ generates a $C_0$ analytic semigroup on $\lso$, its eigenvalues $\{ \lambda_j\}_{j = 1}^{\infty}$ lie in a triangular sector of a well-known type.\\
	
	\noindent The case of interest in stabilization occurs where $\calA$ has a finite number, say $N$, of eigenvalues $\lambda_1, \lambda_2 ,\lambda_3 ,\dots,\lambda_N$ on a complex half plane $\{ \lambda \in \mathbb{C} : Re~\lambda \geq 0 \}$ which we then order according to their real parts, so that	
	\begin{equation}\label{I-2.2}
	\ldots \leq Re~\lambda_{N+1} < 0 \leq Re~\lambda_N \leq \ldots \leq Re~\lambda_1,
	\end{equation}
	
	\noindent each $\lambda_i, \ i=1,\dots,N$, being an unstable eigenvalue repeated according to its geometric multiplicity $\ell_i$. Let $M$ denote the number of distinct unstable eigenvalues $\lambda_j$ of $\calA$, so that $\ell_i$ is equal to the dimension of the eigenspace corresponding to $\lambda_i$. Instead, $\ds N = \sum_{i = 1}^{M} N_i$ is the sum of the corresponding algebraic multiplicity $N_i$ of $\lambda_i$, where $N_i$ is the dimension of the corresponding generalized eigenspace.\\
	
	\noindent There are results in the literature \cite{JT:1993} that quantify the number of unstable eigenvalues in terms of the system parameters. Denote by $P_N$ and $P_N^*$ the projections given explicitly by \cite[p 178]{TK:1966}, \cite{BT:2004}, \cite{BLT1:2006}
	\begin{subequations}\label{I-2.3}
		\begin{align}
		\label{I-2.3a} P_N &= -\frac{1}{2 \pi i}\int_{\Gamma}\left( \lambda I - \calA \right)^{-1}d \lambda : \lso \text{ onto } W^u_N\\
		\label{I-2.3b} P_N^* &= -\frac{1}{2 \pi i}\int_{\bar{\Gamma}}\left( \lambda I - \calA^* \right)^{-1}d \lambda : (\lso)^* \text{ onto } (W^u_N)^* \subset \lo{q'},
		\end{align}
	\end{subequations}
	
	\noindent by (\ref{I-A.2c}), where $\Gamma$ (respectively, its conjugate counterpart $\bar{\Gamma}$) is a smooth closed curve that separates the unstable spectrum from the stable spectrum of $\calA$ (respectively, $\calA^*$). As in \cite[Sect 3.4, p 37]{BLT1:2006}, following \cite{RT:1975},we decompose the space $\lso$ into the sum of two complementary subspaces (not necessarily orthogonal):
	
	\begin{equation}\label{I-2.4}
	\lso = W^u_N \oplus W^s_N; \quad W^u_N \equiv P_N \lso ;\quad W^s_N \equiv (I - P_N) \lso; \quad \text{ dim } W^u_N = N 
	\end{equation}
	
	\noindent where each of the spaces $W^u_N$ and $W^s_N$ (which depend on $q$, but we suppress such dependence) is invariant under $\calA \ ( = \calA_q)$, and let
	
	\begin{equation}\label{I-2.5}
	\calA^u_N = P_N \calA = \calA |_{W^u_N} ; \quad \calA^s_N = (I - P_N) \calA = \calA |_{W^s_N}
	\end{equation}
	
	\noindent be the restrictions of $\calA$ to $W^u_N$ and $W^s_N$ respectively. The original point spectrum (eigenvalues) $\{ \lambda_j \}_{j=1}^{\infty} $ of $\calA$ is then split into two sets
	
	\begin{equation}\label{I-2.6}
	\sigma (\calA_N^u) = \{ \lambda_j \}_{j=1}^{N}; \quad  \sigma (\calA_N^s) = \{ \lambda_j \}_{j=N+1}^{\infty},
	\end{equation}
	
	\noindent and $W_N^u$ is the generalized eigenspace of $\calA^u_N$ in (\ref{I-2.1}). The system (\ref{I-2.1}) on $\lso$ can accordingly be decomposed as
	
	\begin{equation}\label{I-2.7}
	w = w_N + \zeta_N, \quad w_N = P_N w, \quad \zeta_N = (I-P_N)w. 
	\end{equation}
	
	\noindent After applying $P_N$ and $(I-P_N)$ (which commute with $\calA$) on (\ref{I-2.1}), we obtain via (\ref{I-2.5})
	
	\begin{subequations}\label{I-2.8} 
		\begin{equation}\label{I-2.8a}
		\text{on } W_N^u: w'_N - \calA^u_N w_N = P_N P (mu); \  w_N(0) = P_N w_0
		\end{equation}
		\begin{equation}\label{I-2.8b}
		\text{on } W_N^s: \zeta'_N - \calA^s_N \zeta_N = (I-P_N) P (mu); \  \zeta_N(0) = (I - P_N) w_0
		\end{equation}
	\end{subequations}
	\noindent respectively.\\
	
	\noindent \textbf{Main Result:} We may now state the main feedback stabilization result of the linearized problem (\ref{I-1.13}) (=(\ref{I-2.1})) on the complexified space $\lso$). The proof is constructive. How to construct the finitely many stabilizing vectors will be established in the proof.\\
	
	\noindent We anticipate the fact (noted in (\ref{I-3.2}) and (\ref{I-4.0})) below that, for $1 < p,q < \infty$:	
	\begin{equation}\label{I-2.9}
	\begin{aligned}
	W^u_N = &\text{ space of generalized}\\
	&\text{ eigenfunctions of $\calA_q (= \calA^u_N)$}\\
	&\text{ corresponding to its distinct}\\
	&\text{  unstable eigenvalues}
	\end{aligned}
	\subset
	\begin{cases}
	\begin{aligned}
	&\lqaq\\
	&\big[ \calD(A_q), \lso \big]_{1-\alpha} = \calD(A^{\alpha}_q), \ 0 \leq \alpha \leq 1
	\end{aligned}
	\end{cases}
	\subset \lso. 
	\end{equation}
	
	\subsection{Uniform (exponential) stabilization of the linear finite-dimensional $w_N$-problem (\ref{I-2.8a}) in the space $W^u_N$ by means of a finite-dimensional, explicit, spectral based feedback control localized on $\omega$.} \label{I-Sec-2.3}
	
	\begin{thm}\label{I-Thm-2.1}
		Let $\lambda_1,., \lambda_i,.,\lambda_M$ be the unstable distinct eigenvalues of the Oseen operator $\calA (= \calA_q)$ (see (\ref{I-1.10})) with geometric multiplicity $\ell_i$ and set $K = \sup \ \{\ell_j; j = 1,\dots, M \}.$ Let $\omega$ be an arbitrarily small open portion of the interior with sufficiently smooth boundary $\partial \omega$. Then: Given $\gamma > 0$ arbitrarily large, one can construct suitable interior vectors $[u_1,\dots,u_K]$ in the smooth subspace $W^u_N$ of $\ls(\omega), 1 < q < \infty$, and accordingly obtain a K-dimensional interior controller $u=u_N$ acting on $\omega$, of the form
		
		\begin{equation}\label{I-2.10}
		u = \sum_{k=1}^K \mu_k(t)u_k, \quad u_k \in W^u_N \subset \lso, \quad \mu_k(t) = \text{scalar},
		\end{equation}
		
		\noindent such that, once inserted in the finite dimensional projected $w_N$-system in (\ref{I-2.8}), yields the system
		
		\begin{equation}\label{I-2.11}
		w'_N - \calA^u_N w_N = P_N P \Bigg ( m \Bigg( \sum_{k=1}^{K} \mu_k(t) u_k \Bigg) \Bigg),
		\end{equation}
		
		\noindent whose solution then satisfies the estimate
		
		\begin{equation}\label{I-2.12}
		\norm{w_N(t)}_{\lso} + \norm{u_N(t)}_{\ls(\omega)} \leq C_{\gamma} e^{- \gamma t} \norm{P_N w_0}_{\lso}, \ t \geq 0.
		\end{equation}
		
		\noindent In (\ref{I-2.12}) we may replace the $\lso$-norm, $1 < q < \infty$, alternatively either with the $\ds \big(\lso, \calD(A_q) \big)_{1-\frac{1}{p},p}$ norm, $1 < q < \infty$; or else with the $\ds \big[ \calD(A_q), \lso \big]_{1-\alpha} = \calD(A^{\alpha}_q)$-norm, $ 0 \leq \alpha \leq 1,\ 1 < q < \infty$. In particular, we also have 
		\begin{equation}\label{I-2.13}
		\norm{w_N(t)}_{\Bto} + \norm{u_N(t)}_{\Bto} \leq C_{\gamma} e^{- \gamma t} \norm{P_N w_0}_{\Bto}, \ t \geq 0,
		\end{equation}
		\noindent in the $\Bto$-norm, $1 < q < \infty, \ p < \rfrac{2q}{2q - 1}$.\\
		\noindent [Estimate (\ref{I-2.13}) will be invoked in the nonlinear stabilization proof of Section \ref{I-Sec-9}.]\\
		
		\noindent Moreover, the above control $\ds u = u_N =  \sum_{k=1}^K \mu_k(t)u_k$, the terms  $u_k \in W^u_N,$ in (\ref{I-2.10}) can be chosen in feedback form: that is, of the form $\mu_k (t) = (w_N(t),p_k)_{\omega}$ for suitable vectors $p_k \in (W^u_N)^* \subset \lo{q'}$ depending on $\gamma$. Here and henceforth $\ds (v_1,v_2)_{\omega} = \int_{\omega} v_1 \cdot \bar{v}_2 \ d\omega, \ v_1 \in W^u_N \subset \lso, \ v_2 \in (W^u_N)^* \subset \lo{q'}$. In conclusion, $w_N$ in (\ref{I-2.11}) satisfying (\ref{I-2.12}),(\ref{I-2.13}) is the solution of the following equation on $W^u_N$ (see (\ref{I-2.8})):		
		\begin{subequations}\label{I-2.14}
			\begin{equation}\label{I-2.14a}
			w'_N - \calA^u_N w_N = P_N P \Bigg ( m \Bigg( \sum_{k=1}^{K} (w_N(t),p_k)_{\omega} u_k \Bigg) \Bigg), \ u_k \in W^u_N \subset \lso, \ p_k \in (W^u_N)^* \subset \lo{q'},
			\end{equation}
			
			\noindent rewritten as			
			\begin{equation}\label{I-2.14b}
			w'_N = \bar{A}^u w_N, \quad w_N(t) = e^{\bar{A}^u t}P_N w_0, \quad w_N(0) = P_N w_0.
			\end{equation}
		\end{subequations}
	\end{thm}
	\noindent A proof of Theorem \ref{I-Thm-2.1} is given in Section \ref{I-Sec-5}.
	\noindent \subsection{Global well-posedness and uniform exponential stabilization on the linearized $w$-problem (\ref{I-2.1}) in various $\lso$-based spaces, by means of the same feedback control obtained for the $w_N$-problem in Section \ref{I-Sec-2.3}.}
	
	\noindent Again, $1 < q < \infty$ throughout this section.
	\begin{thm}\label{I-Thm-2.2}
		With reference to the unstable, possibly repeated, eigenvalues $\{ \lambda \}_{j=1}^N$ in (\ref{I-2.2}), $M$ of which are distinct, let $\varepsilon > 0$ and set $\gamma_0 = |Re~ \lambda_{N+1}| - \varepsilon$. Then the same K-dimensional feedback controller 
		
		\begin{equation}\label{I-2.15}
		u = u_N = \sum_{k = 1}^{K} (w_N(t),p_k)_{\omega} u_k, \quad u_k \in W^u_N \subset \lso, \ p_k \in (W^u_N)^* \subset \lo{q'},
		\end{equation}
		
		\noindent constructed in Theorem \ref{I-Thm-2.1}, (\ref{I-2.14a}) and yielding estimate (\ref{I-2.12}), (\ref{I-2.13}) for the finite-dimensional projected $w_N$-system (\ref{I-2.8}), once inserted, this time in the full linearized $w$-problem (\ref{I-2.1}), yields the linearized feedback dynamics $(w_N = P_N w)$:
		
		\begin{equation}\label{I-2.16}
		\frac{dw}{dt} = \calA w + P \Bigg ( m \Bigg( \sum_{k=1}^{K} (P_N w,p_k)_{\omega} u_k \Bigg) \Bigg) \equiv \BA_{_F} w
		\end{equation}
		
		\noindent where $\BA_{_F}$ is the generator of a s.c. analytic semigroup in the space $\lso$. Here, $\ds \calA = \calA_q, \ P = P_q, \ \BA_{_F} = \BA_{_{F,q}}$Moreover, such dynamics $w$ in (\ref{I-2.16}) (equivalently, such generator $\BA_{_F}$ in (\ref{I-2.16})) is uniformly stable in the space $\lso$:
		
		\begin{equation}\label{I-2.17}
		\norm{e^{\BA_{_F} t} w_0}_{\lso} = \norm{w(t;w_0)}_{\lso} \leq C_{\gamma_0} e^{-\gamma_0 t}\norm{w_0}_{\lso}, \quad t \geq 0
		\end{equation} 
		\noindent or for $0 < \theta < 1$ and $\delta > 0$ arbitrarily small
		
		\begin{subequations}\label{I-2.18}
			\begin{align}
			&C_{\gamma_0, \theta} e^{-\gamma_0 t}\norm{A^{\theta}_q \ w_0}_{\lso}, \quad t \geq 0, \ w_0 \in \calD(A^{\theta}_q) \label{I-2.18a}\\
			\begin{picture}(110,0)
			\put(-96,10){$ \norm{ A^{\theta}_q \  e^{\BA_{_F} t} w_0}_{\lso} = \norm{ A^{\theta}_q \ w(t;w_0)}_{\lso} \leq \left\{\rule{0pt}{25pt}\right.$}\end{picture}
			&C_{\gamma_0, \theta, \delta} e^{-\gamma_0 t}\norm{w_0}_{\lso}, \quad t \geq \delta > 0. \label{I-2.18b}
			\end{align}
		\end{subequations}
		\noindent As in the case of Theorem \ref{I-2.1}, we may replace the $\lso$-norm in (\ref{I-2.17}), $1 < q < \infty$, with the \\ $\ds \lqaq$-norm, $1 < p,q < \infty$; in particular, with the $\ds \Bto$-norm
		\begin{equation}\label{I-2.19}
		\begin{aligned}
		\norm{e^{\BA_{_F} t} w_0}_{\Bto} = \norm{w(t;w_0)}_{\Bto} &\leq C_{\gamma_0} e^{-\gamma_0 t}\norm{w_0}_{\Bto}, \ t  \geq 0\\
		\quad & \hspace{4cm} 1 < q < \infty, \ 1 < p < \frac{2q}{2q-1}.
		\end{aligned}
		\end{equation}
	\end{thm}
	\noindent A proof of Theorem \ref{I-Thm-2.2} is given in Section \ref{I-Sec-6}.
	
	\subsection{Local well-posedness and uniform (exponential) null stabilization of the translated nonlinear $z$-problem (\ref{I-1.7}) or (\ref{I-1.12}) by means of a finite dimensional explicit, spectral based feedback control localized on $\omega$.}\label{I-Sec-2.5}
	
	Starting with the present section, the nonlinearity of problem (\ref{I-1.1}) will impose for $d=3$ the requirement $q > 3$, see (\ref{I-8.16}) below. As our deliberate goal is to obtain the stabilization result in the space $\ds \Bto$ which does not recognize boundary conditions, then the limitation $p < \rfrac{2q}{2q - 1}$ of this space applies. In conclusion, our well-posedness and stabilization results will hold under the restriction $q > 3, 1 < p < \rfrac{6}{5}$ for $d = 3$, and $q > 2, 1 < p < \rfrac{4}{3}$ for $d = 2$. 
	
	\begin{thm}\label{I-Thm-2.3}
		For $d = 3$, let $1 < p < \rfrac{6}{5}$ and $q > 3$, while for $d = 2$, let $1 < p < \rfrac{4}{3}$ and $q > 2$. Consider the nonlinear $z$-problem (\ref{I-1.12}) in the following feedback form
		\begin{equation}\label{I-2.20}
		\frac{dz}{dt} - \calA_q z + \calN_q z = P_q \Bigg ( m \Bigg( \sum_{k=1}^{K} (P_N z,p_k)_{\omega} u_k \Bigg) \Bigg)
		\end{equation}
		\noindent i.e. subject to a feedback control of the same structure as in the linear $w$-dynamics (\ref{I-2.16}), Here $p_k,u_k$ are the same vectors as constructed in Theorem \ref{I-Thm-2.1}, and appearing in (\ref{I-2.14}) or (\ref{I-2.16}). There exists a positive constant  $\rho > 0$ such that, if the initial condition $z_0$ satisfies
		\begin{equation}\label{I-2.21}
		\norm{z_0}_{\Bto} < \rho,
		\end{equation}
		then problem (\ref{I-2.20}) defines a unique solution $z$ in the space (see (\ref{I-1.28}), (\ref{I-1.30}))
		\begin{align}
		z \in \xipqs &\equiv L^p(0,\infty ; \calD(A_q)) \cap W^{1,p}(0, \infty; \lso)\label{I-2.22} \\
		&\hookrightarrow C([0, \infty); \Bto), \label{I-2.23}
		\end{align}
		where $\calD(A_q)$ is topologically $W^{2,q}(\Omega) \cap \lso$, see (\ref{I-1.8}).
	\end{thm}
	\noindent A proof of Theorem \ref{I-Thm-2.3} is given in Section \ref{I-Sec-8}.
	
	\begin{thm}\label{I-Thm-2.4}
		In the situation of Theorem \ref{I-Thm-2.3}, we have that such solution is uniformly stable on the space $\Bto$: there exist constants $\widetilde{\gamma} > 0, M_{\widetilde{\gamma}} \geq 1$, such that said solution satisfies 
		\begin{equation}\label{I-2.24}
		\begin{aligned}
		\norm{z(t;z_0)}_{\Bto} &\leq M_{\widetilde{\gamma}} e^{-\widetilde{\gamma} t}\norm{z_0}_{\Bto}, \ t \geq 0.
		\end{aligned}
		\end{equation}
	\end{thm}
	\noindent A proof of Theorem \ref{I-Thm-2.4} is given in Section \ref{I-Sec-9}. It will be critically based on the maximal regularity of the semigroup $\ds e^{\BA_{_F} t}$ giving the solution of the feedback $w$-problem (\ref{I-2.16}), $\ds \BA_{_{F}} = \BA_{_{F,q}}$. Remark \ref{I-Rmk-9.1}, at the end of Section \ref{I-Sec-9}, will provide insight on the relationship between $\wti{\gamma}$ in the nonlinear case in (\ref{I-2.24}) and $\gamma_0$ in the corresponding linear case in (\ref{I-2.17}). 
	
	\subsection{Local well-posedness and uniform (exponential) stabilization of the original nonlinear $y$-problem (\ref{I-1.1}) in a neighborhood of an equilibrium solution $y_e$, by means of a finite dimensional explicit, spectral based feedback control localized on $\omega$.}\label{I-Sec-2.6}
	
	\noindent The result of this subsection is an immediate corollary of section \ref{I-Sec-2.5}.
	
	\begin{thm}\label{I-Thm-2.5}
		Let $1 < p < \rfrac{6}{5}, q > 3, d = 3$; and $1 < p < \rfrac{4}{3}, q > 2, d = 2$. Consider the original N-S problem (\ref{I-1.1}). Let $y_e$ be a given equilibrium solution as guaranteed by Theorem \ref{I-Thm-1.1} for the steady state problem (\ref{I-1.2}). For a constant $\rho > 0$, let the initial condition $y_0$ in (\ref{I-1.1d}) be in $\ds \Bto$ and satisfy 
		\begin{equation}\label{I-2.25}
		\calV_{\rho} \equiv \Big\{ y_0 \in \Bto: \norm{y_0 - y_e}_{\Bto} \leq \rho \Big\}, \quad \rho > 0.
		\end{equation}
		\noindent If $\rho > 0$ is sufficiently small, then 
		\begin{enumerate}[(i)]
			\item for each $y_0 \in \calV_{\rho}$, there exists an interior finite dimensional feedback controller \label{I-Thm-2.5.i}
			\begin{equation}\label{I-2.26}
			u = F(y - y_e) = \sum_{k=1}^{K} (P_N(y - y_e), p_k)_{\omega} u_k
			\end{equation}
			\noindent that is, of the same structure as in the translated N-S $z$-problem (\ref{I-2.20}), with the same vectors $\ds p_k, u_k$ in (\ref{I-2.14}) or (\ref{I-2.16}), such that the closed loop problem corresponding to (\ref{I-1.1})
			\begin{subequations}\label{I-2.27}
				\begin{align}
				y_t - \nu \Delta y +  (y \cdot \nabla)y + \nabla \pi & = m(F(y - y_e)) + f(x)  &\text{ in } Q \label{I-2.27a}\\ 
				div \ y &= 0  &\text{ in } Q \label{I-2.27b}\\
				\begin{picture}(120,0)
				\put(-18,9){$\left\{\rule{0pt}{35pt}\right.$}\end{picture}
				y &= 0 &\text{ on } \Sigma \label{I-2.27c}\\
				y|_{t = 0} &= y_0 &\text{ in } \Omega \label{I-2.27d}
				\end{align}
			\end{subequations}
			\noindent rewritten abstractly after application of the Helmholtz projection $P_q$ as
			\begin{subequations}\label{I-2.28}
				\begin{align}
				y_t + \nu A_q y +  \calN_q y &= P_q \Big[ m \big( F(y - y_e)\big) + f(x) \Big] \label{I-2.28a}\\ 
				&= P_q \bigg[ m \bigg( \sum_{k=1}^{K} \big(P_N(y - y_e), p_k \big)_{\omega} u_k \bigg) + f(x) \bigg]\label{I-2.28b}\\ 
				y(0) &= y_0 \in \Bto \label{I-2.28d}
				\end{align}
			\end{subequations}
			\noindent has a unique solution $\ds y \in C \big([0,\infty); \Bto \big).$
			\item Moreover, such solution exponentially stabilizes the equilibrium solution $y_e$ in the space $\ds \Bto$: there exist constants $\widetilde{\gamma} > 0$ and $M_{\widetilde{\gamma}} \geq 1$ such that said solution satisfies
			\begin{equation}\label{I-2.29}
			\norm{y(t) - y_e}_{\Bto} \leq M_{\widetilde{\gamma}} e^{- \widetilde{\gamma} t}\norm{y_0 - y_e}_{\Bto}, \quad t \geq 0, \ y_0 \in \calV_{\rho}.
			\end{equation}
			\noindent Once the neighborhood $\calV_{\rho}$ is obtained to ensure the well-posedness, then the values of $M_{\widetilde{\gamma}}$ and  $\widetilde{\gamma}$ do not depend on $\calV_{\rho}$ and $\widetilde{\gamma}$ can be made arbitrarily large through a suitable selection of the feedback operator $F$.
		\end{enumerate}	
	\end{thm} 

	\noindent  See Remark \ref{I-Rmk-9.1} comparing $\wti{\gamma}$ in (\ref{I-2.29}) with $\gamma_0$ in (\ref{I-2.17}).
	
	\subsection{Results on the real space setting}\label{I-Sec-2.7}
	
	\noindent Here we shall complement the results of Theorems \ref{I-Thm-2.1} through \ref{I-Thm-2.5} by giving their version in the real space setting. We shall quote from \cite{BT:2004}. In the complexified setting $\ds \lso + i \lso$
	we have that the complex unstable subspace $W^u_N$ is,
	\begin{align}
	W^u_N &= W^1_N + i W^2_N \\
	&= \text{space of generalized eigenfunctions } \{\phi_j\}_{j=1}^N \text{ of the operator } \calA_q (=\calA_q^u) \text{ corresponding to} \nonumber \\
	& \quad \text{ its } N \text{ unstable eigenvalues.}
	\end{align}
	\noindent Set $\ds \phi_j = \phi_j^1 + i \phi_j^2$ with $\phi_j^1,\phi_j^2$ real. Then:
	\begin{equation}
	W^1_N = Re \ W^u_N = \text{span} \{ \phi_j^1 \}_{j = 1}^N; \quad W^2_N = Im \ W^u_N = \text{span} \{ \phi_j^2 \}_{j = 1}^N.
	\end{equation}
	\noindent The stabilizing vectors $p_k, u_k, \ k = 1, \dots, K$ are complex valued with $u_k \in W^u_N \subset \lso$, and $p_k \in (W^u_N)^* \subset \lo{q'}$, as in (\ref{I-2.15}).\\
	
	\noindent The complex-valued uniformly stable linear $w$-system in (\ref{I-2.16}) with $K$ complex valued stabilizing vectors admits the following real-valued uniformly stable counterpart
	\begin{equation}
	\frac{dw}{dt} = \calA_q w + P_q \Bigg ( m \Bigg( \sum_{k=1}^{K} \text{Re } (w_N(t),p_k)_{\omega} \text{ Re } u_k - \sum_{k=1}^{K} \text{Im } \ (w_N(t),p_k)_{\omega} \text{ Im } u_k\Bigg) \Bigg)
	\end{equation}
	\noindent with $2K \leq N$ real stabilizing vectors, see \cite[Eq 3.52a, p 1472]{BT:2004}. If $K = \text{sup } \{ \ell_i, i = 1, \dots,M \}$ is achieved for a real eigenvalue $\lambda_i$ (respectively, a complex eigenvalue $\lambda_i$), then the \textit{effective} number of stabilizing controllers is $K \leq N$, as the generalized functions are then real, since $y_e$ is real; respectively, $2K \leq N$, for, in this case, the complex conjugate eigenvalue $\bar{\lambda}_j$ contributes an equal number of components in terms of generalized eigenfunctions $\ds \phi_{\bar{\lambda}_j} = \bar{\phi}_{\lambda_j}$. In all cases, the actual (\textit{effective}) upper bound $2K$ is $2K \leq N$. For instance, if all unstable eigenvalues were real and simple then $K=1$, and only one stabilizing controller is actually needed.\\
	
	\noindent Similarly, the complex-valued locally (near $y_e$) uniformly stable nonlinear $y$-system (\ref{I-2.28}) with $K$ complex-valued stabilizing vectors admits the following real-valued locally uniformly stable counterpart
	\begin{equation}
	\frac{dy}{dt} - \nu A_q y + \calN_q y = P_q \Bigg ( m \Bigg( \sum_{k=1}^{K} \text{Re }(y - y_e,p_k)_{\omega} \text{ Re } u_k - \sum_{k=1}^{K} \text{Im } (y - y_e,p_k)_{\omega} \text{ Im } u_k\Bigg) \Bigg)
	\end{equation} 
	\noindent with $2K \leq N$ real stabilizing vectors, see \cite[p 43]{BLT1:2006}.
	
	\section{Algebraic rank condition for the $w_N$-dynamics in (\ref{I-2.8a}) under the (preliminary) Finite-Dimensional Spectral Assumption (FDSA)}\label{I-Sec-3}
	
	\textbf{Preliminaries:} For $i = 1,\dots,M$, we now denote by $\phiij$, $\phiijs$ the normalized linearly independent eigenfunctions (on $\lso$ and $\ds (\lso)' = \lo{q'}$, respectively $\ds \rfrac{1}{q} + \rfrac{1}{q'} = 1$ invoking property (\ref{I-A.2b}) of Appendix \hyperref[I-app-A]{A} )  corresponding to the unstable distinct eigenvalues $\lambda_1,\ldots,\lambda_M$ of $\calA$ and $\bar{\lambda}_1,\ldots,\bar{\lambda}_M$ of $\calA^*$, respectively:
	
	\begin{subequations}\label{I-3.1}
		\begin{align}
		\calA \phi_{ij} &= \lambda_i \phi_{ij} \in \calD(\calA_q) = W^{2,q}(\Omega) \cap W^{1,q}_0(\Omega) \cap \lso \in L^q(\Omega)  \label{I-3.1a}\\
		\calA^* \phi_{ij}^* &= \bar{\lambda}_i \phi_{ij}^* \in \calD(\calA_q^*) = W^{2,q'}(\Omega) \cap W^{1,q'}_0(\Omega) \cap \lo{q'} \in L^{q'}(\Omega). \label{I-3.1b}
		\end{align}
	\end{subequations}

	\noindent FDSA: We henceforth assume in this section that for each of the distinct eigenvalues $\lambda_1,\ldots,\lambda_M$ of $\calA$, algebraic and geometric multiplicity coincide:
	
	\begin{equation}\label{I-3.2}
	W_{N,i}^u \equiv P_{N,i} \lso = \text{span} \phiij; \quad (W_{N,i}^u)^* \equiv P_{N,i}^* (\lso)^* = \text{span} \phiijs;
	\end{equation}
	
	\noindent Here $P_{N,i},P_{N,i}^*$ are the projections corresponding to the eigenvalues $\lambda_i$ and $\bar{\lambda}_i$, respectively. For instance, $P_{N,i}$ is given by an integral such as the one on the RHS of (\ref{I-2.3a}), where now $\Gamma$ is a closed smooth curve encircling the eigenvalue $\lambda_i$ and no other. Similarly $P_{N,i}^*$.
	The space $W_{N,i}^u = $ range of $P_{N,i}$ is the algebraic eigenspace of the eigenvalues $\lambda_i$, and $\ell_i =$ dim $W_{N,i}^u$ is the algebraic multiplicity of $\lambda_i$, so that $\ell_1 + \ell_2 + \dots +\ell_M = N$.  As a consequence of the FDSA, we obtain 
	
	\begin{equation}\label{I-3.3}
	W_N^u \equiv P_N \lso = \text{span} \phiijm; \quad (W_N^u)^* \equiv P_N^* (\lso)^* = \text{span} \phiijms.
	\end{equation}
	
	\noindent Without the FDSA, $W_N^u$ is the span of the generalized eigenfunctions of $\calA$, corresponding to its unstable distinct eigenvalues $\{\lambda_j \}_{j=1}^M$; and similarly for $(W_N^u)^*$ (see the subsequent section). In other words, the FDSA says that the restriction $\calA_N^u$ in (\ref{I-2.5}) is \textit{diagonalizable} or that $\calA_N^u$ is \textit{semisimple} on $W_N^u$ in the terminology of \cite[p 43]{TK:1966}. Under the FDSA, any vector $w \in W_N^u$ admits the following unique expansion \cite[p 12, Eq (2.16)]{TK:1966}, \cite[p 1453]{BT:2004}, in terms of the basis $\phiijm$ in $\lso$ and its adjoint basis \cite[p 12]{TK:1966} $\phiijms$ in $(\lso)^*$:
	
	\begin{equation}\label{I-3.4}
	W_N^u \ni w = \sum_{i,j}^{M,\ell_i} (w,\phi_{ij}^*)\phi_{ij}; \quad
	(\phi_{ij},\phi_{hk}^*) =
	\begin{cases}
	1 & \text{if $i=h,j=k$} \\
	0 & \text{otherwise} \\
	\end{cases}
	\end{equation}
	
	\noindent that is, the system consisting of $\{ \phi_{ij} \}$ and $\{ \phi_{ij}^* \}$, $i = 1,\ldots,M, \ j = 1,\ldots,\ell_i$, can be chosen to form bi-orthogonal sequences. Here $( \ , \ )$ denotes the scalar product between $W^u_N$ and $(W^u_N)^*$ \cite[p 12]{TK:1966}. i.e. ultimately, the duality pairing in $\Omega$ between $\lso$ and $(\lso)^*$. Next, we return to the $w_N$-dynamics in (\ref{I-2.8a}), rewritten here for convenience 
	
	\begin{equation}\label{I-3.5}
	\text{on } W_N^u: w'_N - \calA^u_N w_N = P_N P (mu); \  w_N(0) = P_N w_0.
	\end{equation}
	
	\noindent \textbf{The term $P_NP(mu)$ expressed in terms of adjoint bases.} Next, let $mu \in L^q(\omega)$ where $q > 1$. Here below we compute the RHS of the term $P_NP(mu)$ via the adjoint bases expansion in (\ref{I-3.4}), where we notice that $P^*P_N^* \phi_{ij}^* = \phi_{ij}^*$ because $\phi_{ij}^* \in \calD(\calA^*)$, so that $\phi_{ij}^*$ is invariant under the projections $P^*$ and $P_N^*$. With $\ds (f,g)_{\omega} = \int_{\omega} f \bar{g} \ d \omega$, we obtain
	
	\begin{equation}\label{I-3.6}
	W_N^u \ni P_NP(mu) = \sum_{i,j=1}^{M,\ell_i}(P_NP(mu),\phi_{ij}^*)\phi_{ij} =  \sum_{i,j=1}^{M,\ell_i}(mu,\phi_{ij}^*)\phi_{ij} = \sum_{i,j=1}^{M,\ell_i}(u,\phi_{ij}^*)_{\omega}\phi_{ij},
	\end{equation}
	
	\noindent so that the dynamics (\ref{I-3.5}) on $W^u_N$ becomes
	
	\begin{equation}\label{I-3.7}
	\text{on } W_N^u: w'_N - \calA^u_N w_N = \sum_{i,j=1}^{M,\ell_i}(u,\phi_{ij}^*)_{\omega}\phi_{ij}. \  
	\end{equation} 
	
	\noindent \textbf{Selection of the scalar interior control function $u$ in finite dimensional separated form (with respect to $K$ coordinates).} Next, we select the control $u$ of the form given in (\ref{I-2.10})
	
	\begin{equation}\label{I-3.8}
	u = \sum_{k=1}^K \mu_k(t)u_k, \quad u_k \in W^u_N \subset \lso, \quad \mu_k(t) = \text{scalar}
	\end{equation}
	
	\noindent so that the term in (\ref{I-3.6}) in $W^u_N$ specializes to 
	
	\begin{equation}\label{I-3.9}
	W_N^u \ni P_NP(mu) = \sum_{i,j=1}^{M,\ell_i} \Bigg\{ \sum_{k=1}^{K} (u_k,\phi_{ij}^*)_{\omega} \mu_k(t) \Bigg\} \phi_{ij}.\\
	\end{equation}
	
	\noindent Substituting (\ref{I-3.9}) on the RHS of (\ref{I-3.5}), we finally obtain
	
	\begin{equation}\label{I-3.10}
	\text{on } W_N^u: w'_N - \calA^u_N w_N = \sum_{i,j=1}^{M,\ell_i} \Bigg\{ \sum_{k=1}^{K} (u_k,\phi_{ij}^*)_{\omega} \mu_k(t) \Bigg\} \phi_{ij}.
	\end{equation} 
	
	\noindent \textbf{The dynamics (\ref{I-3.10}) in coordinate form on $W^u_N$.} Our next goal is to express the finite dimensional dynamics (\ref{I-3.10}) on the $N$-dimensional space $W^u_N$ in a component-wise form. To this end, we introduce the following ordered bases $\beta_i$ and $\beta$ of length $\ell_i$ and $N$ respectively:	
	\begin{equation}\label{I-3.11}
	\begin{aligned}
	\beta_i &= [ \phi_{i1},\dots,\phi_{i \ell_i} ]: \ \text{ basis on } W^u_{N,i}\\
	\beta &= \beta_1 \cup \beta_2 \cup \dots \cup \beta_M = [ \phi_{11},\dots,\phi_{1 \ell_1},\phi_{21},\dots,\phi_{2 \ell_2},\dots,\phi_{M 1},\dots,\phi_{M \ell_M} ]: \ \text{ basis on } W^u_N.
	\end{aligned}
	\end{equation}
	
	\noindent Thus, we can represent the $N$-dimensional vector $w_N \in W^u_N$ as column vector $\hat{w}_N = [w_N]_{\beta}$ as,	
	\begin{equation}\label{I-3.12}
	w_N = \sum_{i,j=1}^{M,\ell_i} w_N^{ij} \phi_{ij}; \text{ and set } \hat{w}_N = \text{col}[w_N^{1,1},\dots,w_N^{1, \ell_1},\dots,w_N^{i,1},\dots,w_N^{i, \ell_i},\dots,w_N^{M,1},\dots,w_N^{M,\ell_M}].
	\end{equation}
	
	\begin{rmk}\label{I-Rmk-3.1}
		The eigenfunction $\phi_{ij}$ belongs to $\lso$ as well as to $\ds \calD(A_q) = \calD(\calA_q)$. Thus, by real/complex interpolation, see (\ref{I-1.16})/(\ref{I-1.21}) they also belong to $\ds \lqaq$ as well as to $\ds \big[ \calD(A_q), \lso \big]_{1-\alpha} = \calD(A^{\alpha}_q), \ 0 \leq \alpha \leq 1$; in particular, $\phi_{ij} \in \Bto$. See (\ref{I-B.11}) or (\ref{I-B.12}) in Appendix \hyperref[I-app-B]{B}. Thus, exponential decay in $\BC^N$ of the $\BC^N$-vector $\hat{w}_N$ translates at once into exponential decay with the same rate in any of the spaces $\ds \lso, \lqaq, \calD(A^{\alpha}_q)$, in particular, $\ds \Bto$ for the vector $w_N$, views as a vector on any one of these spaces. This remark applies to $w_N(t)$ and $u_N(t)$ in Theorem \ref{I-Thm-2.1}, Eqts (\ref{I-2.12}), (\ref{I-2.13}) as well as Theorem \ref{I-Thm-2.2}, Eqts (\ref{I-2.17})- (\ref{I-2.19}).
	\end{rmk}  
	
	\begin{lemma}
		In $\mathbb{C}^{N}$, with respect to the ordered basis $\beta : \ \phiijm$ of normalized eigenfunctions of $\calA_N^u$, we may rewrite system (\ref{I-3.10}) = (\ref{I-3.12}) = (\ref{I-2.8a}) as
		
		\begin{equation}\label{I-3.13}
		(\hat{w}_N)' - \Lambda \hat{w}_N = U \hat{\mu}_K
		\end{equation}
		
		where 
		
		\begin{equation}\label{I-3.14}
		\Lambda = \begin{bmatrix}
		\lambda_1 I_1 & & & \text{\huge0}\\
		& \lambda_2 I_2 & & &\\
		& & \ddots & &\\
		\text{\huge0} & & & \lambda_M I_M
		\end{bmatrix} : N \times N, \quad I_i:\ell_i \times \ell_i \ \text{identity}
		\end{equation}
		
		\begin{equation}\label{I-3.15}
		U_i = 
		\begin{bmatrix}
		(u_1,\phi_{i1}^*)_{\omega} & \dots & (u_K,\phi_{i1}^*)_{\omega} \\
		(u_1,\phi_{i2}^*)_{\omega} & \dots & (u_K,\phi_{i2}^*)_{\omega} \\
		\vdots & \ddots & \vdots \\
		(u_1,\phi_{i \ell_i}^*)_{\omega} & \dots & (u_K,\phi_{i \ell_i}^*)_{\omega} \\
		\end{bmatrix}
		:\ell_i \times K; \quad U = \begin{bmatrix}
		U_1 \\
		U_2 \\
		\vdots\\
		U_M
		\end{bmatrix} :N \times K; \quad \hat{\mu}_K = \begin{bmatrix}
		\mu_1 \\
		\mu_2 \\
		\vdots\\
		\mu_K
		\end{bmatrix} : K \times 1; 
		\end{equation}
		
		\noindent where $\ds (f,g)_{\omega} = \int_{\omega} f \bar{g} \ d \omega$ and we take $K \geq \ell_i, \ i = 1, \dots, M$. Thus (\ref{I-3.13}) gives the dynamics on $W^u_N$ as a linear $N$-dimensional ordinary differential equation in coordinate form in $\BC^N$.
	\end{lemma}
	
	\begin{proof}
		\noindent Recalling the basis $\beta_i$ and the definitions of $U_i$ in (\ref{I-3.15}), we can rewrite the term in (\ref{I-3.9}) with respect to this basis as
		
		\begin{equation}\label{I-3.16}
		[P_NP(mu)]_{\beta_i} = U_i \hat{\mu}_K: \ell_i \times 1;
		\end{equation}
		
		\noindent Then with respect to the basis $\beta$ in (\ref{I-3.11}) and recalling the definition $U$ in (\ref{I-3.15}), we can rewrite the term (\ref{I-3.9}) with respect to this basis as
		
		\begin{equation}\label{I-3.17}
		[P_NP(mu)]_{\beta} = \begin{bmatrix}
		U_1 \\
		U_2 \\
		\vdots\\
		U_M
		\end{bmatrix}\hat{\mu}_K = \begin{bmatrix}
		U_1 \hat{\mu}_K \\
		U_2 \hat{\mu}_K \\
		\vdots\\
		U_M \hat{\mu}_K
		\end{bmatrix} = U \hat{\mu}_K: N \times 1.
		\end{equation} 
		
		\noindent Finally, clearly $\calA_N^u$ becomes the diagonal matrix $\Lambda$ in (\ref{I-3.14}) with respect to the basis $\beta$, recalling its eigenvalues in (\ref{I-3.1}).	
	\end{proof}
	
	\noindent The following is the main result of the present section.
	
	\begin{thm}\label{I-Thm-3.1}
		Assume the FDSA. It is possible to select vectors $u_1,\dots,u_K \in \ls(\omega), q > 1, K = \sup \ \{\ell_i:i = 1,\dots,M \}$, such that the matrix $U_i$ of size $\ell_i \times K$ in (\ref{I-3.15}) satisfies		
		\begin{equation}\label{I-3.18}
		rank [U_i] = full = \ell_i \text{ or } 
		\text{ rank}
		\begin{bmatrix}
		(u_1,\phi_{i1}^*)_{\omega} & \dots & (u_K,\phi_{i1}^*)_{\omega} \\
		(u_1,\phi_{i2}^*)_{\omega} & \dots & (u_K,\phi_{i2}^*)_{\omega} \\
		\vdots &  & \vdots\\
		(u_1,\phi_{i \ell_i}^*)_{\omega}& \dots & (u_K,\phi_{i \ell_i}^*)_{\omega} \\
		\end{bmatrix}
		= \ell_i; \ \ell_i \times K \ \text{for each } i = 1,\dots,M.
		\end{equation}		
	\end{thm}
	
	\begin{proof}
		\textit{Step 1}. By selection, see (\ref{I-3.1}) and statement preceding it, the set of vectors $\phi^*_{i 1}, \dots, \phi^*_{i \ell_i}$ is linearly independent in $\lo{q'}$, $q'$  is the H\"{o}lder conjugate of $\ds q, \rfrac{1}{q} + \rfrac{1}{q'} = 1 $, for each $i = 1, \dots, M$. Next, if the set of vectors  $\{\phi_{i1}^*,\dots,\phi_{i \ell_i}^*\}$ were linearly independent in $\lo{q'}$, $i = 1, \dots, M$, the desired conclusion (\ref{I-3.18}) for the matrix $U_i$ to be full rank, would follow for infinitely many choices of the vectors  $u_1,\dots,u_K \in \lso$.\\
		
		\noindent \textit{Claim:} The set $\{\phi_{i1}^*,\dots,\phi_{i \ell_i}^*\}$ is linearly independent on $L^{q'}_{\sigma}(\omega)$, for each $i = 1, \dots, M$.\\

		\noindent The proof will critically depend on a unique continuation result \cite{RT:2009} see also \cite[Lemma 3.7 p1466]{BT:2004}. By contradiction, let us assume that the vectors $\{\phi_{i1}^*,\dots,\phi_{i \ell_i}^*\} \in \lo{q'}$ are instead linearly dependent, so that 
		
		\begin{equation}\label{I-3.19}
		\phi_{i \ell_i}^* = \sum_{j=1}^{\ell_i - 1} \alpha_j \phi_{i \ell_j}^* \text{ in } \lo{q'}
		\end{equation}
		
		\noindent with constants $\alpha_j$ not all zero. We shall then conclude by \cite[Lemma 3.7]{BT:2004} and \cite{RT:2009} below, that in fact $\phi_{i \ell_i}^* \equiv 0$ on all of $\Omega$ as well, thereby making the system $\{ \phi_{ij}^*, j = 1,\dots,\ell_i \}$ linearly dependent on $\Omega$, a contradiction. To this end, define the following function (depending on $i$) in $\lo{q'}$
		
		\begin{equation}\label{I-3.20}
		\phi^* = \Bigg[ \sum_{j=1}^{\ell_i - 1} \alpha_j \phi_{i \ell_j}^* - \phi_{i \ell_i}^* \Biggl] \in \lo{q'}, \quad i = 1,\dots,M.
		\end{equation}

		\noindent As each $\phi_{ij}^*$ is an eigenvalue of $\calA^*$ (or $(\calA_N^u)^*$) corresponding to the eigenvalue $\bar{\lambda}_i$, see (\ref{I-3.1}), so is the linear combination $\phi^*$. This property, along with (\ref{I-3.19}) yields that $\phi^*$ satisfies the following eigenvalue problem for the operator $\calA^*$ (or $(\calA_N^u)^*$):
		
		\begin{equation}\label{I-3.21}
		\calA^* \phi^* = \bar{\lambda} \phi^*, \text{ div } \phi^* = 0 \text{ in } \Omega; \quad \phi^* = 0 \text{ in } \omega, \text{ by (\ref{I-3.19})}.
		\end{equation}
		
		But the linear combination $\phi^*$ in (\ref{I-3.20}) of the eigenfunctions $\phi_{ij}^* \in \calD(\calA^*)$ satisfies itself the Dirichlet B.C $\ds \left. \phi^* \right \rvert_{\partial \Omega} = 0$. Thus the explicit PDE version of problem (\ref{I-3.21}) is
		
		\begin{subequations}\label{I-3.22}
			\begin{align}
			- \nu \Delta \phi^* - (L_e)^* \phi^* + \nabla p^* &= \bar{\lambda}_i \phi^* \text{ in } \Omega;\\
			\begin{picture}(0,0)
			\put(-100,0){ $\left\{\rule{0pt}{30pt}\right.$}\end{picture}
			\text{div } \phi^* &= 0 \text{ in } \Omega;\\
			\phi^* \rvert_{\partial \Omega} = 0;\quad \phi^* &= 0 \text{ in } \omega;
			\end{align}
		\end{subequations}
		
		\begin{equation}\label{I-3.23}
		\phi^* \in \calD(\calA^*); \quad (L_e)^* \phi^* = (y_e.\nabla)\phi^* + (\phi^*. \nabla)^* y_e,
		\end{equation}
		
		\noindent where $(f.\nabla)^*y_e$ is a $d$-vector whose $i$\textsuperscript{th} component is $\displaystyle \sum_{j=1}^{d} (D_i y_{e_j})f_j$ \cite[p 55]{BLT1:2006}.\\
		
		\noindent \textit{Step 2}. The critical point is now that the over-determined problem (\ref{I-3.22}) implies the following unique continuation result:		
		\begin{equation}\label{I-3.24}
		\phi^* = 0 \text{ in } \lo{q'}; \ \text{or by (\ref{I-3.20})} \ \phi_{i \ell_i}^* = \alpha_1\phi_{i1}^* + \alpha_2\phi_{i2}^* + \dots + \alpha_{\ell_i - 1}\phi_{i \ell_i - 1}^* \text{ in } \lo{q'},
		\end{equation}
		
		\noindent i.e. the set $\{ \phi_{i1}^*,\dots,\phi_{i \ell_i}^* \}$ in linearly dependent on $\lo{q'}$. But this is false, by the very selection of such eigenvectors, see (\ref{I-3.1}) and statement preceding it. Thus, the condition (\ref{I-3.24}) cannot hold.\\
		
		\noindent The required unique continuation result is established in \cite[Lemma 3.7]{BT:2004} or \cite{RT:2009}. The original proof is done in the Hilbert setting but we may invoke the same result because $\phi^*$ has more regularity and integrability than required since $\phi^*$ is an eigenfunction of $\calA^*$. Thus the \textit{claim} is established. In conclusion: it is possible to select, in infinitely many ways, interior functions $u_1,\dots,u_K \in \lso$ such that the algebraic full rank condition (\ref{I-3.18}) holds true for each $i = 1,\dots,M$.
	\end{proof}
	
	\section{Algebraic rank conditions for the dynamics $w_N$ in (\ref{I-2.8a}) in the general case} 
	\label{I-Sec-4}
	In the present section we dispense with the FDSA (\ref{I-3.2}). More precisely, we shall obtain Theorem \ref{I-2.1} without assuming the FDSA (\ref{I-3.2}). Thus now	
	\begin{equation*}
	\begin{aligned}
	W^u_N = &\text{ space of generalized eigenfunctions of $\calA_q (= \calA^u_N)$}\\
	&\text{ corresponding to its distinct unstable eigenvalues.} 
	\end{aligned}\tag{4.0} \label{I-4.0}
	\end{equation*}
	\noindent \textit{Warning}: In this section we shall denote by $\ell_i$ the geometric multiplicity of the eigenvalue $\lambda_i$ and by $N_i$ its algebraic multiplicity.\\
	
	\noindent \textit{Step 1:} To treat this computationally more complicated case we shall, essentially invoke the classical result on controllability of a finite-dimensional, time-invariant system $\{\BA,\BB\}, \BA:N \times N, \BB: N \times p$ where $\BA$ is given in Jordan form $J$. Let again $\lambda_1,\lambda_2,\dots,\lambda_M$ be the distinct eigenvalues of $\BA = J$. Let $\BA_i$ denote all the Jordan blocks associated with the eigenvalue $\lambda_i$; let $\ell_i$ be the number of Jordan blocks of $\BA$ (i.e the number of linearly independent eigenvectors associated with the eigenvalue $\lambda_i$). Let $\BA_{ij}$ be $j \textsuperscript{th}$ Jordan block in $\BA_i$ corresponding to a Jordan cycle of length $N_j^i$. That is:
	
	\begin{equation}\label{I-4.1}
	\BA = \text{diag} \{ \BA_1,\BA_2,\dots,\BA_M \};\quad \BA_i = \text{diag} \{ \BA_{i1},\BA_{i2},\dots,\BA_{i \ell_i} \}.
	\end{equation}
	
	Partition the matrix $\BB$ accordingly:
	
	\begin{equation}\label{I-4.2}
	\underset{(N \times N)}{\BA}
	= \begin{bmatrix}
	\BA_1 & & & \text{\huge0}\\
	& \BA_2 & & &\\
	& & \ddots & &\\
	\text{\huge0} & & & \BA_M 
	\end{bmatrix};
	\quad
	\underset{(N \times p)}{\BB}
	= \begin{bmatrix}
	\BB_1 \\ \BB_2 \\ \vdots \\ \BB_M
	\end{bmatrix}
	\end{equation}\\	
	\begin{equation}\label{I-4.3}
	\underset{(N_i \times N_i)}{\BA_i}
	= \begin{bmatrix}
	\BA_{i1} & & & \text{\huge0}\\
	& \BA_{i2} & & &\\
	& & \ddots & &\\
	\text{\huge0} & & & \BA_{i \ell_i} 
	\end{bmatrix};
	\quad
	\underset{(N_i \times p)}{\BB_i}
	= \begin{bmatrix}
	\BB_{i1} \\ \BB_{i2} \\ \vdots \\ \BB_{i \ell_i} 
	\end{bmatrix}
	\end{equation}\\	
	\begin{equation}\label{I-4.4}
	\underset{(N_j^i \times N_j^i)}{\BA_{ij}}
	= \begin{bmatrix}
	\lambda_i & 1 &  & \text{\huge0}\\
	& \lambda_i & 1 &\\
	& & \ddots & 1\\
	\text{\huge0} & & & \lambda_i 
	\end{bmatrix};
	\quad
	\underset{(N_j^i \times p)}{\BB_{ij}}
	= \begin{bmatrix}
	b_{1ij} \\ b_{2ij} \\ \vdots \\ b_{Lij}
	\end{bmatrix}.
	\end{equation}\\
	
	\noindent If $E_{\lambda_i}$ and $K_{\lambda_i}$ denote the eigenspace and the generalized eigenspace associated with the eigenvalue $\lambda_i, \ i = 1,\dots,M$, then dim $E_{\lambda_i} = \ell_i = \# $ of Jordan blocks in $\BA_i$, dim $K_{\lambda_i} = N_i, N_j^i = $ length of j\textsuperscript{th}-cycle associated with $\lambda_i; \ j = 1,\dots,\ell_i$. We have dim $\displaystyle W_N^u = N = \sum_{i = 1}^{M} N_i = \sum_{i = 1}^{M} \sum_{j = 1}^{\ell_i} N_j^i $. In (\ref{I-4.4}), the last row of $\BB_{ij}$ is denoted by $b_{Lij}$. The following result is classical \cite[ p 165]{CTC:1984}.
	
	\begin{thm}\cite[Theorem 3.1]{LT1:2015}\label{I-Thm-4.1}
		The pair $\{J,\BB\}$, $J:N \times N$, Jordan form, $\BB: N \times p$ is controllable if and only if, for each $i = 1,\dots,M$ (that is for each distinct eigenvalue) the rows of the $\ell_i \times p$ matrix constructed with all ``last" rows $b_{Li1},\dots,b_{Li \ell_i}$
		
		\begin{equation}\label{I-bli}
		\BB_i^L
		= \begin{bmatrix}
		\BB_{Li1} \\ \BB_{Li2} \\ \vdots \\ \BB_{Li \ell_i}
		\end{bmatrix}: \ell_i \times p
		\end{equation}
		
		\noindent are linearly independent (on the field of complex numbers). [A direct proof uses Hautus criterion for controllability \cite{CTC:1984}.]
	\end{thm}
	
	\noindent We next apply the above Theorem \ref{I-Thm-4.1} to the $w_N$-problem (\ref{I-2.8a}) and (\ref{I-3.5}). To this end, we select a Jordan basis $\beta_i$ for the operator $(\calA_N^u)_i$ on $W^u_{N,i}$ given by
	
	\noindent \textbf{Jordan Basis:} 
	\begin{subequations}\label{I-4.6}
		\begin{equation}\label{I-4.6.a}
		\beta_i = \Big\{ e^1_1(\lambda_i), e^1_2(\lambda_i),\dots,e^1_{N^i_1}(\lambda_i) \vdots e^2_1(\lambda_i), e^2_2(\lambda_i),\dots,e^2_{N^i_2}(\lambda_i) \vdots \dots \vdots e^{\ell_i}_1(\lambda_i), e^{\ell_i}_2(\lambda_i),\dots,e^{\ell_i}_{N^i_{\ell_i}}(\lambda_i) \Big\}.
		\end{equation}
		
		\noindent Here the first vectors of each cycle: $ e^1_1(\lambda_i), e^2_1(\lambda_i), \dots, e^{\ell_i}_1(\lambda_i)$ are eigenvectors of $(\calA_N^u)_i$ corresponding to the eigenvalue $\lambda_i$, while the remaining vectors in $\beta_i$ are corresponding generalized eigenvectors. Thus, in the notation (\ref{I-3.1}), we have
		
		\begin{equation}\label{I-4.6.b}
		\phi_{i1} = e^1_1(\lambda_i); \ \phi_{i2} = e^2_1(\lambda_i);  \dots ; \ \phi_{i \ell_i} = e^{\ell_i}_1(\lambda_i).
		\end{equation}
	\end{subequations}
	
	\noindent Next, we can choose a bi-orthogonal basis $\beta^*_i$ of $((\calA^u_N)^*)_i$ corresponding to its eigenvalue $\bar{\lambda_i}$ given by
	
	\noindent \textbf{Bi-orthogonal Basis:}
	\begin{subequations}\label{I-4.7}
		\begin{equation}\label{I-4.7a}
		\beta^*_i = \Big\{ {\Phi}^1_1(\bar{\lambda}_i), {\Phi}^1_2(\bar{\lambda}_i),\dots,{\Phi}^1_{N^i_1}(\bar{\lambda}_i) \vdots {\Phi}^2_1(\bar{\lambda}_i), {\Phi}^2_2(\bar{\lambda}_i),\dots,{\Phi}^2_{N^i_2}(\bar{\lambda}_i) \vdots \dots \vdots {\Phi}^{\ell_i}_1(\bar{\lambda}_i), {\Phi}^{\ell_i}_2(\bar{\lambda}_i),\dots,{\Phi}^{\ell_i}_{N^i_{\ell_i}}(\bar{\lambda}_i) \Big\}.
		\end{equation}
		
		\noindent Thus, in the notation (\ref{I-3.1}), we have
		
		\begin{equation}\label{I-4.7b}
		\phi^*_{i1} = {\Phi}^1_1(\bar{\lambda}_i); \ \phi^*_{i2} = {\Phi}^2_1(\bar{\lambda}_i);  \dots ; \ \phi^*_{i \ell_i} = {\Phi}^{\ell_i}_1(\bar{\lambda}_i).
		\end{equation}
	\end{subequations}
	
	In the bi-orthogonality relationship between the vectors in (\ref{I-4.6}) and those in (\ref{I-4.7}), the first eigenvector $e^1_1(\lambda_i)$ of the first cycle in $\beta_i$ is associated with the last generalized eigenvector $\Phi^1_{N^i_1}(\bar{\lambda}_i)$ of the first cycle in $\beta^*_i$; etc, the last generalized eigenvector $e^1_{N^i_1}(\lambda_i)$ of the first cycle in $\beta_i$ is associated with the first eigenvector $\Phi^1_1(\bar{\lambda}_i)$ of the first cycle in $\beta^*_i$; etc.	
	\begin{center}
		\begin{figure}[h]
			\begin{equation}
			\begin{tikzpicture}[<->,node distance=2cm, thick ]
			
			\node (11) {$e_1^1(\lambda_i)$};
			\node (12) [right of=11] {$e_2^1(\lambda_i)$};
			\node (13) [right of=12] {};
			\node (14) [right of=13] {$\cdots$};
			\node (15) [right of=14] {};
			\node (16) [right of=15] {$e_{N_1^i}^1(\lambda_i)$};
			\node (21) [below of=11] {$\Phi_1^1(\bar{\lambda}_i)$};
			\node (22) [right of=21] {$\Phi_2^1(\bar{\lambda}_i)$};
			\node (23) [right of=22] {};
			\node (24) [right of=23] {$\cdots$};
			\node (25) [right of=24] {};
			\node (26) [right of=25] {$\Phi_{N_1^i}^1(\bar{\lambda}_i)$};
			\node[text width=3cm] (all) at (13,-1) {, $\cdots$};	
			\path[every node/.style={font=\sffamily\small}]
			(11) edge node {} (26)
			(12) edge node {} (25)
			(21) edge node {} (16)
			(22) edge node {} (15);	
			\end{tikzpicture}
			\end{equation}
			\caption{Relation between the generalized eigenvectors of $\calA^u_N$ and $(\calA^u_N)^*$}
		\end{figure}
	\end{center}	
	\noindent Thus, if $f\in W_{N,i}^u$, the following expression holds true:
	\begin{eqnarray}\label{I-4.9} f = (f,\Phi_{N_{1}^i}^1(\bar{\lambda}_i))e_1^1(\lambda_i)+\cdots+ (f,\Phi_1^1(\bar{\lambda}_i))e^1_{N_1^i}(\lambda_i)&
	\nonumber \\[2mm]
	+ \cdots + (f,\Phi_{N^i_{\ell_{i}}}^{\ell_i}(\bar{\lambda}_i))
	e_1^{\ell_i}(\lambda_i)&+ \cdots + (f,\Phi_1^{\ell_i}(\bar{\lambda}_i))e^{\ell_i}_{N_{\ell_i}^i}(\lambda_i).
	\end{eqnarray}
	
	\noindent This expansion is the counterpart of $\ds \sum_{j=1}^{\ell_i}(w,\phi_{ij}^*)\phi_{ij} \in W^u_{N,i}$ in (\ref{I-3.4}) under the FDSA. Next, we apply (\ref{I-4.9}) to $f = P_NP(mu)$. More specifically, we shall write the vector representation of $P_NP(mu)$ with respect to the basis $\beta_i$ in (\ref{I-4.6.a}), and moreover, in line with Theorem \ref{I-Thm-4.1}, we shall explicitly note only the coordinates corresponding to the vectors $ e^1_{N^i_1}(\lambda_i), e^2_{N^i_2}(\lambda_i), \dots, e^{\ell_i}_{N^i_{\ell_i}}(\lambda_i)$, each being the last vector of each cycle in (\ref{I-4.6.a}).
	
	\begin{equation}\label{I-4.10}
	[P_N P(mu)]_{\beta_i} = 
	\begin{bmatrix}
	\times \times \times\\
	\Big(u,\Phi_1^1(\bar{\lambda}_i) \Big)_{\omega}\\
	\hdotsfor{1}\\
	\times \times \times\\
	\Big(u,\Phi_1^2(\bar{\lambda}_i) \Big)_{\omega}\\
	\hdotsfor{1}\\
	\times \times \times\\
	\Big(u,\Phi_1^{\ell_i}(\bar{\lambda}_i) \Big)_{\omega}\\
	\end{bmatrix}
	\begin{aligned}\\
	\xleftarrow{} \text{ last row of the $1^{st}$ cycle }\\
	\\ \\ 
	\xleftarrow{} \text{ last row of the $2^{nd}$ cycle }\\
	\\ \\ 
	\xleftarrow{} \text{ last row of the $\ell_i^{th}$ cycle }
	\end{aligned}
	\end{equation}
	
	The symbol $\times \times \times$ corresponds to terms which we do not care about. In fact, to exemplify, since $P^* P_N^* \Phi^1_1(\bar{\lambda}_i) = \Phi^1_1(\bar{\lambda}_i)$ see above (\ref{I-3.6})
	
	\begin{equation}\label{I-4.11}
	\Big(P_N P(mu), \Phi^1_1(\bar{\lambda}_i) \Big)_{\Omega} = \big(mu, \Phi^1_1(\bar{\lambda}_i) \big)_{\Omega} = \big(u, \Phi^1_1(\bar{\lambda}_i) \big)_{\omega}.
	\end{equation}
	
	This is the relevant counterpart of expansion $ \ds P_N P(mu) = \sum_{i,j = 1}^{M, \ell_i} (u, \Phi_{ij}^*)_{\omega}\Phi_{ij}$ in (\ref{I-3.6}) under the FDSA. Notice that (\ref{I-4.10}) involves only the eigenvectors $\Phi^1_1(\bar{\lambda}_i), \Phi^2_1(\bar{\lambda}_i), \dots, \Phi^{\ell_i}_1(\bar{\lambda}_i)$ of $(\calA^u_N)^*$ corresponding to the eigenvalue $\bar{\lambda}_i$. Next, recalling (\ref{I-2.10}): $ \ds u = \sum_{k=1}^{K} \mu_k(t) u_k$, we obtain that the corresponding counterpart of (\ref{I-3.15}) is
	
	\begin{equation}\label{I-4.12}
	U_i = 
	\begin{bmatrix}
	\times & \times & & \times\\
	(u_1,\Phi_1^1)_{\omega} & (u_2,\Phi_1^1)_{\omega} & \dots & (u_K,\Phi_1^1)_{\omega} \\
	\hdotsfor{4}\\
	\times & \times & & \times\\
	(u_1,\Phi_1^2)_{\omega} & (u_2,\Phi_1^2)_{\omega} & \dots & (u_K,\Phi_1^2)_{\omega} \\
	\hdotsfor{4}\\
	\times & \times & & \times\\
	(u_1,\Phi_1^{\ell_i})_{\omega} & (u_2,\Phi_1^{\ell_i})_{\omega} & \dots & (u_K,\Phi_1^{\ell_i})_{\omega} \\
	\end{bmatrix}
	\begin{aligned}
	\\
	\xleftarrow{} \text{ row } b_{Li1}(u)\\
	\\
	\xleftarrow{} \text{ row } b_{Li2}(u)\\
	\\
	\xleftarrow{} \text{ row } b_{Li \ell_i}(u)
	\end{aligned}
	\end{equation}
	
	\noindent Again, the relevant rows exhibited in (\ref{I-4.12}) correspond to the last rows of each Jordan sub-block $\{ \BA_{i1},\BA_{i2},\dots,\BA_{i \ell_i} \}$ in (\ref{I-4.3}). In (\ref{I-4.12}) we have displayed only such relevant rows: $b_{L i1},b_{L i2},\dots, b_{L i \ell_i}$ According to Theorem \ref{I-Thm-4.1}, the test for controllability as applied to the system (\ref{I-3.5}), i.e to the pair $\{ \calA^u_N,B \} , B $ = col $[B_1,B_2,\dots,B_M]$, is
	
	\begin{equation}\label{I-4.13}
	\text{ rank } 
	\begin{bmatrix}
	\vspace{2mm} \text{ row } b_{Li1} \text{ of } B_i\\
	\vspace{2mm} \text{ row } b_{Li2} \text{ of } B_i\\
	\vspace{2mm} \vdots\\
	\vspace{2mm} \text{ row } b_{Li \ell_i} \text{ of } B_i
	\end{bmatrix}
	= \text{ rank }
	\begin{bmatrix}
	\vspace{2mm}
	\big(u_1,\Phi_1^1 \big(\bar{\lambda}_i \big) \big)_{\omega} & \dots & \big(u_K,\Phi_1^1 \big(\bar{\lambda}_i \big) \big)_{\omega} \\
	\vspace{2mm}
	\big(u_1,\Phi_1^2 \big(\bar{\lambda}_i \big) \big)_{\omega} & \dots & \big(u_K,\Phi_1^2 \big(\bar{\lambda}_i \big) \big)_{\omega} \\
	\vspace{2mm} \vdots &  & \vdots\\ \vspace{2mm}
	\big(u_1,\Phi_1^{\ell_i} \big(\bar{\lambda}_i \big) \big)_{\omega}& \dots & \big(u_K,\Phi_1^{\ell_i} \big(\bar{\lambda}_i \big) \big)_{\omega} \\
	\end{bmatrix}
	= \ell_i
	\end{equation}
	
	\noindent $i = 1, \dots, M$. But this is exactly the test obtained in (\ref{I-3.18}) via the identification in (\ref{I-4.7b}):
	\begin{equation}\label{I-4.14}
	\phi^*_{i1} = \Phi^1_1(\bar{\lambda}_i), \phi^*_{i2} = \Phi^2_1(\bar{\lambda}_i), \dots, \phi^*_{i \ell_i} = \Phi^{\ell_i}_1(\bar{\lambda}_i)
	\end{equation}
	\noindent involving only eigenvectors, not generalized eigenvectors. Thus the remainder of the proof in section 3 past (\ref{I-3.18}) applies and shows Theorem \ref{I-Thm-3.1} without the FDSA. We have 
	
	\begin{thm}\label{I-Thm-4.2}
	With reference to $U_i$ in (\ref{I-4.12}), it is possible to select interior vectors  $u_1,\dots,u_K \in W^u_N \subset \lso,\ K = \sup \ \{\ell_i:i = 1,\dots,M \}$, such that the algebraic conditions (\ref{I-4.13}) hold true, $i = 1,\dots,M$.
	\end{thm}
	
	\noindent We close this section by writing down the counterpart of the expansion (\ref{I-3.10}) for the $w_N$-dynamics in terms of the basis $\ds \beta = \beta_1 \cup \beta_2 \cup \dots \cup \beta_M$, see (\ref{I-3.11}), (\ref{I-4.6.a}), (\ref{I-4.9}) of the generalized eigenvectors in the present general case.
	\begin{eqnarray}\label{I-4.15}
	\text{on} && W_N^u : w'_N - \calA^u_N w_N \nonumber \\
	&&= \sum_{i=1}^{M} \Bigg\{ \sum_{k=1}^{K} \bigg[(u_k,\Phi_{N^i_1}^1 (\bar{\lambda}_i))_{\omega} \mu_k(t) \bigg] e^1_1(\lambda_i) + \dots + \dots + \sum_{k=1}^{K} \bigg[(u_k,\Phi_1^1 (\bar{\lambda}_i))_{\omega} \mu_k(t) \bigg] e^1_{N^i_1}(\lambda_i) \nonumber\\
	&& + \dots + \dots \nonumber \\
	&&+ \sum_{k=1}^{K} \bigg[(u_k,\Phi_{N^i_{\ell_i}}^{\ell_i} (\bar{\lambda}_i))_{\omega} \mu_k(t) \bigg] e^{\ell_i}_1(\lambda_i) + \dots + \dots + \sum_{k=1}^{K} \bigg[(u_k,\Phi^{\ell_i}_1 (\bar{\lambda}_i))_{\omega} \mu_k(t) \bigg] e_{N^i_{\ell_i}}^{\ell_i}(\lambda_i) \Bigg \}. \nonumber\\
	\end{eqnarray} 	
	\section{Proof of Theorem \ref{I-2.1}:  arbitrary decay rate of the $w_N$-dynamics (\ref{I-3.5}) or (\ref{I-4.15}) (or (\ref{I-3.13}) under the FDSA) by a suitable finite-dimensional interior localized feedback control u}\label{I-Sec-5}
	
	We are now in a position to obtain Theorem \ref{I-2.1}, which we restate for convenience. Let $1 < q < \infty.$
	
	\begin{thm}\label{I-Thm-5.1}
		Let $\lambda_1,\dots,\lambda_M$ be the unstable distinct eigenvalues of $\calA$ and let $\omega$ be an arbitrarily small open portion of the interior with smooth boundary $\partial \omega$. By virtue of Theorem \ref{I-4.2}, pick interior vectors $[u_1,\dots,u_K]$ in $W^u_N \subset \lso$ such that the rank conditions (\ref{I-4.13}) hold true, with $K = \sup \ \{\ell_i: \ i = 1,\dots,M \}$ (respectively, Theorem \ref{I-Thm-3.1} and the (same) rank conditions (\ref{I-3.18}) under FDSA).\\
		
		\noindent Then: Given $\gamma > 0$ arbitrarily large, there exists a $K$-dimensional interior controller $u = u_N$ acting on $\omega$, of the form given by (\ref{I-3.8}), with the vectors $u_k$ given by Theorem \ref{I-Thm-4.2} via the rank conditions (\ref{I-4.13}), such that, once inserted in (\ref{I-4.15}) yield the estimate
		\begin{subequations}\label{I-5.1}
			\begin{equation}\label{I-5.1a}
			\norm{w_N(t)}_{\lso} + \norm{u_N(t)}_{\ls(\omega)} \leq C_{\gamma} e^{- \gamma t} \norm{P_N w_0}_{\lso}, \ t \geq 0,
			\end{equation}
			\noindent where the $\lso$-norm in (\ref{I-5.1a}) may be replaced by the $\ds \lqaq$- norm , $1 < p, q < \infty$; in particular the $\ds \Bto$-norm, $\ds 1 < q < \infty,\ 1 < p < \frac{2q}{2q-1}$:	
			\begin{equation}\label{I-5.1b}
			\norm{w_N(t)}_{\Bto} + \norm{u_N(t)}_{\Bto} \leq C_{\gamma} e^{- \gamma t} \norm{P_N w_0}_{\Bto}, \ t \geq 0.
			\end{equation}
			
		\end{subequations}
		\noindent Here, $w_N$ is the solution of (\ref{I-4.15}) (respectively (\ref{I-3.10}) under the FDSA) , i.e., (\ref{I-3.5}) corresponding to the control $u = u_N$ in (\ref{I-3.8}). Moreover, such controller $u = u_N$can be chosen in feedback form: that is, with reference to the explicit expression (\ref{I-3.8}) for $u$, of the form $\mu_k(t) = (w_N(t),p_k)_{\omega}$ for suitable vectors $p_k \in (W^u_N)^* \subset \lo{q'}$ depending on $\gamma$. In conclusion, $w_N$ in (\ref{I-5.1}) is the solution of the equation on $W^u_N$ (see (\ref{I-3.5})) specialized as (\ref{I-4.15})
		\begin{equation}\label{I-5.2}
		w'_N - \calA^u_N w_N = P_N P \Bigg ( m \Bigg( \sum_{k=1}^{K} (w_N(t),p_k)_{\omega} u_k \Bigg) \Bigg), \ u_k \in W^u_N \subset \lso, \ p_k \in (W^u_N)^* \subset \lo{q'}
		\end{equation}		
		\noindent rewritten as	
		\begin{equation}\label{I-5.3}
		w'_N = \bar{A}^u w_N, \quad w_N(t) = e^{\bar{A}^u t}P_N w_0, \quad w_N(0) = P_N w_0.
		\end{equation}
	\end{thm}
	
	\begin{proof}
		\noindent \underline{Step 1}: Following \cite{LT1:2015} the proof consists in testing controllability of the linear, finite-dimensional system (\ref{I-3.5}), in short, the pair 
		\begin{equation}\label{I-5.4}
		\{J,B\}, \ B = U: \ N \times K, K = \text{sup } \{\ell_i; i = 1.\dots, M \}
		\end{equation}
		$U = [U_1, \dots, U_M]^{\text{tr}}$, $U_i$ given by (\ref{I-4.12}) (or by (\ref{I-3.15}) under FDSA). $J$ is the Jordan form of $\calA^u_N$ with respect to the Jordan basis $\beta = \beta_1 \cup \dots \cup \beta_M, \ \beta_i$ being given by (\ref{I-4.6.a}). But the rank conditions (\ref{I-4.13}) precisely asserts such controllability property of the pair $\{\calA^u_N = J, B\}$, in light of Theorem \ref{I-4.1}.\\
		
		\noindent \underline{Step 2}: Having established the controllability criterion for the pair $\{\calA^u_N = J, B\}$ then by the well-known Popov's criterion in finite-dimensional theory, there exists a real feedback matrix $Q = K \times N$, such that the spectrum of the matrix $(J + BQ) = (J + UQ)$ may be arbitrarily preassigned; in particular, to lie in the left half-plane $\{ \lambda: \text{Re } \lambda < - \gamma < - \text{Re } \lambda_{N+1} \}$, as desired. The resulting closed-loop system
		
		\begin{equation}\label{I-5.5}
		(\hat{w}'_N) - J \hat{w}_N = U u_N,
		\end{equation}
		
		\noindent is obtained with $\mathbb{C}^N$-vector $u_N = Q \hat{w}_N$, Q being the $K \times N$ matrix with row vectors $[\hat{p}_1,\dots,\hat{p}_K],\ \mu^k_N = (\hat{w}_N,\hat{p}_k)$ in the ${\mathbb{C}}^N$-inner product and hence decays with exponential rate
		
		\begin{equation}\label{I-5.6}
		\abs{\hat{w}_N(t)}_{\BC^N} \leq C_{\gamma} e^{-\gamma t} \abs{\hat{w}_N(0)}_{\BC^N}, \ t \geq 0.
		\end{equation}
		
		\noindent But the $N$-dimensional vector $\ds w_N \in W^u_N \subset \lso$ is represented by the $\BC^N$-vector $\ds \hat{w}_N = [w_N]_{\beta}$, where in the general case of Section \ref{I-Sec-4}, $\beta$ is a Jordan basis of generalized eigenfunctions of $\ds \calA_q(=\calA_N^u)$ corresponding to its $M$ distinct unstable eigenvalues. Such basis is given by $\ds \beta = \beta_1 \cup \beta_2 \cup \cdots \cup \beta_M$, where a representative $\beta_i$ is given in (\ref{I-4.6.a}). The whole basis can be read off from (\ref{I-4.15}). In the special case of Section \ref{I-Sec-3} where the FDSA holds, the basis $\beta$ in $W^u_N$ is given by the eigenfunctions of the $\calA_N^u$ corresponding to its $M$ distinct eigenvalues, see (\ref{I-3.11}). But such eigenfunctions/generalized eigenfunctions are in $\calD(\calA_q)$, hence smooth. Thus, the exponential decay in (\ref{I-5.6}) of the coordinate vector $\hat{w}_N$ in $\BC^N$ translates in same exponential decay of the vector $w_N(t) \in W^u_N$ not only in the $\lso$-norm but also in the $\calD(\calA_q) = \calD(A_q)$-norm, hence in the $\ds \lqaq$-norm, in particular in the $\ds \Bto$-norm. See also Remark \ref{I-Rmk-3.1}. Thus, returning from ${\mathbb{C}}^N \times {\mathbb{C}}^N$ back to $W^u_N \times (W^u_N)^*$, there exist suitable $p_1,\dots,p_K \in (W^u_N)^* \subset \lo{q'}$, such that $\mu^k_N = (w_k,p_k)$, whereby the closed-loop system (\ref{I-5.2}) corresponds precisely to (\ref{I-4.15}) via $P_N P (mu)$ written in terms of the Jordan basis of eigenvectors $\beta$ in (\ref{I-4.6.a}).
		
		\noindent Thus not only we obtain in view of (\ref{I-5.2}), (\ref{I-5.3}) and (\ref{I-5.6}) 
		\begin{equation}\label{I-5.7}
		\norm{w_N(t)}_{\lso} = \norm{e^{\bar{A}^u t}P_N w_0}_{\lso} \leq C_{\gamma} e^{-\gamma t} \norm{P_N w_0}_{\lso}, \ t \geq 0,
		\end{equation}
		\noindent but also, say $\ds 1 < q < \infty, 1 < p < \frac{2q}{2q - 1}$
		\begin{equation}\label{I-5.8}
		\norm{w_N(t)}_{\Bto} = \norm{e^{\bar{A}^u t}P_N w_0}_{\Bto} \leq C_{\gamma} e^{-\gamma t} \norm{P_N w_0}_{\Bto}, \ t \geq 0.
		\end{equation}
		\noindent Hence with $u_N = Q w_N$, we obtain not only
		\begin{align}
		\norm{w_N(t)}_{\lso} + \norm{u_N(t)}_{\ls(\omega)} &= \norm{w_N(t)}_{\lso} + \norm{Qw_N(t)}_{\lso} \label{I-5.9}\\
		&\leq \big(\abs{Q} + 1 \big) \norm{e^{\bar{A}^u t}P_N w_0}_{\lso} \leq C_{\gamma} e^{- \gamma t} \norm{P_N w_0}_{\lso} \label{I-5.10}
		\end{align}
		\noindent but also, say
		\begin{equation}\label{I-5.11}
		\norm{w_N(t)}_{\Bto} + \norm{u_N(t)}_{\Bto} \leq C_{\gamma} e^{-\gamma t}\norm{P_N w_0}_{\Bto}, \ t \geq 0.
		\end{equation}
	\end{proof}
	
	\begin{rmk}\label{I-Rmk-5.1}
		Under the FDSA, checking controllability of the system (\ref{I-3.13}) is easier. To this end, we can pursue, as usual, two strategies.\\
		
		\noindent A first strategy invokes the well-known Kalman controllability criterion by constructing the $N \times KN$ Kalman controllability matrix
		
		\begin{equation}\label{I-5.12}
		\mathcal{K} = [B, \Lambda B,\Lambda^2 B,\dots,\Lambda^{N-1} B] =
		\begin{bmatrix}
		B_1 & J_1 B_1 & \dots & J_1^{N-1}B_1 \\
		B_2 & J_2 B_2 & \dots & J_2^{N-1}B_2 \\
		\hdotsfor{4}\\
		B_M & J_M B_M & \dots & J_M^{N-1}B_M 
		\end{bmatrix},
		\end{equation}\\
		\begin{equation}\label{I-5.13}
		B = \text{col }[B_1,B_2,\dots,B_M], \quad B_i = U_i: \ell_i \times \ell_i
		\end{equation}
		
		\noindent of size $N \times KN, \ N = \text{dim } W^u_N, \ J_i = \lambda_i I_i: \ell_i \times \ell_i, \ B_i = U_i: \ell_i \times \ell_i$, and requiring that it be full rank.
		
		\begin{equation}\label{I-5.14}
		\text{rank } \mathcal{K} = \text{full} = N.
		\end{equation}
		
		\noindent In view of generalized Vandermond determinants, we then have 
		\begin{equation}\label{I-5.15}
		\text{rank } \calK = N \quad \text{if and only if rank } U_i = \ell_i \text{ (full) } i = 1,\dots,M, 
		\end{equation}
		\noindent precisely as guaranteed by (\ref{I-3.18}). A second strategy invokes the Hautus controllability criterion:
		\begin{equation}\label{I-5.16}
		rank [\Lambda - \lambda_i I, B] = rank [\Lambda - \lambda_i I, U] = N \ (\text{full})
		\end{equation} 
		for all unstable eigenvalues $\lambda_i, 1, \dots, M$, yielding again the condition that rank $[U_i] = \ell_i, 1, \dots, M$, as generated by (\ref{I-3.18}).
	\end{rmk}
	
	\section{Proof of Theorem \ref{I-Thm-2.2}: Feedback stabilization of the original linearized $w$-Oseen system (\ref{I-1.13}) by a finite dimensional feedback controller}\label{I-Sec-6}
	
	\noindent The main result on the feedback stabilization of the linearized $w$-system (\ref{I-1.13}) = (\ref{I-2.1}) by a finite dimensional controller is Theorem \ref{I-Thm-2.2}, here reformulated in part for convenience in the context of the development of the present proof. Throughout this section $1 < q < \infty$.
	
	\begin{thm}\label{I-Thm-6.1}
		Let the Oseen operator $\calA$ have N possibly repeated unstable eigenvalues $\{ \lambda_j \}_{j = 1}^N$ of which $M$ are distinct. Let $\varepsilon > 0$ and set $\gamma_0 = |Re \ \lambda_{N+1}| - \varepsilon$. Consider the setting of Theorem \ref{I-Thm-5.1} so that, in particular, the feedback finite-dimensional control $u = u_N$ is given by $\ds u = u_N = \sum_{k = 1}^{K} (w_N(t),p_k)u_k$ and satisfies estimates (\ref{I-5.1}) with $\gamma > 0$ arbitrarily large, for vectors $p_1,\dots,p_k \in (W^u_N)^* \subset \lo{q'}$ and vectors $u_1,\dots,u_k \in W^u_N \subset \lso$ given by Theorem \ref{I-Thm-5.1}. Thus, the linearized problem (\ref{I-2.1}) specializes to (\ref{I-2.16})		
		\begin{equation}\label{I-6.1}
		\frac{dw}{dt} = \calA w + P \Bigg ( m \Bigg( \sum_{k=1}^{K} (w_N(t),p_k)_{\omega} u_k \Bigg) \Bigg) \equiv \BA_{_F} w.
		\end{equation}
		\noindent Here $\ds \BA_{_F} = \BA_{_{F,q}}$ is the generator of a s.c. analytic semigroup on either the space $\ds \lso, \ 1 < q < \infty$, or on the space $\ds \lqaq, \ 1 < p,q < \infty,$ in particular on the space $\ds \Bto, 1 < q, 1 < p < \rfrac{2q}{2q-1}$. Moreover, such dynamics $w$ (equivalently, generator $\ds \BA_{_F}$) in (\ref{I-6.1}) is uniformly stable in each of these spaces, say
		\begin{equation}\label{I-6.2}
		\norm{e^{\BA_{_F}t}w_0}_{\lso} = \norm{w(t,w_0)}_{\lso} \leq C_{\gamma_0} e^{-\gamma_0 t} \norm{w_0}_{\lso}, \quad t \geq 0.
		\end{equation}
		\noindent or say
		\begin{equation}\label{I-6.3}
		\norm{e^{\BA_{_F}t}w_0}_{\Bto} = \norm{w(t,w_0)}_{\Bto} \leq C_{\gamma_0} e^{-\gamma_0 t} \norm{w_0}_{\Bto}, \quad t \geq 0.
		\end{equation}
	\end{thm} 
	
	\begin{proof}
		\noindent \underline{Step 1}: According to Theorem \ref{I-Thm-5.1}, the finite-dimensional system $w_N$ in (\ref{I-2.8a}) = (\ref{I-3.5}) is uniformly stabilized by the finite dimensional feedback controller $u = u_N$ given in the RHS of (\ref{I-5.2}) = RHS of (\ref{I-6.1}) with an arbitrary preassigned decay rate $\gamma > 0$, as given, either in the $\lso$-norm, or in the $\ds \lqaq$-norm in (\ref{I-5.1a}), or in particular, in the $\Bto$-norm as in (\ref{I-5.1b}).\\
		
		\noindent \underline{Step 2}: Next, we examine the impact of such constructive feedback control  $u_N$ on the $\zeta_N$-dynamics (\ref{I-2.8b}), whose explicit solution can be given by a variation of parameter formula,
		
		\begin{equation}\label{I-6.4}
		\zeta_N(t) = e^{\calA^s_N t} \zeta(0) + \int_{0}^{t} e^{\calA^s_N (t-r)}(I-P_N)P(m u_N(r))dr. 
		\end{equation}
		
		\noindent in the notation $\calA_N^s = (I - P_N) \calA, \ \calA = \calA_q$, of (\ref{I-2.5}). We now recall from Section \hyperref[I-Sec-1.10d]{1.10 (d)} that the Oseen operator $\calA_q$ generates a s.c. analytic semigroup not only on $\lso$ but also on $\ds \lqaq$, in  particular on $\Bto$. Hence the feedback operator $\ds \BA_{_F} = \BA_{_{F,q}}$ similarly generates a s.c. analytic semigroup on these spaces, being a bounded perturbation of the Oseen operator $\ds \calA = \calA_q$. So we can estimate (\ref{I-6.4}) in the norm of either of these spaces. Furthermore, the (point) spectrum of the generator $\calA_N^s$ on $W^s_N$ satisfies $\sup \{Re \ \sigma(\calA^s_N)\} < - \abs{\lambda_{N+1}} < - \gamma_0$ by assumption. We shall carry our the supplemental computations explicitly in the space $\Bto$ for the case of greatest interest in the nonlinear analysis of sections \ref{I-Sec-8}, \ref{I-Sec-9}. In the norm of $\Bto$, we obtain from (\ref{I-6.4}) since the operators $(I-P_N),P$ are bounded
		
		\begin{align}
		\norm{\zeta_N(t)} &\leq \norm{ e^{\calA_N^s t} \zeta(0)} + C \int_{0}^{t} \norm{e^{\calA_N^s (t -\tau)}} \norm{u_N(\tau)} d \tau \label{I-6.5}\\
		\norm{\zeta_N(t)}_{\Bto} &\leq C e^{- \gamma_0 t} \norm{\zeta(0)}_{\Bto} + C \int_{0}^{t} e^{ -\gamma_0 (t -r)} e^{-\gamma r}dr \norm{P_N w_0}_{\Bto}. \label{I-6.6}
		\end{align}
		
		\noindent recalling estimate (\ref{I-2.13}) or (\ref{I-5.11}) for $\ds \norm{u_N}$ in the $\ds \Bto$-norm. Since we may choose $\gamma > \gamma_0$ by Theorem \ref{I-Thm-2.1} (or Theorem \ref{I-Thm-5.1}), we then obtain 
		
		\begin{align}
		\norm{\zeta_N(t)}_{\Bto} &\leq C \left[ e^{-\gamma_0 t} + e^{-\gamma_0 t} \frac{1 - e^{-(\gamma - \gamma_0) t}}{\gamma - \gamma_0}\right] \norm{w_0}_{\Bto} \label{I-6.7}\\
		&\leq C e^{-\gamma_0 t} \norm{w_0}_{\Bto}, \quad \forall t>0. \label{I-6.8}
		\end{align}
		\noindent Then, estimate (\ref{I-6.8}) for $\zeta_N(t)$ along with estimate (\ref{I-2.13}) = (\ref{I-5.11}) for $w_N(t)$ with $\gamma > \gamma_0$ yields the desired estimate (\ref{I-6.3}) for $w = w_N + \zeta_N$ in the $\ds \Bto$-norm:
		\begin{align}
		\norm{w(t)}_{\Bto} &\leq \norm{\zeta_N(t)}_{\Bto} + \norm{w_N(t)}_{\Bto} \\
		&\leq \big[ \widetilde{C}_{\gamma_0} e^{-\gamma_0 t } + C_{\gamma} e^{-\gamma t } \big] \norm{w_0}_{\Bto}\\
		&\leq C_{\gamma_0} e^{-\gamma_0 t } \norm{w_0}_{\Bto}
		\end{align}
		\noindent and (\ref{I-6.3}) is proved. Similar computations from (\ref{I-6.4}) to (\ref{I-6.8}) apply in the $\ds \lso$-norm for $\zeta_N(t)$, as the Oseen operator generates a s.c. analytic semigroup on $\lso$ from Section \hyperref[I-Sec-1.10d]{1.10 (d)}. This, coupled with estimate (\ref{I-2.12}) for $w_N(t)$, yields estimate (\ref{I-6.2}) for the $\ds w = w_N + \zeta_N$ with $\lso$-norm. Theorem \ref{I-Thm-6.1} is established.
	\end{proof}
	
	\begin{rmk}
		Computations such as those in \cite[p 1473]{BT:2004} using the analyticity of the Oseen semigroup $e^{\calA_q t}$ show the alternative estimates (\ref{I-2.18a}-b) of Theorem \ref{I-Thm-2.2}.
	\end{rmk}
	
	\section{Maximal $L^p$ regularity on $\lso$ and for $T = \infty$ of the s.c. analytic semigroup $e^{\BA_{F,q}t}$ yielding uniform decay of the linearized $w$-problem (\ref{I-2.1}), once specialized as in (\ref{I-6.1}) of Theorem \ref{I-Thm-2.2} = Theorem \ref{I-Thm-6.1}.} \label{I-Sec-7}
	
	\noindent In this section, we return to the $w$-feedback problem (\ref{I-6.1}), $\ds w_t = \BA_{F,q} w$, where $p_k, u_k$ are the vectors claimed and constructed in Theorem \ref{I-Thm-2.1}, or Theorem \ref{I-Thm-2.2} (Theorem \ref{I-Thm-6.1}) and Remark \ref{I-Rmk-3.1}. As stated in Theorem \ref{I-Thm-6.1}, problem (\ref{I-6.1}) defines a s.c. analytic, uniformly stable semigroup $e^{\BA_{F,q}t}$ as in (\ref{I-6.2}):
	
	\begin{equation}\label{I-7.1}
	\norm{e^{\BA_{F,q}t}}_{\calL(\cdot)} \leq M_{\gamma_0} e ^{-\gamma_0 t}, \quad t \geq 0
	\end{equation}
	
	\noindent where $(\cdot)$ denotes the space $\lso$ or else $\lqaq$, in particular $\Bto$. Define the ``good'' bounded operator 
	\begin{equation}\label{I-7.2}
	Gw = m \Bigg( \sum_{k=1}^{K} (P_N w,p_k)_{\omega} u_k \Bigg), \ u_k \in W^u_N \subset \lso, \ p_k \in (W^u_N)^* \subset \lo{q'},
	\end{equation}
	\noindent By Theorem \ref{I-1.6}, the Oseen operator $\calA_q$ enjoys maximal $L^p$ regularity on $\lso$ up to $T < \infty$, see (\ref{I-1.48}) as well as (\ref{I-1.50}), (\ref{I-1.52}). Then the same property holds true up to $T < \infty$ for $\ds \BA_{_{F,q}} = \calA_q + G$, as G is a bounded operator \cite{Dore:2000}, \cite{KW:2004}, \cite{We:2001}. We now seek to establish maximal $L^p$ regularity up to $T = \infty$ of $\ds \BA_{_{F,q}}$, i.e. of the following problem  
	\begin{subequations}\label{I-7.3}
		\begin{align}
		w_t - \Delta w + L_e(w) + \nabla \pi &= Gw + F &\text{ in } (0, T] \times \Omega \equiv Q \label{I-7.3a}\\		
		div \ w &\equiv 0 &\text{ in } Q\\
		\begin{picture}(0,0)
		\put(-110,5){ $\left\{\rule{0pt}{35pt}\right.$}\end{picture}
		\left. w \right \rvert_{\Sigma} &\equiv 0 &\text{ in } (0, T] \times \Gamma \equiv \Sigma\\
		\left. w \right \rvert_{t = 0} &= w_0 &\text{ in } \Omega,
		\end{align}
	\end{subequations}
	
	\noindent $L_e$ defined in (\ref{I-1.39}) rewritten abstractly, upon application of the Helmholtz projection $P_q$ to (\ref{I-7.3a}) and $\Fs = P_q F$, as 
	\begin{align}
	w_t = \BA_{F,q} w + P_q F &= \calA_q w + P_q G w + P_q F \label{I-7.4}\\
	&= - \nu A_q w - A_{o,q}w + P_q G w + P_q F. \label{I-7.5}
	\end{align}
	
	\noindent Wlog, we take $\nu = 1$ henceforth. Here we have appended a subscript ``$q$'' to the generator $\BA_{_F}$ defined in (\ref{I-6.1}) which we rewrite as $\BA_{F,q}$. With $F_{\sigma} = P_q F$ its solution on $\lso$ is
	\begin{align}
	w(t) &= e^{\BA_{F,q} t} w_0 + \int_{0}^{t} e^{\BA_{F,q} (t - \tau)} \Fs(\tau) d \tau \label{I-7.6}\\
	&= e^{-A_q t} w_0 + \int_{0}^{t} e^{-A_q (t - \tau)} (P_qG - A_{o,q})w(\tau) d \tau + \int_{0}^{t} e^{-A_q (t - \tau)} \Fs(\tau) d \tau. \label{I-7.7}
	\end{align}
	\noindent As the present section is preparatory for the subsequent sections \ref{I-Sec-8} and \ref{I-Sec-9}, the case of greatest interest here is then for $\ds w_0 \in \Bto$, i.e. $\ds 1 < q, 1 < p < \rfrac{2q}{2q-1}$. Nevertheless we shall treat the general case $1 < p,q < \infty$.
	\begin{thm}\label{I-Thm-7.1}
		As in (\ref{I-1.43}) of Theorem \ref{I-Thm-1.6}, but now with $T = \infty$, assume 
		\begin{equation}\label{I-7.8}
		\Fs \in L^p(0,\infty;\lso), \quad w_0 \in \Big( \lso, \calD(A_q)\Big)_{1-\rfrac{1}{p},p}.
		\end{equation}
		Then there exists a unique solution of problem (\ref{I-7.3}) = (\ref{I-7.4}) = (\ref{I-7.5}).	
		\begin{subequations}\label{I-7.9}
			\begin{align}
			w \in \xipqs &= L^p(0, \infty; \calD(A_q)) \cap W^{1,p}(0, \infty; \lso) \text{, equivalently } \label{I-7.9a}\\		
			\begin{picture}(0,0)
			\put(-28,8){ $\left\{\rule{0pt}{20pt}\right.$}\end{picture}
			w \in \xipq &= L^p(0, \infty; W^{2,q}(\Omega)) \cap W^{1,p}(0, \infty; \lso) \hookrightarrow C \big(0,\infty; \Bso \big) \label{I-7.9b}
			\end{align}
		\end{subequations}
		
		\noindent (recall \cite[Theorem 4.10.2, p180 in BUC for $T = \infty$]{HA:1995} already noted in (\ref{I-1.30})) continuously on the data: there exist constants $C_0, C_1$ such that		
		\begin{subequations}
		\begin{align}
		C_0 \norm{w}_{C \big(0,\infty; \Bso \big)} 
		&\leq \norm{w}_{\xipqs} + \norm{\pi}_{\yipq} \nonumber\\
		&\equiv \norm{w'}_{L^p(0,\infty;L^q(\Omega))} +  \norm{A_q w}_{L^p(0,\infty;L^q(\Omega))} + \norm{\pi}_{\yipq} \label{I-7.10a}\\
		&\leq C_1 \bigg \{ \norm{\Fs}_{L^p(0,\infty;\lso)}  + \norm{w_0}_{\lqaq} \bigg \}. \label{I-7.10b}
		\end{align} 
		\end{subequations}
		
		\noindent Thus for $\ds 1 < q, 1 < p < \frac{2q}{2q-1}$, then the I.C. $w_0$ is in $\ds \Bto$. Equivalently,
		\begin{enumerate}[(i)]
			\item The map 
			\begin{equation}
			\begin{aligned}
			\Fs &\longrightarrow \int_{0}^{t} e^{\BA_{F,q}(t-\tau)}\Fs(\tau) d\tau \ : \text{continuous} \label{I-7.11}\\
			L^p(0,\infty;\lso) &\longrightarrow L^p(0,\infty;\calD(\BA_{F,q}) = \calD(\calA_q) = \calD(A_q)),
			\end{aligned}			
			\end{equation}
			whereby then automatically 
			\begin{equation}
			L^p(0,\infty;\lso) \longrightarrow W^{1,p}(0,\infty;\lso)
			\end{equation}
			and ultimately
			\begin{equation}\label{I-7.13}
			L^p(0,\infty;\lso) \longrightarrow \xipqs = L^p(0, \infty; \calD(\BA_{_{F,q}})) \cap W^{1,p}(0, \infty; \lso)
			\end{equation}
			\item The s.c. analytic semigroup $e^{\BA_{F,q}t}$ on the space $\lqaq, 1 < p < \infty $, as asserted in Theorem \ref{I-6.1}, in particular on the space $\ds \Bto, 1 < q, 1 < p < \frac{2q}{2q-1}$, satisfies
			\begin{equation}
			\begin{aligned} \label{I-7.14}
			e^{\BA_{F,q} t}: \ \text{continuous} \quad \lqaq &\longrightarrow \xipqs \quad (\text{equivalently } \longrightarrow \xipq)\\
			\text{in particular} \quad \Bto &\longrightarrow \xipqs \quad (\text{equivalently } \longrightarrow \xipq).
			\end{aligned}		
			\end{equation}
		\end{enumerate}
	\end{thm}
	
	\begin{proof}
		\underline{Part i}\\
		\noindent \textbf{Orientation}
		The proof is a suitable modification of the proof of Theorem \ref{I-Thm-1.6}, that is, of the maximal regularity of the Oseen operator $\calA_q$ on $\lso$, given in Appendix \hyperref[I-app-B]{B}. Namely, Step 1 = (\ref{I-B.3}) of that proof now exploits the uniform stability of $e^{\BA_{F,q}t}$ in (\ref{I-6.2})=(\ref{I-7.1}) which was not available for the Oseen semigroup $e^{\calA_qt}$ in Appendix \hyperref[I-app-B]{B}. Hence the convolution argument in (\ref{I-B.8}) may now be applied up to $T = \infty$, see below (\ref{I-7.16}). Next, Step 2 of the proof in (\ref{I-B.13})-(\ref{I-B.20}) in Appendix \hyperref[I-app-B]{B} applies also in the present proof, for $T \leq \infty$, to include $T = \infty$, as the term $-A_{o,q}$ in (\ref{I-B.13}) is replaced in the present proof by $(P_q G - A_{o,q})$, with $P_q G$ bounded.\\
		
		\noindent \textit{Step 1}: With reference to (\ref{I-7.6}) with $w_0 = 0$, we first establish the inequality
		
		\begin{equation}\label{I-7.15}
		\int_{0}^{\infty} \norm{w(t)}_{\lso}^p dt \leq C \int_{0}^{\infty} \norm{\Fs(t)}_{\lso}^p dt.
		\end{equation}
		
		\noindent Indeed, from (\ref{I-7.6}), in the $\lso$-norm, recalling (\ref{I-7.1})
		
		\begin{align}\label{I-7.16}
		\norm{w(t)} &\leq \int_{0}^{t} \norm{e^{\BA_{F,q}(t - \tau)}} \norm{\Fs(\tau)} d \tau \nonumber \\
		&\leq M_{\gamma_0} \int_{0}^{t} e^{-\gamma_0 (t - \tau)} \norm{\Fs(\tau)} d \tau \in L^p(0, \infty)
		\end{align}
		\noindent being the convolution of a $L^1(0, \infty)$-function with an $L^p(0, \infty)$-function (Young's Theorem) \cite{CS:1979}. Then (\ref{I-7.15}) is proved.\\
		
		\noindent \textit{Step 2}: Again for $w_0 = 0$ we obtain from (\ref{I-7.7})
		\begin{equation}
		A_q w(t) = A_q\int_{0}^{t} e^{-A_q (t - \tau)} (P_qG - A_{o,q})w(\tau) d \tau + A_q \int_{0}^{t} e^{-A_q (t - \tau)} \Fs(\tau) d \tau \label{I-7.17}
		\end{equation}
		\noindent We shall establish the following inequality
		\begin{equation}\label{I-7.18}
		\int_{0}^{\infty} \norm{A_q w(t)}_{\lso}^p dt \leq C \int_{0}^{\infty} \norm{w(t)}_{\lso}^p dt + C \int_{0}^{\infty} \norm{\Fs(t)}_{\lso}^p dt.
		\end{equation}
		\noindent (Compare with (\ref{I-B.4}), which holds true for any $T \leq \infty$, including $T = \infty$). In fact, to this end, as in that proof, using the maximal regularity up to $T = \infty$ of the Stokes semigroup, as well as (\ref{I-7.8}) for $\Fs$, we estimate from (\ref{I-7.17})	
		\begin{align}
		\norm{A_q w}_{L^P(0, \infty, \lso)} &\leq C \norm{\Fs}_{L^P(0, \infty, \lso)} + C \norm{[G - A_{o,q}]w}_{L^P(0, \infty, \lso)} \label{I-7.19}\\
		&\leq C \Big\{ \norm{\Fs}_{L^P(0, \infty, \lso)} + C \norm{w}_{L^P(0, \infty, \lso)} \Big\} + C \norm{A_{o,q}w}_{L^P(0, \infty, \lso)}, \label{I-7.20}
		\end{align}
		\noindent as $G$ is bounded. Using the same interpolation argument leading to (\ref{I-B.20}), based on the interpolation inequality (\ref{I-B.11}), we obtain from (\ref{I-7.20})	
		\begin{multline}\label{I-7.21}
		\norm{ A_q w}_{L^P(0, \infty, \lso)} \leq C\norm{\Fs}_{L^P(0, \infty, \lso)} + C \norm{w}_{L^P(0, \infty, \lso)}\\
		+\varepsilon C \norm{A_q w}_{L^P(0, \infty, \lso)} + C_{\varepsilon} \norm{w}_{L^P(0, \infty, \lso)}
		\end{multline}
		\noindent from which we obtain
		\begin{equation}
		\norm{ A_q w}_{L^P(0, \infty, \lso)} \leq \left( \frac{C}{1 - \varepsilon C} \right) \norm{\Fs}_{L^P(0, \infty, \lso)} + \left( \frac{C + C_{\varepsilon}}{1 - \varepsilon C}\right) \norm{w}_{L^P(0, \infty, \lso)}
		\end{equation}
		and then estimate in (\ref{I-7.18}) in Step 2 is established.\\
		
		\noindent \textit{Step 3}: Substituting (\ref{I-7.15}) in the RHS of (\ref{I-7.18}) yields
		\begin{equation}
		\norm{ A_q w}_{L^P(0, \infty, \lso)} \leq C\norm{\Fs}_{L^P(0, \infty, \lso)}
		\end{equation}
		\noindent and (\ref{I-7.11}) is established via (\ref{I-7.6}) with $w_0 = 0$, and $\calD(\BA_{F,q}) = \calD(A_q)$.\\
		
		\noindent \underline{Part ii}\\
		Let $\ds w_0 \in \lqaq$, \big[in particular $\ds w_0 \in \Bto$ for $\ds 1 < q < \infty, 1 < p < \rfrac{2q}{2q-1}$ by (\ref{I-1.16b}) \big] and consider the s.c. analytic exponentially stable semigroup $e^{\BA_{F,q} t}$ in such space, as guaranteed by Theorem \ref{I-Thm-6.1}, see (\ref{I-7.1}): 
		\begin{align}
		w(t) &= e^{\BA_{F,q}t}w_0; \quad w_t = \BA_{F,q} w = -A_q w + (P_qG - A_{o,q})w\\
		w(t) &= e^{-A_q t}w_0 + \int_{0}^{t}e^{-A_q(t-\tau)}(P_qG - A_{o,q})w(\tau)d \tau\\
		A_q w(t) &= A_q e^{-A_q t}w_0 + A_q \int_{0}^{t}e^{-A_q(t-\tau)}(P_qG - A_{o,q})w(\tau)d \tau
		\end{align}	
		\noindent counterpart of (\ref{I-B.18}), that is with $-A_{o,q}$ in (\ref{I-B.18}) replaced by $P_q G - A_{o,q}$ now, with $P_q G$ bounded, see (\ref{I-7.2}).Thus essentially the same proof leading to (\ref{I-B.24}) yields now
		\begin{align}
		\norm{\BA_{F,q}w}_{L^P(0, \infty, \lso)} &= \norm{\BA_{F,q} e^{\BA_{F,q}t}w_0}_{L^P(0, \infty, \lso)} \nonumber\\
		&\leq C \norm{w_0}_{\lqaq} \label{I-7.27}
		\end{align}
		\noindent with $\calD(\BA_{F,q}) = \calD(A_q)$. Then (\ref{I-7.27}) proves (\ref{I-7.14}).
	\end{proof}
	
	\section{Proof of Theorem \ref{I-Thm-2.3}: Well-posedness on $\xipq$ of the non-linear $z$-dynamics in feedback form}\label{I-Sec-8}
	
	\noindent In this section we return to the translated non-linear $z$-dynamics (\ref{I-1.12a}) and apply to it the feedback control $\ds u = \sum_{k = 1}^{K} (P_Nz,p_k)_{\omega}u_k$, i.e. of the same structure as the feedback identified on the RHS of the linearized $w$-dynamics (\ref{I-6.1}), which produced the s.c. analytic, uniformly stable feedback semigroup $e^{\BA_{F,q}t}$ on $\lso$. Here the vectors $\ds p_k \in (W^u_N)^*, \ u_k \in W^u_N$ are precisely those identified in Theorem \ref{I-Thm-5.1} = Theorem \ref{I-Thm-2.1}. Thus, returning to (\ref{I-1.12}), in this section we consider the following translated feedback non-linear problem
	\begin{equation}\label{I-8.1}
	\frac{dz}{dt}  - \calA_q z + \calN_q z = P_q \Bigg ( m \Bigg( \sum_{k=1}^{K} (z_N,p_k)_{\omega} u_k \Bigg) \Bigg); \quad z_0 = P_N z(0).
	\end{equation}
	\noindent Recalling from Theorem \ref{I-2.2} = Theorem \ref{I-Thm-6.1}, Eq (\ref{I-6.1}) the feedback generator $\BA_{F,q}$ as well as the bounded operator $G$ in (\ref{I-7.2}), we can rewrite (\ref{I-8.1}) as
	\begin{equation}\label{I-8.2}
	z_t = \BA_{F,q}z - \calN_q z = -(\nu A_q + A_{o,q})z + P_q Gz - \calN_q z, \quad z(0) = z_0, 
	\end{equation}
	\noindent whose variation of parameters formula is 
	\begin{equation}\label{I-8.3}
	z(t) = e^{\BA_{F,q}t}z_0 - \int_{0}^{t} e^{\BA_{F,q}(t - \tau)}\calN_q z(\tau) d \tau.
	\end{equation}
	\noindent We already know from (\ref{I-6.3}) that for $\ds z_0 \in \Bto, 1 < q < \infty, \ 1 < p < \rfrac{2q}{2q-1}$ we have
	\begin{equation}\label{I-8.4}
	\norm{e^{\BA_{F,q}t} z_0}_{\Bto} \leq M_{\gamma_0} e ^{-\gamma_0 t} \norm{z_0}_{\Bto}, \quad t \geq 0
	\end{equation}
	\noindent with $M_{\gamma_0}$ possibly depending on $p,q$. Maximal regularity properties corresponding to the solution operator formula in (\ref{I-8.3}) were established in section \ref{I-Sec-7}. Accordingly, for $z_0 \in \Bto$ and $f \in \xipqs \equiv L^p(0, \infty; \calD(\BA_{F,q})) \cap W^{1,p}(0, \infty;\lso), \ \calD(\BA_{F,q}) = \calD(A_q)$, recall (\ref{I-7.11}) we define the operator $\calF$ by
	\begin{equation}\label{I-8.5}
	\calF(z_0,f)(t) = e^{\BA_{F,q}t}z_0 - \int_{0}^{t} e^{\BA_{F,q}(t - \tau)}\calN_q f(\tau) d \tau.
	\end{equation}
	\noindent The main result of this section is Theorem \ref{I-Thm-2.3}. restated as  
	
	\begin{thm}\label{I-Thm-8.1}
		Let $d = 2,3, \ q > d$ and $\ds 1 < p < \rfrac{2q}{2q-1}$. There exists a positive constant $r_1 > 0$ (identified in the proof below in (\ref{I-8.24})), such that if
		\begin{equation}\label{I-8.6}
		\norm{z_0}_{\Bto} < r_1,
		\end{equation}
		\noindent then the operator $\calF$ in (\ref{I-8.5}) has a unique fixed point nonlinear semigroup solution on $\xipqs$ 
		\begin{equation}\label{I-8.7}
		\calF(z_0,z) = z, \text{ or } z(t) = e^{\BA_{F,q}t}z_0 - \int_{0}^{t} e^{\BA_{F,q}(t - \tau)}\calN_q z(\tau) d \tau
		\end{equation}
		\noindent which therefore is the unique solution of problem (\ref{I-8.2}) (= (\ref{I-8.1})) in $\xipqs$.
	\end{thm} 
	
	\noindent The proof of Theorem \ref{I-Thm-2.3} = Theorem \ref{I-Thm-8.1} is accomplished in two steps.\\
	
	\noindent \underline{Step 1}:
	\begin{thm}\label{I-Thm-8.2}
		Let $d = 2,3, \ q > d$ and $\ds 1 < p < \rfrac{2q}{2q-1}$. There exists a positive constant $r_1 > 0$ (identified in the proof below in (\ref{I-8.24})) and a subsequent constant $r>0$ (identified in the proof below in (\ref{I-8.22})) depending on $r_1 > 0$ and the constant $C>0$ in (\ref{I-8.20}), such that with $\ds \norm{z_0}_{\Bto} < r_1$ as in (\ref{I-8.6}), the operator $\calF(z_0,f)$ maps the ball $B(0,r)$ in $\xipqs$ into itself. \qedsymbol
	\end{thm}
	\noindent Theorem \ref{I-Thm-8.1} will follow then from Theorem \ref{I-Thm-8.2} after establishing that\\
	
	\noindent \underline{Step 2}:
	\begin{thm}\label{I-Thm-8.3}
		Let $d = 2,3, \ q > 3$ and $\ds 1 < p < \rfrac{2q}{2q-1}$. There exists a positive constant $r_1 > 0$, such that if $\ds \norm{z_0}_{\Bto} < r_1$ as in (\ref{I-8.6}), then there exists a constant $0 < \rho_0 < 1$, such that the operator $\calF(z_0,f)$ defines a contraction in the ball $B(0,\rho_0)$ of $\xipqs$ \qedsymbol
	\end{thm}
	
	\noindent The Banach contraction principle then establishes Theorem \ref{I-Thm-8.1}, once we prove Theorems \ref{I-Thm-8.2} and \ref{I-Thm-8.3}.\\
	
	\noindent \textbf{Proof of Theorem \ref{I-Thm-8.2}}.  \textit{Step 1}: We start from definition (\ref{I-8.5}) of $\calF$ and invoke the maximal regularity properties (\ref{I-7.14}) for $e^{\BA_{F,q}t}$ and (\ref{I-7.13}) for $\ds \int_{0}^{t} e^{\BA_{F,q}(t - \tau)} \calN_q f(\tau) d \tau$. We obtain from (\ref{I-8.5})
	
	\begin{align}
	\norm{\calF(z_0,f)(t)}_{\xipqs} &\leq \norm{e^{\BA_{F,q}t}z_0} _{\xipqs}+ \norm{\int_{0}^{t} e^{\BA_{F,q}(t - \tau)}\calN_q f(\tau) d \tau}_{\xipqs} \label{I-8.8}\\
	&\leq C \Big[ \norm{z_0}_{\Bto} + \norm{\calN_q f}_{L^p(0, \infty; \lso)} \Big]. \label{I-8.9} 
	\end{align}
	
	\noindent \textit{Step 2}: By the definition $\calN_q f = P_q [(f \cdot \nabla)f]$ in (\ref{I-1.11}), we estimate ignoring $\ds \norm{P_q}$ and using, $\ds \sup_{\cdot} \ \big[ \abs{g(\cdot)} \big]^r = [\sup_{\cdot} \ (\abs{g(\cdot)})]^r$
	
	\begin{align}
	\norm{\calN_q f}^p_{L^p(0,\infty;\lso)} &\leq \int_{0}^{\infty} \norm{P_q [(f \cdot \nabla)f]}^p_{\lso} dt \nonumber\\
	&\leq \int_{0}^{\infty} \bigg\{ \int_{\Omega} \abs{f(t,x)}^q \abs{ \nabla f(t,x)}^q d \Omega \bigg\}^{\rfrac{p}{q}} dt\\
	&\leq \int_{0}^{\infty} \bigg\{ \bigg[ \sup_{\Omega} \abs{ \nabla f(t, \cdot)}^q  \bigg]^{\rfrac{1}{q}} \bigg[ \int_{\Omega} \abs{f(t,x)}^{q} d \Omega \bigg]^{\rfrac{1}{q}}  \bigg\}^p dt\\
	&\leq \int_{0}^{\infty} \norm{\nabla f(t,\cdot)}^p_{L^{\infty}(\Omega)} \norm{f(t,\cdot)}^p_{\lso} dt \label{I-8.12}\\
	&\leq \sup_{0\leq t \leq \infty} \norm{f(t,\cdot)}^p_{\lso} \int_{0}^{\infty} \norm{\nabla f(t,\cdot)}^p_{L^{\infty}(\Omega)} dt\\
	&= \norm{f}^p_{L^{\infty}(0,\infty; \lso)} \norm{\nabla f}^p_{L^p(0,\infty; L^{\infty}(\Omega))}. \label{I-8.14}
	\end{align}
	
	\noindent \textit{Step 3}: The following embeddings hold true:
	\begin{enumerate}[(i)]
		\item \cite[Proposition 4.3, p 1406 with $\mu = 0, s = \infty, r = q$]{GGH:2012}  so that the required formula reduces to $1 \geq \rfrac{1}{p}$, as desired
		\begin{subequations}\label{I-8.15}
			\begin{align}
			f \in \xipqs \hookrightarrow f &\in L^{\infty}(0,\infty; \lso) \label{I-8.15a}\\
			\text{ so that, } \norm{f}_{L^{\infty}(0,\infty; \lso)} &\leq C\norm{f}_{\xipqs} \label{I-8.15b}
			\end{align}
		\end{subequations}
		\item \cite[Theorem 2.4.4, p 74 requiring $C^1$-boundary]{SK:1989}
		\begin{equation}\label{I-8.16}
		W^{1,q}(\Omega) \subset L^{\infty}(\Omega) \text{ for q}>\text{dim }\Omega = d, \ d = 2,3,
		\end{equation}
	\end{enumerate}
	
	\noindent so that, with $p>1, q>3$:
	\begin{align}
	\norm{\nabla f}^p_{L^p(0,\infty; L^{\infty}(\Omega))} &\leq C \norm{ \nabla f}^p_{L^p(0,\infty; W^{1,q}(\Omega))} \leq C \norm{f}^p_{L^p(0,\infty; W^{2,q}(\Omega))} \label{I-8.17}\\
	&\leq C \norm{f}^p_{\xipqs} \label{I-8.18}
	\end{align}
	
	\noindent In going from (\ref{I-8.17}) to (\ref{I-8.18}) we have recalled the definition of $f \in \xipqs$ in (\ref{I-1.28}), (\ref{I-7.13}), as $f$ was taken at the outset on $\ds \calD(\BA_{_{F,q}}) = \calD(\calA_q) \subset \lso$. Then, the sought-after final estimate of the non-linear term $\calN_q f, f \in \xipqs$ below (\ref{I-8.4}), is obtained from substituting (\ref{I-8.15b}) and (\ref{I-8.18}) into the RHS of (\ref{I-8.14}). We obtain 
	
	\begin{equation}\label{I-8.19}
	\norm{\calN_q f}_{L^p(0,\infty; \lso)} \leq C \norm{f}^2_{\xipqs}, \quad f \in \xipqs.
	\end{equation}
	
	\noindent Returning to (\ref{I-8.8}), we finally, obtain by (\ref{I-8.19})
	
	\begin{equation}\label{I-8.20}
	\norm{\calF (z_0, f)}_{\xipqs} \leq C \Big\{ \norm{z_0}_{\Bto} + \norm{f}^2_{\xipqs} \Big\}.
	\end{equation}
	
	\noindent \textit{Step 4}: We now impose the restrictions on the data on the RHS of (\ref{I-8.20}): $z_0$ is in a ball of radius $r_1 > 0$ in $\Bto$ and $f$ is in a ball of radius $r>0$ in $\xipqs$. We further demand that the final result $\calF(z_0,f)$ shall lie in a ball of radius $r$ in $\xipqs$. Thus we obtain from (\ref{I-8.20})
	
	\begin{equation}\label{I-8.21}
	\norm{\calF(z_0,f)}_{\xipqs} \leq C \Big\{ \norm{z_0}_{\Bto} + \norm{f}^2_{\xipqs} \Big\} \leq C(r_1 + r^2) \leq r
	\end{equation}
	
	\noindent This implies 
	\begin{equation}\label{I-8.22}
	Cr^2 - r + Cr_1 \leq 0 \quad \text{or} \quad \frac{1 - \sqrt{1-4C^2r_1}}{2C} \leq r \leq \frac{1 + \sqrt{1-4C^2r_1}}{2C}
	\end{equation}
	\noindent whereby 
	\begin{equation}\label{I-8.23}
	\begin{Bmatrix}
	\text{ range of values of r }
	\end{Bmatrix}
	\longrightarrow \text{ interval } \Big[ 0, \frac{1}{C} \Big], \text{ as } r_1 \searrow 0
	\end{equation}
	\noindent a constraint which is guaranteed by taking 
	\begin{equation}\label{I-8.24}
	r_1 \leq \frac{1}{4C^2},\ C \text{ being the constant in } (\ref{I-8.20}).
	\end{equation}
	\noindent We have thus established that by taking $r_1$ as in (\ref{I-8.24})  and subsequently $r$ as in (\ref{I-8.22}), then the map
	\begin{multline}\label{I-8.25}
	\calF(z_0, f) \text{ takes: }
	\begin{Bmatrix}
	\text{ ball in } \Bto \\
	\text{of radius } r_1
	\end{Bmatrix}
	\times 
	\begin{Bmatrix}
	\text{ ball in } \xipqs \\
	\text{of radius } r
	\end{Bmatrix}
	\text{ into }
	\begin{Bmatrix}
	\text{ ball in } \xipqs \\
	\text{of radius } r
	\end{Bmatrix},\\ \ d < q, \ 1 < p < \frac{2q}{2q-1}.
	\end{multline}
	\noindent This establishes Theorem \ref{I-Thm-8.2}. \qedsymbol \\
	
	\noindent \textbf{Proof of Theorem \ref{I-Thm-8.3}} \underline{Step 1}: For $f_1,f_2$ both in the ball of $\xipqs$ of radius $r$ obtained in the proof of Theorem \ref{I-Thm-8.2}, we estimate from (\ref{I-8.5}):
	\begin{align}
	\norm{ \calF(z_0,f_1) - \calF(z_0,f_2)}_{\xipqs} &= \norm{\int_{0}^{t} e^{\BA_{F,q}(t-\tau)} \big[ \calN_q f_1(\tau) - \calN_q f_2(\tau) \big] d \tau}_{\xipqs} \label{I-8.26}\\
	&\leq  \widetilde{m} \norm{\calN_q f_1 - \calN_q f_2}_{L^p(0,\infty;\lso)} \label{I-8.27}
	\end{align}
	\noindent after invoking the maximal regularity property (\ref{I-7.13}).\\
	
	\noindent \underline{Step 2}: Next recalling $\calN_qf = P_q [(f \cdot \nabla )f]$ from (\ref{I-1.11}), we estimate the RHS of (\ref{I-8.27}). In doing so, we add and subtract $(f_2 \cdot \nabla) f_1$, set $ \ds A = (f_1 \cdot \nabla) f_1 - (f_2 \cdot \nabla) f_1, \ B = (f_2 \cdot \nabla) f_1 - (f_2  \cdot \nabla) f_2,$ and use 
	\begin{equation*}
	\abs{A+B}^q \leq 2^q \big[\abs{A}^q + \abs{B}^q \big] \qquad (*). 
	\end{equation*}
	\noindent \cite[p 12]{TL:1980} We obtain, again ignoring $\norm{P_q}$:
	\begin{align}
	\norm{\calN_q f_1 - \calN_q f_2}_{L^p(0,\infty;\lso)} &\leq \int_{0}^{\infty} \bigg\{ \bigg[ \int_{\Omega} \abs{(f_1 \cdot \nabla) f_1 - (f_2 \cdot \nabla) f_2}^q d \Omega \bigg]^{\rfrac{1}{q}}\bigg\}^p dt\\
	&= \int_{0}^{\infty} \bigg[ \int_{\Omega} \abs{A+B}^q d \Omega \bigg]^{\rfrac{p}{q}}dt\\
	&\leq 2^q \int_{0}^{\infty} \bigg\{ \int_{\Omega} \big[\abs{A}^q + \abs{B}^q \big] d \Omega \bigg\}^{\rfrac{p}{q}}dt\\
	&= 2^q \int_{0}^{\infty} \bigg\{ \Big[ \int_{\Omega} \abs{A}^q d \Omega + \int_{\Omega} \abs{B}^q d \Omega   \Big]^{\rfrac{1}{q}} \bigg\}^p dt\\
	&= 2^q \int_{0}^{\infty} \bigg\{ \Big[ \norm{A}^q_{L^q(\Omega)} + \norm{B}^q_{L^q(\Omega)} \Big]^{\rfrac{1}{q}} \bigg\}^p dt\\ (\text{ by } (*)) \qquad
	&\leq 2^q \cdot 2^{\rfrac{1}{q}}\int_{0}^{\infty} \Big\{ \norm{A}_{L^q(\Omega)} + \norm{B}_{L^q(\Omega)} \Big\}^p dt\\ (\text{ by } (*)) \qquad
	&\leq 2^{p+q+\rfrac{1}{q}}\int_{0}^{\infty} \Big[ \norm{A}^p_{L^q(\Omega)} + \norm{B}^p_{L^q(\Omega)} \Big] dt\\ 
	&= 2^{p+q+\rfrac{1}{q}}\int_{0}^{\infty} \Big[ \norm{((f_1 - f_2)\cdot \nabla)f_1}^p_{L^q(\Omega)}\nonumber \\ &\hspace{5cm}+\norm{(f_2 \cdot \nabla)(f_1 - f_2 )}^p_{L^q(\Omega)} \Big] dt\\
	&= 2^{p+q+\rfrac{1}{q}}\int_{0}^{\infty} \Big\{ \norm{f_1 - f_2}^p_{L^q(\Omega)} \norm{\nabla f_1}^p_{L^q(\Omega)}\nonumber \\ &\hspace{5cm}+\norm{f_2}^p_{\lso} \norm{\nabla(f_1 - f_2)}^p_{L^q(\Omega)} \Big\} dt \label{I-8.36}
	\end{align}
	
	\noindent \underline{Step 3}: We now notice that regarding each of the integral term in the RHS of (\ref{I-8.36}) we are structurally and topologically as in the RHS of (\ref{I-8.12}), except that in (\ref{I-8.36}) the gradient terms $\nabla f_1, \nabla(f_1 - f_2)$ are penalized in the $\lso$-norm which is dominated by the $L^{\infty}(\Omega)$-norm, as it occurs for the gradient term $\nabla f$ in (\ref{I-8.12}). Thus we can apply to each integral term on the RHS of (\ref{I-8.36}) the same argument as in going from (\ref{I-8.12}) to the estimates (\ref{I-8.15b}) and (\ref{I-8.18}) with $q >$ dim $\Omega=3$. We obtain
	\begin{align}
	\norm{\calN_q f_1 - \calN_q f_2}^p_{L^p(0,\infty;\lso)} &\leq \text{RHS of (\ref{I-8.36})} \nonumber \\
	\text{see (\ref{I-8.14})} \hspace{3cm}&\leq C \Big\{ \norm{f_1 - f_2}^p_{L^{\infty}(0, \infty;\lso)} \norm{\nabla f_1}^p_{L^p(0, \infty;L^{\infty}(\Omega))}\nonumber \\ &\hspace{3cm} + \norm{f_2}^p_{L^{\infty}(0, \infty;\lso)} \norm{\nabla(f_1 - f_2)}^p_{L^p(0, \infty;L^{\infty}(\Omega)} \Big\}\\
	\text{see (\ref{I-8.15b}) and (\ref{I-8.18})} \qquad &\leq C \Big\{ \norm{f_1 - f_2}^p_{\xipqs} \norm{f_1}^p_{\xipqs} +  \norm{f_2}^p_{\xipqs} \norm{f_1 - f_2}^p_{\xipqs} \Big\}\\
	&= C \Big\{ \norm{f_1 - f_2}^p_{\xipqs} \big( \norm{f_1}^p_{\xipqs} + \norm{f_2}^p_{\xipqs}\big) \Big\} \label{I-8.39}
	\end{align}
	\noindent Finally (\ref{I-8.39}) yields
	\begin{align}
	\norm{\calN_q f_1 - \calN_q f_2}_{L^p(0,\infty;\lso)} &\leq C^{\rfrac{1}{p}} \norm{f_1 - f_2}_{\xipqs} \Big( \norm{f_1}^p_{\xipqs} + \norm{f_2}^p_{\xipqs} \Big)^{\rfrac{1}{p}}\\
	(\text{by } (*)) \qquad &\leq 2^{\rfrac{1}{p}} C^{\rfrac{1}{p}} \norm{f_1 - f_2}_{\xipqs} \Big( \norm{f_1}_{\xipqs} + \norm{f_2}_{\xipqs} \Big). \label{I-8.41}
	\end{align}
	\noindent \underline{Step 4}: Using estimate (\ref{I-8.41}) on the RHS of estimate (\ref{I-8.27}) yields
	\begin{equation}\label{I-8.42}
	\norm{\calF(z_0,f_1) - \calF(z_0,f_2)}_{\xipqs} \leq K_p \norm{f_1 - f_2}_{\xipqs} \Big( \norm{f_1}_{\xipqs} + \norm{f_2}_{\xipqs} \Big)
	\end{equation}
	\noindent $K_p = \widetilde{m} 2^{\rfrac{1}{p}}C^{\rfrac{1}{p}}$ ($\widetilde{m}$ as in (\ref{I-8.27}), $C$ as in (\ref{I-8.39})). Next, pick $f_1,f_2$ in the ball of $\xipqs$ of radius $R$:
	\begin{equation}
	\norm{f_1}_{\xipqs},\norm{f_2}_{\xipqs} \leq R.
	\end{equation}
	\noindent Then 
	\begin{equation}
	\norm{\calF(z_0,f_1) - \calF(z_0,f_2)}_{\xipqs} \leq \rho_0 \norm{f_1 - f_2}_{\xipqs}
	\end{equation}
	\noindent and $\calF(z_0,f)$ is a contraction on the space $\xipqs$ as soon as 
	\begin{equation}
	\rho_0 \equiv 2K_pR < 1 \text{ or } R < \rfrac{1}{2 K_p}, \ K_p = \widetilde{m} 2^{\rfrac{1}{p}} C^{\rfrac{1}{p}}. 
	\end{equation}
	\noindent In this case, the map $\calF(z_0,f)$ defined in (\ref{I-8.5}) has a fixed point $z$ in $\xipqs$
	\begin{equation}
	\calF(z_0,z) = z, \text{ or } z = e^{\BA_{F,q}t}z_0 - \int_{0}^{t}e^{\BA_{F,q}(t - \tau)}\calN_q z(\tau) d \tau
	\end{equation}
	\noindent and such fixed point $z \in \xipqs$ is the unique solution of the translated non-linear equation (\ref{I-8.1}), or (\ref{I-8.2}) with finite dimensional control $u$ in feedback form, as described by the RHS of (\ref{I-8.1}). Theorem \ref{I-Thm-8.1} is proved. \qedsymbol 
	
	\section{Proof of Theorem \ref{I-Thm-2.4}. Local exponential decay of the non-linear translated $z$-dynamics (\ref{I-8.1}) with finite dimensional localized feedback control}\label{I-Sec-9}
	
	\noindent In this section we return to the feedback problem (\ref{I-8.1}) rewritten equivalently as in (\ref{I-8.3})
	
	\begin{equation}\label{I-9.1}
	z(t) = e^{\BA_{F,q}t}z_0 - \int_{0}^{t}e^{\BA_{F,q}(t - \tau)} \calN_q z(\tau) d \tau.
	\end{equation} 
	\noindent For $z_0$ in a small ball of $\Bto$, Theorem \ref{I-Thm-8.1} provides a unique solution in a ball of $\xipqs$. We recall from (\ref{I-6.3}) = (\ref{I-8.4})
	\begin{equation}\label{I-9.2}
	\norm{e^{\BA_{F,q}t} z_0}_{\Bto} \leq M_{\gamma_0} e^{-\gamma_0 t} \norm{z_0}_{\Bto}, \ t \geq 0.
	\end{equation}
	\noindent Our goal now is to show that for $z_0$ in a small ball of $\Bto$, problem (\ref{I-9.1}) satisfies the exponential decay 
	\begin{equation}
	\norm{z(t)}_{\Bto} \leq C e^{-at} \norm{z_0}_{\Bto}, \ t \geq 0, \ \text{for some constants, } a>0, C = C_a \geq 1. \nonumber
	\end{equation}
	\noindent \underline{Step 1}: Starting from (\ref{I-9.1}) and using (\ref{I-9.2}) we estimate
	\begin{align}
	\norm{z(t)}_{\Bto} &\leq M_{\gamma_0} e^{-\gamma_0 t} \norm{z_0}_{\Bto} + \sup_{0\leq t \leq \infty} \norm{\int_{0}^{t}e^{\BA_{F,q}(t - \tau)} \calN_q z(\tau) d \tau}_{\Bto} \label{I-9.3}\\
	&\leq M_{\gamma_0} e^{-\gamma_0 t} \norm{z_0}_{\Bto} + C_1 \norm{\int_{0}^{t}e^{\BA_{F,q}(t - \tau)} \calN_q z(\tau) d \tau}_{\xipqs} \label{I-9.4}\\
	&\leq M_{\gamma_0} e^{-\gamma_0 t} \norm{z_0}_{\Bto} + C_2 \norm{\calN_q z}_{L^p(0,\infty;\lso)} \label{I-9.5}\\
	\norm{z(t)}_{\Bto} &\leq M_{\gamma_0} e^{-\gamma_0 t} \norm{z_0}_{\Bto} + C_3 \norm{z}^2_{\xipqs}, \quad C_3 = C_2C. \label{I-9.6}
	\end{align}
	\noindent In going from (\ref{I-9.3}) to (\ref{I-9.4}) we have recalled the embedding $\ds \xipqs \hookrightarrow L^{\infty}\big(0,\infty;\Bto \big)$ from (\ref{I-1.30}). Next, in going from (\ref{I-9.4}) to (\ref{I-9.5}) we have used the maximal regularity property (\ref{I-7.13}). Finally, to go from (\ref{I-9.5}) to (\ref{I-9.6}) we have invoked estimate (\ref{I-8.19}).\\
	
	\noindent \underline{Step 2}: We shall next establish that
	\begin{equation}\label{I-9.7}
	\norm{z}_{\xipqs} \leq M_1 \norm{z_0}_{\Bto} + K \norm{z}^2_{\xipqs}, \text{ hence } \norm{z}_{\xipqs} \big(1-K\norm{z}_{\xipqs} \big) \leq M_1 \norm{z_0}_{\Bto}
	\end{equation}
	\noindent In fact, to this end, we take the $\xipqs$-estimate of equation (\ref{I-9.1}). We obtain 
	\begin{equation}
	\norm{z}_{\xipqs} \leq \norm{e^{\BA_{F,q}t}z_0}_{\xipqs} + \norm{\int_{0}^{t} e^{\BA_{F,q}(t-\tau)}\calN_qz(\tau) d \tau}_{\xipqs}
	\end{equation}
	\noindent from which then (\ref{I-9.7}) follows by invoking the maximal regularity property (\ref{I-7.14}) on $e^{\BA_{F,q}t}$ as well as the maximal regularity estimate (\ref{I-7.13}) followed by use of of (\ref{I-8.19}), as in going from (\ref{I-9.4}) to (\ref{I-9.6})
	\begin{align}
	\norm{\int_{0}^{t} e^{\BA_{F,q}(t-\tau)}\calN_qz(\tau) d \tau}_{\xipqs} &\leq \wti{M} \norm{\calN_q z}_{L^p(0,\infty;\lso)} \label{I-9.9}\\
	&\leq \wti{M} C \norm{z}^2_{\xipqs}. \label{I-9.10}
	\end{align}
	\noindent Thus (\ref{I-9.7}) is proved with $K = \wti{M}C$ where $C$ is the same constant occurring in (\ref{I-8.19}), hence in (\ref{I-8.21}), (\ref{I-8.22}).\\
	
	\noindent \underline{Step 3}: The well-posedness Theorem \ref{I-Thm-8.1} says that
	\begin{equation}\label{I-9.11}
	\begin{Bmatrix}
	\text{ If } \norm{z_0}_{\Bto} \leq r_1 \\
	\text{for } r_1 \text{ sufficiently small}
	\end{Bmatrix}
	\implies
	\begin{Bmatrix}
	\text{ The solution } z \text{ satisfies} \\
	\norm{z}_{\xipqs} \leq r
	\end{Bmatrix}
	\end{equation}
	where $r$ satisfies the constraint (\ref{I-8.22}) in terms of $r_1$ and some constant $C$ in (\ref{I-8.19}) that occurs for $K = \wti{M}C$ in (\ref{I-9.10}). We seek to guarantee that we can obtain
	\begin{equation}\label{I-9.12}
	\begin{cases}
	\norm{z}_{\xipqs} \leq r < \frac{1}{2K} = \frac{1}{2 \wti{M} C} \ \Big( < \frac{1}{2C}\Big)\\
	\ \\
	\text{hence } \frac{1}{2} < 1 - K \norm{z}_{\xipqs},
	\end{cases}
	\end{equation}
	\noindent where w.l.o.g. we can take the maximal regularity constant $\wti{M}$ in (\ref{I-7.13}) to satisfy $\wti{M} > 1$. Again, the constant $C$ arises from application of estimate (\ref{I-8.19}). This is indeed possible by choosing $r_1 > 0$ sufficiently small. In fact, as $r_1 \searrow 0$, (\ref{I-8.23}) shows that the interval $r_{min} \leq r \leq r_{max}$ of corresponding values of $r$ tends to the interval $\ds \bigg[ 0, \frac{1}{C} \bigg]$. Thus (\ref{I-9.12}) can be achieved as $r_{min} \searrow 0$: $\ds 0 < r_{min} < r < \frac{1}{2 \wti{M} C} < \frac{1}{2C}$. Next, (\ref{I-9.12}) implies that (\ref{I-9.7}) holds true and yields then 
	\begin{equation}\label{I-9.13}
	\norm{z}_{\xipqs} \leq 2M_1 \norm{z_0}_{\Bto} \leq 2M_1 r_1.
	\end{equation}
	\noindent Substituting (\ref{I-9.13}) in estimate (\ref{I-9.6}) then yields with $\ds \widehat{M} = \max \{ M_{\gamma_0}, M_1 \}$
	\begin{align}
	\norm{z(t)}_{\Bto} &\leq M_{\gamma_0} e^{-\gamma_0 t} \norm{z_0}_{\Bto} + 4 C_3 M_1^2 \norm{z_0}^2_{\Bto}\\
	&= \widehat{M} \bigg[ e^{-\gamma_0 t} + 4 \widehat{M} C_3 \norm{z_0}_{\Bto} \bigg] \norm{z_0}_{\Bto} \label{I-9.15}\\
	\norm{z(t)}_{\Bto} &\leq \widehat{M} \big[ e^{-\gamma_0 t} + 4 \widehat{M} C_3 r_1 \big] \norm{z_0}_{\Bto}
	\end{align}
	\noindent recalling the constant $r_1 > 0$ in (\ref{I-9.11}).\\
	
	\noindent \underline{Step 4}: Now take $T$ sufficiently large and $r_1 > 0$ sufficiently small such that 
	\begin{equation}\label{I-9.17}
	\beta \equiv \widehat{M} e^{-\gamma_0 T} + 4 \widehat{M}^2 C_3 r_1 < 1.
	\end{equation}
	\noindent Then (\ref{I-9.15}) implies by (\ref{I-9.17})
	\begin{subequations}\label{I-9.18}
	\begin{align}
	\norm{z(T)}_{\Bto} &\leq \beta \norm{z_0}_{\Bto} \text{ and hence } \label{I-9.18a}\\
	\norm{z(nT)}_{\Bto} &\leq \beta \norm{z((n-1)T)}_{\Bto} \leq \beta^n \norm{z_0}_{\Bto}. \label{I-9.18b}
	\end{align}	
	\end{subequations}
	\noindent Since $\beta < 1$, the semigroup property of the evolution implies that there are constants $M_{\wti{\gamma}} \geq 1, \wti{\gamma} > 0$ such that
	\begin{equation}\label{I-9.19}
	\norm{z(t)}_{\Bto} \leq M_{\wti{\gamma}} e^{-\wti{\gamma} t} \norm{z_0}_{\Bto}, \quad t \geq 0
	\end{equation}
	\noindent This proves Theorem \ref{I-Thm-2.4}. \qedsymbol  
	
	\begin{rmk}\label{I-Rmk-9.1}
		The above computations - (\ref{I-9.17}) through (\ref{I-9.19}) - can be used to support qualitatively the intuitive expectation that ``the larger the decay rate $\gamma_0$ in (\ref{I-6.3}) = (\ref{I-8.4}) = (\ref{I-9.2}) of the linearized feedback $w$-dynamics (\ref{I-6.1}), the larger the decay rate $\wti{\gamma}$ in (\ref{I-9.19}) of the nonlinear feedback $z$-dynamics (\ref{I-2.20}) = (\ref{I-8.1}); hence the larger the rate $\wti{\gamma}$ in (\ref{I-2.29}) of the original $y$-dynamics in (\ref{I-2.28})".\\
		
		\noindent The following considerations are somewhat qualitative. Let $S(t)$ denote the non-linear semigroup in the space $\ds \Bto$, with infinitesimal generator $\ds \big[ \BA_{_{F,q}} - \calN_q \big]$ describing the feedback $z$-dynamics (\ref{I-2.20})=(\ref{I-8.1}), hence (\ref{I-8.2}), as guaranteed by the well posedness Theorem \ref{I-Thm-2.3} = Theorem \ref{I-Thm-8.1}. Thus, $\ds z(t;z_0) = S(t)z_0$ on $\ds \Bto$. By (\ref{I-9.17}), we can rewrite (\ref{I-9.18a}) as:
		\begin{equation}\label{I-9.20}
			\norm{S(T)}_{\calL \big(\Bto \big)} \leq \beta < 1.
		\end{equation}
		\noindent It follows from \cite[p 178]{Bal:1981} via the semigroup property that
		\begin{equation}\label{I-9.21}
			- \wti{\gamma} \ \ \text{is just below} \ \ \frac{\ln \beta}{T} < 0.
		\end{equation}
		\noindent Pick $r_1 > 0$ in (\ref{I-9.17}) so small that $4\widehat{M}^2 C_3 r_1$ is negligible, so that $\beta$ is just above $\ds \widehat{M} e^{- \gamma_0 T}$, so $\ln \beta$ is just above $\ds \big[ \ln \widehat{M} - \gamma_0 T \big]$, hence
		\begin{equation}\label{I-9.22}
			\frac{\ln \beta}{T} \text{ is just above } (-\gamma_0) + \frac{\ln \widehat{M}}{T}.
		\end{equation}
		\noindent Hence, by (\ref{I-9.21}), (\ref{I-9.22}),
		\begin{equation}\label{I-9.23}
			\wti{\gamma} \sim \gamma_0 - \frac{\ln \widehat{M}}{T}
		\end{equation}
		\noindent and the larger $\gamma_0$, the larger is $\wti{\gamma}$, as desired.
	\end{rmk}
	
	\section{Well-posedness of the pressure $\chi$ for the $z$-problem (\ref{I-1.7}) in feedback form, and of the pressure $\pi$ for the $y$-problem (\ref{I-1.1}) in feedback form (\ref{I-2.22}) in the vicinity of the equilibrium pressure $\pi_e$.}\label{I-Sec-10}
	
	\noindent \underline{The $z$-problem in feedback form:} We return to the translated $z$ problem (\ref{I-1.7}), with $\L_e(z)$ given by (\ref{I-1.39})	
	\begin{subequations}\label{I-10.1}
		\begin{align}
		z_t - \nu \Delta z + L_e(z) + (z \cdot \nabla) z + \nabla \chi &= m(Fz)   &\text{ in } Q \label{I-10.1a}\\ 
		\div \ z &= 0   &\text{ in } Q \\
		z &= 0 &\text{ on } \Sigma\\
		z(0,x) &= y_0(x) - y_e(x) &\text{ on } \Omega
		\end{align}	
	\noindent with $Fz$ given in the feedback form as in (\ref{I-2.20}) = (\ref{I-8.1})	
	\begin{equation}
		m(Fz) = m \Bigg( \sum_{k = 1}^K \big( z_N, p_k \big)_{\omega} u_k \Bigg), \quad z_N = P_N z
	\end{equation}
	\end{subequations}
	\noindent for which Theorem \ref{I-Thm-2.3} = Theorem \ref{I-Thm-8.1} provides a local well-posedness result (\ref{I-2.22}), (\ref{I-2.23}) for the $z$ variable. We now complement such well-posedness for $z$ with a corresponding local well-posedness result for the pressure $\chi$.
	
	\begin{thm}
		Consider the setting of Theorem \ref{I-Thm-2.3} = Theorem \ref{I-8.1} for problem (\hyperref[I-10.1a]{10.1a-e}). Then the following well-posedness result for the pressure $\chi$ holds true, where we recall the spaces $\ds \yipq$ for $T = \infty$ and $\ds \widehat{W}^{1,q}(\Omega)$ in (\ref{I-1.29}) as well as the steady state pressure $\pi_e$ from Theorem \ref{I-Thm-1.1}:
		\begin{equation}\label{I-10.2}
		\norm{\chi}_{\yipq} \leq \wti{C} \norm{y_0 - y_e}_{\Bto} \Big\{ \norm{y_0 - y_e}_{\Bto} + 1  \Big\}.
		\end{equation}
	\end{thm}

\begin{proof}
	We first apply the full maximal-regularity (\ref{I-1.33}) to the Stokes component of problem (\ref{I-10.1}) with $F_q = P_q \big( mF(z) - L_e(z) - (z \cdot \nabla)z \big)$ to obtain
	\begin{align}
	\norm{z}_{\xipqs} + \norm{\chi}_{\yipq} &\leq C \Big\{ \norm{ P_q [m(Fz) - (z \cdot \nabla) z - L_e(z)]}_{\lplqs} + \norm{z_0}_{\Bto}  \Big\} \nonumber\\
	&\leq C \Big\{ \norm{P_q [m(Fz)]}_{\lplqs} + \norm{P_q (z \cdot \nabla) z}_{\lplqs} \nonumber\\
	& \hspace{4cm}+\norm{P_q L_e(z)}_{\lplqs} + \norm{z_0}_{\Bto}  \Big\}. \label{I-10.3}
	\end{align}
	
	\noindent But $P_q [m F(z)] = m F(z)$ as the vectors $u_k$ in the definition of $F$ in (\ref{I-2.26}) are $\ds u_k \in W^u_N \subset \lso $. Moreover $F \in \calL (\lso)$, we obtain
	\begin{equation}\label{I-10.4}
	\norm{P_q[m(Fz)]}_{\lplqs} \leq C_1 \norm{z}_{\xipqs}, 
	\end{equation} 
	\noindent recalling the space $\ds \xipqs$ from (\ref{I-1.28}). Next, recalling (\ref{I-8.19}) for $\ds \calN_q z = P_q \big[ (z \cdot \nabla) z\big]$, see (\ref{I-1.11}), we obtain
	\begin{equation}\label{I-10.5}
	\norm{P_q (z \cdot \nabla) z}_{\lplqs} \leq C_2 \norm{z}^2_{\xipqs}.
	\end{equation} 
	\noindent The equilibrium solution $\{y_e,\pi_e\}$ is given by Theorem \ref{I-Thm-1.1} as satisfying 
	\begin{equation}\label{I-10.6}
		\norm{y_e}_{W^{2,q}(\Omega)} + \norm{\pi_e}_{\widehat{W}^{q,1}} \leq c \norm{f}_{L^q(\Omega)}, \quad 1 < q < \infty.
	\end{equation}
	\noindent We next estimate the term $\ds P_q L_e(z) = P_q [(y_e \cdot \nabla)z + (z \cdot \nabla) y_e]$ in (\ref{I-10.3})
	\begin{align}
	\norm{P_q L_e(z)}_{\lplqs} &= \norm{P_q (y_e . \nabla) z + P_q (z. \nabla) y_e }_{\lplqs}\\
	&\leq \norm{P_q (y_e . \nabla) z}_{\lplqs} + \norm{P_q (z. \nabla) y_e }_{\lplqs}\\
	&\leq \norm{y_e}_{L^q(\Omega)} \norm{\nabla z}_{\lplqs} + \norm{z}_{\lplqs} \norm{\nabla y_e}_{L^q(\Omega)}\\
	&\leq 2C_2 \norm{f}_{L^q(\Omega)} \norm{z}_{\lplqs}\\
	&\leq C_3 \norm{z}_{\xipqs} \label{I-10.11}
	\end{align}
	\noindent with the constant $C_3$ depending on the $L^q(\Omega)$-norm of the datum $f$. Setting now $C_4 = C \cdot \max \{ C_1, C_2, C_3 \} $ and substituting (\ref{I-10.4}), (\ref{I-10.5}), (\ref{I-10.11}) in (\ref{I-10.3}), we obtain
	\begin{equation}\label{I-10.12}
	\norm{z}_{\xipqs} + \norm{\chi}_{\yipq} \leq C_4 \Big\{ \norm{z}^2_{\xipqs} + 2\norm{z}_{\xipqs} + \norm{z_0}_{\Bto}  \Big\}
	\end{equation}
	\noindent Next we drop the term $\ds \norm{z}_{\xipqs}$ on the left hand side of (\ref{I-10.12}) and invoking (\ref{I-9.13}) to estimate $\ds \norm{z}_{\xipqs}$. Thus we obtain
	\begin{align}
	\norm{\chi}_{\yipq} &\leq C_5 \Big\{ \norm{z_0}^2_{\Bto} + 2\norm{z_0}_{\Bto} + \norm{z_0}_{\Bto}  \Big\}\\
	&\leq \wti{C} \norm{z_0}_{\Bto} \Big\{ \norm{z_0}_{\Bto} + 1  \Big\}, \quad \wti{C} = 3C_5 \label{I-10.14}
	\end{align}
	\noindent and (\ref{I-10.14}) proves (\ref{I-10.2}), as desired, recalling (\ref{I-1.7e}). 
\end{proof}	
	
	\noindent \underline{The $y$-problem in feedback form} We return to the original $y$-problem however in feedback form as in (\ref{I-2.26}), (\ref{I-2.27}), for which Theorem \hyperref[I-Thm-2.5.i]{2.5(i)} proves a local well-posedness result. We now complement such well-posedness result for $y$ with the corresponding local well-posedness result for the pressure $\pi$.

	\begin{thm}\label{I-Thm-10.2}
		Consider the setting of Theorem \ref{I-Thm-2.5} for the $y$-problem in (\ref{I-2.27}). Then, the following well-posedness result for the pressure $\pi$ holds true:
		\begin{align}
			\norm{\pi - \pi_e}_{\ytpq} &\leq \norm{\pi - \pi_e}_{\yipq} \leq \wti{C} \norm{y_0 - y_e}_{\Bto} \Big\{ \norm{y_0 - y_e}_{\Bto} + 1  \Big\} \label{I-10.15}\\
			&\leq \widehat{C} \Big\{ \norm{y_0}_{\Bto} + \norm{y_e}_{W^{2,q}(\Omega)} \Big\} \Big\{ \norm{y_0}_{\Bto} + \norm{y_e}_{W^{2,q}(\Omega)} + 1  \Big\} \label{I-10.16}\\
			&\leq \widehat{C} \Big\{ \norm{y_0}_{\Bto} + \norm{f}_{L^q(\Omega)} \Big\} \Big\{ \norm{y_0}_{\Bto} + \norm{f}_{L^q(\Omega)} + 1  \Big\} \label{I-10.17}
		\end{align}
		\begin{align}
			\norm{\pi}_{\ytpq} &\leq \widehat{C} \norm{y_0 - y_e}_{\Bto} \Big\{ \norm{y_0 - y_e}_{\Bto} + 1  \Big\} + cT^{\rfrac{1}{p}} \norm{\pi_e}_{\widehat{W}^{1,q}(\Omega)}, \ 0 < T < \infty\\
			&\leq \widehat{C} \Big\{ \norm{y_0}_{\Bto} + \norm{f}_{L^q(\Omega)} \Big\} \Big\{ \norm{y_0}_{\Bto} + \norm{f}_{L^q(\Omega)} + 1  \Big\} \nonumber\\
			&\hspace{8cm}+ cT^{\rfrac{1}{p}} \norm{f}_{L^q(\Omega)}, \ 0 < T < \infty.
		\end{align}
	\end{thm}

\begin{proof}
	We return to the estimate (\ref{I-10.2}) for $\chi$ and recall $\chi = \pi - \pi_e$ from (\ref{I-1.7e}) to obtain (\ref{I-10.15}). We next estimate $y-y_e$ by 
	\begin{equation}
	\norm{y_0 - y_e}_{\Bto} \leq C  \big\{ \norm{y_0}_{\Bto} + \norm{y_e}_{W^{2,q}(\Omega)} \big\}.
	\end{equation}
	\noindent which substituted in (\ref{I-10.15}) yields (\ref{I-10.16}). In turn, (\ref{I-10.16}) leads to (\ref{I-10.17}) by means of (\ref{I-10.6}).
\end{proof}

	\begin{appendices}
	\renewcommand{\appendixname}{Appendix}
	\section{On Helmholtz Decomposition}\label{I-app-A}
	\setcounter{equation}{0}
	\setcounter{theorem}{1}
	\renewcommand{\theequation}{{\rm A}.\arabic{equation}}
	\renewcommand{\thetheorem}{{\bf A}.\arabic{theorem}}
	\renewcommand{\theprop}{{\bf A}.\arabic{theorem}}
	\renewcommand{\thermk}{{\bf A}.\arabic{theorem}}
	\noindent We return to the Helmholtz decomposition in (\ref{I-1.4}), (\ref{I-1.5}) and provide additional information.
	
	\noindent For $M \subset L^q(\Omega), \ 1 < q < \infty$, we denote the annihilator of $M$ by
	
	\begin{equation}
	M^{\perp} = \bigg \{ f \in L^{q'}(\Omega) : \int_{\Omega} fg \ d \Omega = 0, \text{ for all } g \in M \bigg \}
	\end{equation}
	\noindent where $q'$ is the dual exponent of $\ds q: \ \rfrac{1}{q} + \rfrac{1}{q'} = 1$.
	
	\begin{prop}\cite[Prop 2.2.2 p6]{HS:2016}, \cite[Ex. 16 p115]{Ga:2011}\label{I-Prop-1.2}, \cite{FMM:1998} \ \\
		Let $\Omega \subset \BR^d$ be an open set and let $1 < q < \infty$.
		\begin{enumerate}[a)]
			\item The Helmholtz decomposition exists for $L^q(\Omega)$ if and only if it exists for $L^{q'}(\Omega)$, and we have: (adjoint of $P_q$) = $P^*_q = P_{q'}$ (in particular $P_2$ is orthogonal), where $P_q$ is viewed as a bounded operator $\ds L^q(\Omega) \longrightarrow L^q(\Omega)$, and $\ds P^*_q = P_{q'}$ as a bounded operator $\ds L^{q'}(\Omega) \longrightarrow L^{q'}(\Omega), \ \rfrac{1}{q} + \rfrac{1}{q'} = 1$.
			\item Then, with reference to (\ref{I-1.5})
			\begin{subequations}
				\begin{equation}
				\Big[ \lso \Big]^{\perp} = G^{q'}(\Omega) \text{ and } \Big[ G^q(\Omega) \Big]^{\perp} = L^{q'}_{\sigma}(\Omega). \label{I-A.2a}
				\end{equation}
				\begin{rmk}
					\noindent Throughout the paper we shall use freely that
					\begin{equation}
					\big( \lso \big)' = \lo{q'}, \quad \frac{1}{q} + \frac{1}{q'} = 1. \label{I-A.2b}
					\end{equation}
					\noindent Thus can be established as follows. From (\ref{I-1.5}) write $\ds \lso$ as a factor space $\ds \lso = L^q(\Omega) / G^q(\Omega) \equiv X/M$ so that \cite[p 135]{TL:1980}.
					\begin{equation}
					\big( \lso \big)' = \big( L^q(\Omega) / G^q(\Omega) \big)' =  \big( X/M \big)' = M^{\perp} = \Big[ G^q(\Omega) \Big]^{\perp} = \lo{q'}. \label{I-A.2c}
					\end{equation}
					\noindent In the last step, we have invoked (\ref{I-A.2a}), which is also established in \cite[Lemma 2.1, p 116]{Ga:2011}. Similarly
					\begin{equation}
					\big( G^q(\Omega) \big)' = \big( L^q(\Omega) / \lso \big)' = \Big[ \lso \Big]^{\perp} = G^{q'}(\Omega).
					\end{equation}
				\end{rmk}
			\end{subequations}
			
		\end{enumerate}
	\end{prop}
	
	\section{Proof of Theorem \ref{I-Thm-1.6}: maximal regularity of the Oseen operator $\calA_q$ on $\lso$, $1 < p,q < \infty, T < \infty$.}\label{I-app-B}
	\setcounter{equation}{0}
	\renewcommand{\theequation}{{\rm B}.\arabic{equation}}
	\renewcommand{\thetheorem}{{\bf B}.\arabic{theorem}}
	
	\noindent \textbf{Part I}: (\ref{I-1.46}). By (\ref{I-1.41}) with $\psi_0 = 0$
	\begin{equation}\label{I-B.1}
	\psi(t) = \int_{0}^{t} e^{\calA_q(t-\tau)} \Fs(\tau) d \tau.
	\end{equation}
	where by the statement preceding Theorem \ref{I-Thm-1.4}
	\begin{equation}\label{I-B.2}
	\norm{e^{\calA_q(t-\tau)}}_{\calL(\lso)} \leq M e^{b (t - \tau)}, \quad 0 \leq \tau \leq t
	\end{equation}
	for $M \geq 1, \ b$ possibly depending on $q$.\\
	
	\noindent \textit{Step 1}: We have the following estimate
	\begin{equation}\label{I-B.3}
	\int_{0}^{T} \norm{\psi(t)}_{\lso}^p dt \leq C_T \int_{0}^{T} \norm{\Fs(t)}_{\lso}^p dt
	\end{equation}
	\noindent where the constant $C_T$ may depend also on $p,q,b$. This follows at once from the Young's inequality for convolutions \cite[p26]{CS:1979}:
	\begin{equation*}
	\norm{\psi(t)}_{\lso} \leq M \int_{0}^{t} e^{b (t - \tau)} \norm{\Fs(\tau)}_{\lso} d \tau \in L^p(0,T), \ T < \infty, 
	\end{equation*}
	\noindent and the convolution of the $L^p(0,T)$-function $\ds \norm{\Fs}_{\lso}$ and the $\ds L^1(0,T)$-function $e^{bt}$ is in $L^p(0,T)$. More elementary, one can use H\"{o}lder inequality with $\ds \rfrac{1}{p} + \rfrac{1}{\tilde{p}} = 1$ and obtain an explicit constant.\\
	
	\noindent \textit{Step 2}: \underline{Claim}: Here we shall next complement (\ref{I-B.3}) with the estimate 
	\begin{equation}\label{I-B.4}
	\int_{0}^{T} \norm{A_q \psi(t)}_{\lso}^p dt \leq C \int_{0}^{T} \norm{\psi(t)}_{\lso}^p dt + C \int_{0}^{T} \norm{\Fs(t)}_{\lso}^p dt
	\end{equation}
	
	\noindent to be shown below. Using (\ref{I-B.3}) in (\ref{I-B.4}) then yields
	\begin{equation}\label{I-B.5}
	\int_{0}^{T} \norm{A_q \psi(t)}_{\lso}^p dt \leq C_T \int_{0}^{T} \norm{\Fs(t)}_{\lso}^p dt.
	\end{equation}
	
	\noindent With respect to (\ref{I-1.41}) with $\psi_0 = 0$, then (\ref{I-B.5}) says  
	\begin{equation}\label{I-B.6}
	\Fs \in L^p(0,T; \lso) \longrightarrow \psi \in L^p(0,T;\calD(A_q) = \calD(\calA_q))
	\end{equation}\label{I-B.7}
	\noindent while (\ref{I-1.40}) then yields via (\ref{I-B.6})
	\begin{equation}
	\Fs \in L^p(0,T; \lso) \longrightarrow \psi_t \in L^p(0,T;\lso)
	\end{equation}
	\noindent continuously. Then, (\ref{I-B.6}), (\ref{I-B.7}) show is part (i) of Theorem \ref{I-1.6}.\\
	
	
	\noindent \textit{Proof of (\ref{I-B.4}):} . In this step, with $\psi_0 = 0$, we shall employ the alternative formula, via (\ref{I-1.42}) ($\nu = 1$, wlog)
	
	\begin{equation}\label{I-B.8}
	\psi(t) = \int_{0}^{t} e^{-A_q(t-\tau)} (-A_{o,q}) \psi (\tau) d \tau + \int_{0}^{t} e^{-A_q(t-\tau)} \Fs(\tau) d \tau,
	\end{equation}   
	\noindent where by maximal regularity of the Stokes operator $-A_q$ on the space $\lso$, as asserted in Theorem \ref{I-1.5}.ii, Eq (\ref{I-1.35}), we have in particular 
	\begin{equation}\label{I-B.9}
	\Fs \in L^p(0,T; \lso) \longrightarrow \int_{0}^{t} e^{-A_q(t - \tau)} \Fs(\tau) d \tau \in L^p(0,T;\calD(A_q)) \quad \text{continuously.}
	\end{equation}
	
	\noindent Regarding the first integral term in (\ref{I-B.8}) we shall employ the (complex) interpolation formula (\ref{I-1.22}), and recall from (\ref{I-1.9}) that $\calD(A_{o,q}) = \calD(A^{\rfrac{1}{2}}_q)$:
	
	\begin{equation}\label{I-B.10}
	\calD(A_{o,q}) = \calD(A^{\rfrac{1}{2}}_q) = \left[ \calD(A_q), \lso \right]_{\rfrac{1}{2}}
	\end{equation}
	\noindent so that the interpolation inequality \cite[Theorem p 53, Eq(3)]{HT:1980} with $\theta = \rfrac{1}{2}$ yields from (\ref{I-B.10})
	
	\begin{equation}\label{I-B.11}
	\begin{aligned}
	\norm{a}_{\calD(A_{o,q})} = \norm{a}_{\calD \big(A_q^{\rfrac{1}{2}} \big)} & \leq C \norm{a}_{\calD(A_q)}^{\rfrac{1}{2}} \norm{a}_{\lso}^{\rfrac{1}{2}}\\
	&\leq \varepsilon \norm{a}_{\calD(A_q)} + C_{\varepsilon} \norm{a}_{\lso}.
	\end{aligned}
	\end{equation}
	
	\noindent $\Big[$Since $\calD(A_q^{\rfrac{1}{2}}) = W^{1,q}_0(\Omega) \cap \lso$ by (\ref{I-1.22}), then for $a \in \calD(A_q) = W^{2,q}(\Omega) \cap W^{1,q}_0(\Omega) \cap \lso$, see (\ref{I-1.17}), we may as well invoke the interpolation inequality for $W$-spaces. \cite[Theorem 4.13, p 74]{A:1975}:	
	\begin{equation}
	\norm{a}_{\woiq(\Omega)} \leq \varepsilon \norm{a}_{\wtq(\Omega)} + C_{\varepsilon} \norm{a}_{\lso} \nonumber
	\end{equation}
	
	\noindent We return to (\ref{I-B.8}) and obtain 	
	\begin{equation}\label{I-B.12}
	A_q \psi(t) = A_q \int_{0}^{t} e^{-A_q(t-\tau)} (-A_{o,q}) \psi (\tau) d \tau + A_q \int_{0}^{t} e^{-A_q(t-\tau)} \Fs(\tau) d \tau.
	\end{equation} 
	
	\noindent Hence via the maximal regularity of the uniformly stable Stokes semigroup $e^{-A_q t}$, Eq (\ref{I-1.35}), (\ref{I-B.11}) yields	
	\begin{align}
	\norm{A_q \psi}_{L^p(0,T;\lso)} &\leq C \Big\{ \norm{A_{o,q} \psi}_{L^p(0,T;\lso)} + \norm{\Fs}_{L^p(0,T;\lso)} \Big\}\label{I-B.13}\\
	\text{by (\ref{I-B.11})} \hspace{1cm} \quad &\leq \varepsilon' \norm{A_q \psi}_{L^p(0,T;\lso)} + C_{\varepsilon'} \norm{\psi}_{L^p(0,T;\lso)} + C \norm{\Fs}_{L^p(0,T;\lso)} \label{I-B.14}
	\end{align} 
	
	\noindent $\varepsilon' = \varepsilon C > 0$ arbitrarily small. Hence (\ref{I-B.14}) yields	
	\begin{equation}\label{I-B.15}
	\norm{A_q \psi}_{L^p(0,T;\lso)} \leq \frac{C_{\varepsilon'}}{1-\varepsilon'} \norm{\psi}_{L^p(0,T;\lso)} + \frac{C}{1-\varepsilon'} \norm{\Fs}_{L^p(0,T;\lso)}
	\end{equation} 
	
	\noindent and estimate (\ref{I-B.4}) of Step 2 is established. Part I of Theorem \ref{I-Thm-1.6} is proved. \\
	
	\noindent \textbf{Part II}: (\ref{I-1.49}). For simplicity of notation, we shall write the proof on $\ds \Bto$ i.e. for $\ds 1 < q, p < \rfrac{2q}{2q-1}$. The proof on $\ds \lqaq$ in the other case $ \rfrac{2q}{2q-1} < p$ is exactly the same.
	
	\noindent \textit{Step 1}: Let $\eta_0 \in \Bto$ and consider the s.c. analytic Oseen semigroup $e^{\calA_q t}$ on the space $\Bto$, as asserted by Theorem \hyperref[I-Thm-1.4]{1.4.ii} (take $\nu = 1$ wlog):
	
	\begin{equation}\label{I-B.16}
	\eta(t) = e^{\calA_q t} \eta_0, \quad \text{or } \eta_t = \calA_q \eta = -A_q \eta - A_{o,q} \eta.
	\end{equation}
	\noindent Then we can rewrite $\eta$ as
	\begin{align}
	\eta(t) &=  e^{-A_q t} \eta_0 + \int_{0}^{t} e^{-A_q (t - \tau)}(-A_{o,q})\eta(\tau) \ d \tau \label{I-B.17}\\
	A_q \eta(t) &=  A_q e^{-A_q t} \eta_0 + A_q \int_{0}^{t} e^{-A_q (t - \tau)}(-A_{o,q})\eta(\tau) \ d \tau. \label{I-B.18}
	\end{align}
	\noindent We estimate, recalling the maximal regularity (\ref{I-1.35}), (\ref{I-1.36}) as well as the uniform decay (\ref{I-1.25}) of the Stokes operator.
	\begin{align}
	\norm{A_q \eta}_{L^p(0,T;L^q(\Omega))} &\leq  C \norm{\eta_0}_{\Bto} + C \norm{A_{o,q} \eta}_{L^p(0,T;\lso)} \label{I-B.19}\\
	&\leq C \norm{\eta_0}_{\Bto} + \varepsilon \widetilde{C} \norm{A_q \eta}_{L^p(0,T;\lso)} + C_{\varepsilon}\norm{\eta}_{L^p(0,T;\lso)}\label{I-B.20} 
	\end{align}
	\noindent after invoking, in the last step, the interpolation inequality (\ref{I-B.11}). Thus (\ref{I-B.20}) yields via (\ref{I-1.18}) 
	\begin{align}
	\norm{A_q \eta}_{L^p(0,T;\lso)} &= \norm{\calA_q \eta}_{L^p(0,T;\lso)} \nonumber\\
	&\leq \frac{C}{1 - \varepsilon \widetilde{C}} \norm{\eta_0}_{\Bto} + \frac{C_{\varepsilon}}{1 - \varepsilon \widetilde{C}} \norm{\eta}_{L^p(0,T;\lso)}. \label{I-B.21}
	\end{align}
	\noindent \textit{Step 2}: With $\ds \eta_0 \in \Bto$, since $\ds e^{\calA_q t}$ generates a s.c (analytic) semigroup on $\ds \Bto$, Theorem \ref{I-Thm-1.4}.ii, we have 
	\begin{equation}\label{I-B.22}
	\eta(t) = e^{\calA_qt} \eta_0 \in C \Big(0,T; \Bto \Big) \subset L^p \Big(0,T; \Bto \Big) \subset L^p \big(0,T; \lso \big)
	\end{equation} 
	\noindent continuously, where in the last step, we have recalled that $\ds \Bto$ is the interpolation between $\ds L^q(\Omega)$ and $\ds W^{2,q}(\Omega)$, see (\ref{I-1.16b}). (\ref{I-B.22}) says explicitly 
	\begin{equation}\label{I-B.23}
	\norm{\eta}_{L^p \big(0,T; \lso) \big)} \leq C \norm{\eta_0}_{\Bto}
	\end{equation}
	\noindent \textit{Step 3}: Substituting (\ref{I-B.23}) in (\ref{I-B.21}) yields 
	\begin{equation}\label{I-B.24}
	\norm{A_q \eta}_{L^p \big(0,T; \lso \big)} \leq C \norm{\eta_0}_{\Bto}
	\end{equation} 
	\noindent and (\ref{I-1.49}) is established, from which (\ref{I-1.50}) follows at once. Thus Theorem \ref{I-Thm-1.6} is proved. \qedsymbol
	\end{appendices}
	

	\medskip
	Received xxxx 20xx; revised xxxx 20xx.
	\medskip
\end{document}